\documentclass{article}
\usepackage{amsmath,amsthm,amssymb,amscd}
\usepackage[all,cmtip,color]{xy}
\usepackage[german,english]{babel}
\usepackage{amsmath}
\usepackage{amssymb}
\usepackage{hyperref}
\usepackage{mathrsfs} 
\usepackage{amssymb}
\usepackage{todonotes}
\usepackage{amssymb}
\usepackage[all,cmtip]{xy}
\usepackage{hyperref}
\usepackage[margin = 2.5 cm]{geometry}
\usepackage{colonequals,enumerate}
\usepackage{comment}
\usepackage{bbm}

\usepackage{authblk}
\usepackage{ dsfont }

\usepackage[export]{adjustbox}
\usepackage{graphicx}

\usepackage{enumitem}

\setlist[enumerate,1]{label={(\roman*)}}
\setlist[enumerate,2]{label={(\alph*)}}

\DeclareMathOperator{\iso}{\mathsf{iso}}
\DeclareMathOperator{\sep}{\mathsf{sep}}

\DeclareMathOperator{\Rb}{\mathbb{R}}
\DeclareMathOperator{\Cb}{\mathbb{C}}

\DeclareMathOperator{\Xc}{\mathcal{X}}

\DeclareMathOperator{\Oc}{\mathcal{O}}

\newcommand{\Cc}{\mathcal{C}}

\newcommand{\MF}{\mathfrak{F}}

\DeclareMathOperator{\Hom}{\mathsf{Hom}}

\DeclareMathOperator{\CO}{\mathcal{O}}
\DeclareMathOperator{\stab}{\mathrm{Stab}}

\DeclareMathOperator{\Nb}{\mathbb{N}}

\renewcommand{\lim}{\mathsf{lim}}

\newcommand\varto[1]{\mathrel{\hbox to #1pt{\rightarrowfill}}}

\def\isoto{\stackrel{\sim}{\longrightarrow}}

\DeclareMathOperator{\id}{id}

\DeclareMathOperator{\Mod}{\mathsf{Mod}}

\DeclareMathOperator{\Loc}{\mathsf{Loc}}

\DeclareMathOperator{\Lie}{Lie}

\DeclareMathOperator{\Ab}{\mathbb{A}}

\DeclareMathOperator{\Zb}{\mathbb{Z}}

\DeclareMathOperator{\F}{\mathcal{F}}

\DeclareMathOperator{\image}{im}

\DeclareMathOperator{\Gal}{Gal}
\DeclareMathOperator{\Aut}{\mathsf{Aut}}

\DeclareMathOperator{\uAut}{\underline{\mathsf{Aut}}}

\DeclareMathOperator{\GL}{GL}

\DeclareMathOperator{\M}{{\mathsf{M}}}

\DeclareMathOperator{\Bc}{\mathcal{B}}

\newcommand{\Mc}{\mathcal{M}}

\DeclareMathOperator{\Tr}{\mathsf{Tr}}
\DeclareMathOperator{\Fr}{Fr}

\DeclareMathOperator{\Pc}{\mathcal{P}}

\DeclareMathOperator{\Spec}{\mathsf{Spec}}

\DeclareMathOperator{\Sym}{Sym}
\DeclareMathOperator{\End}{\mathsf{End}}

\DeclareMathOperator{\Oo}{\mathcal{O}}

\DeclareMathOperator{\ord}{ord}

\DeclareMathOperator{\Yc}{\mathcal{Y}}

\DeclareMathOperator{\CIC}{\mathcal{IC}}
\DeclareMathOperator{\Qb}{\mathbb{Q}}

\newcommand{\BA}{{\mathbb{A}}}

\newcommand{\BC}{{\mathbb{C}}}

\newcommand{\BF}{{\mathbb{F}}}
\newcommand{\BG}{{\mathbb{G}}}
\newcommand{\BH}{{\mathbb{H}}}

\newcommand{\BL}{{\mathbb{L}}}
\newcommand{\BM}{{\mathbb{M}}}
\newcommand{\BN}{{\mathbb{N}}}

\newcommand{\BP}{{\Pc}}
\newcommand{\BQ}{{\mathbb{Q}}}
\newcommand{\BR}{{\mathbb{R}}}

\newcommand{\BZ}{{\mathbb{Z}}}

\newcommand{\Bw}{\mathbbm{w}}

\newcommand{\CF}{{\mathcal F}}
\newcommand{\CG}{{\mathcal G}}

\newcommand{\CL}{{\mathcal L}}
\newcommand{\CM}{{\mathcal M}}

\newcommand{\CT}{{\mathcal T}}

\newcommand{\CV}{{\mathcal V}}

\newcommand{\CX}{{\mathcal X}}

\DeclareMathOperator{\Cent}{Cent}

\DeclareMathOperator{\mcf}{CF}
\DeclareMathOperator{\Var}{Var}
\DeclareMathOperator{\im}{Im}

\newcommand{\BPS}{\mathcal{BPS} }

\DeclareMathOperator{\Log}{Log}

\DeclareMathOperator{\Ext}{Ext}

\newcommand{\Dr}[1][r]{\mathbb{D}_F^{1/#1}}

\DeclareMathOperator{\Dinfty}{\mathbb{D}_F^{1/\infty}}

\newcommand{\defeq}{\colonequals}
\let\into\hookrightarrow

\theoremstyle{definition}
\newtheorem{definition}{Definition}[section]
\newtheorem{construction}[definition]{Construction}
\newtheorem{rmk}[definition]{Remark}

\theoremstyle{plain}
\newtheorem{theorem}[definition]{Theorem}
\newtheorem{proposition}[definition]{Proposition}
\newtheorem{corollary}[definition]{Corollary}

\newtheorem{lemma}[definition]{Lemma}

\newtheorem{example}[definition]{Example} 
\newtheorem{claim}[definition]{Claim}

\newtheorem{situation}[definition]{Situation}

\begin{document}
\title{Twisted points of quotient stacks, integration and BPS-invariants\let\thefootnote\relax\footnotetext{M.G. was supported by an NSERC discovery grant and an Alfred P. Sloan
fellowship. D.W. was supported by the Swiss National Science Foundation [No. 196960]. P.Z. was supported by the ERC in the form of Eva Viehmann’s Consolidator Grant 770936: Newton Strata,
by the Deutsche Forschungsgemeinschaft (DFG) through the Collaborative Research Centre TRR
326 ”Geometry and Arithmetic of Uniformized Structures”, project number 444845124 and by the LOEWE professorship in Algebra, project number LOEWE/4b//519/05/01.002(0004)/87.
\\
} }

\author[1]{Michael Groechenig\thanks{\url{michael.groechenig@utoronto.ca}}}
\author[2]{Dimitri Wyss\thanks{\url{dimitri.wyss@epfl.ch}}}
\author[3]{Paul Ziegler\thanks{\url{ziegler@mathematik.tu-darmstadt.de}}}
\affil[1]{Department of Mathematics, University of Toronto}
\affil[2]{\'Ecole Polytechnique F\'ed\'erale de Lausanne}
\affil[3]{TU Darmstadt}

\renewcommand\Authands{ and }
\maketitle 
\abstract{We study $p$-adic manifolds associated with twisted points of quotient stacks $\Xc = [U/G]$ and their quotient spaces $\pi:\Xc \to X$. We prove several structural results about the fibres of $\pi$ and derive in particular a formula expressing $p$-adic integrals on $X$ in terms of the cyclotomic inertia stack of $\Xc$, generalizing the orbifold formula for Deligne-Mumford stacks. 

We then apply our formalism to moduli problems associated to hereditary abelian categories with symmetric Euler pairing, and show that their refined BPS-invariants are computed locally on the coarse moduli space by a $p$-adic integral. As a consequence we recover the $\chi$-independence of these invariants for $1$-dimensional sheaves on del Pezzo surfaces previously proven by Maulik--Shen. Along the way we derive a new formula for the plethystic logarithm on the $\lambda$-ring of functions on $k$-linear stacks, which might be of independent interest.}
\tableofcontents

\section{Introduction}

This paper is the third in a series of articles by the authors devoted to $p$-adic integration for moduli spaces. Let $F/\Qb_p$ be a $p$-adic local field with ring of integers denoted by $\Oc_F$. Given a smooth variety $Y/\Oc_F$ endowed with the action of a finite \'etale group scheme $G \circlearrowright Y$ of order prime to $p$ such that the quotient $X=Y/G$ can be defined, the theory of $p$-adic integration constructs a canonical measure $\mu_{orb}$ on $X(\Oc_F)$ such that the following formula holds:
\begin{equation}\label{eqn:quot}
\mu_{orb}\big(X(\Oc_F)\big) = \sum_{\phi} \sum_{y\in Y^{\phi}} \frac{q^{-w(y,\phi)}}{|C(\phi)|},
\end{equation}
where $\phi$ ranges over a set of representatives of conjugacy classes of Galois-equivariant homomorphisms $\widehat{\mu} \to G$ and $w(y,\phi)$ is a positive rational number computed in terms of the inertial action on the tangent space of $y$. This formula appeared first in the context of motivic integration in work of Denef--Loeser \cite{DL2002} and Yasuda \cite{Ya06}. A proof of the $p$-adic version was given in \cite{GWZ20a}.

This formula can serve as a powerful tool combining arithmetic and geometric methods. In \cite{GWZ20b,GWZ20a}, $p$-adic integration was applied to moduli spaces of Higgs bundles, to prove the Hausel--Thaddeus Conjecture \cite{HT03} and to reprove the Langlands--Shelstad Fundamental Lemma (Ng\^o's theorem \cite{MR2856372}) by a similar strategy, which allowed us to circumvent the use of perverse sheaves. Formula \eqref{eqn:quot} plays a central role in these proofs. For the moduli space of Higgs bundles, the fix loci $Y^{\phi}$ arising in Formula \ref{eqn:quot} admit a natural description as moduli spaces of co-endoscopic groups, an observation which lies at the heart of the $p$-adic integration strategy to prove the Fundamental Lemma.

We mention that another example where this formula allows one to relate arithmetic and geometric phenomena is provided by Yasuda--Wood's article \cite{wood2015mass}. They apply a generalisation of this formula (to wild actions) to reprove the \emph{mass formulae} of Serre and Bhargava. Their strategy exploits the birational equivalence of the quotient stack $[\Ab^{2n}_{\Oc_F}/S_n]$ with the Hilbert scheme of $n$-points $(\Ab^2_{\Oc_F})^{[n]}$.

The precursors \cite{GWZ20b,GWZ20a} were restricted in scope to finite group quotients, or equivalently, to smooth Deligne-Mumford stacks and their coarse moduli spaces. The purpose of the present paper is to explore $p$-adic integration in a more general setting comprising (singular) GIT quotients of smooth varieties by reductive groups. In fact, we consider the more general geometric framework (see Situation \ref{sit:good}) of a global quotient stack $\Xc = [U/G]$ with a morphism 
$\pi: \Xc \to X$,
to an algebraic space $X$ relative to $\Spec(\Oc_F)$. We require this morphism to be birational and to satisfy several assumptions. Our first main result states that there is a natural measure $\mu_{orb}$ on $X(\Oc_F)$, whose volume can be computed in terms of the special fibre $\Xc_k$ of the stack. 

\begin{theorem}[\ref{inbase}]\label{inthm} The generating series $\sum_{r\geq 1} \sum_{y \in I_{\mu_r}\Xc(k)} \frac{q^{-w(y)}}{|\Aut(y)(k)|}T^r$ is a rational function $Q_X(T)$ of degree $0$ and we have
\begin{equation}\label{introeq}
\int_{X(\Oc_F)} \mu_{orb} = -\lim_{T\to \infty} Q_X(T).
\end{equation}
\end{theorem}

In the main body of this article, this result will be established more generally, including the integration of a natural class of functions. We refer the reader to Theorem \ref{inbase} for more details. In Example \ref{DMcase}, it is shown that for $\Xc$ a smooth DM-stack, Theorem \ref{inthm} recovers Formula \eqref{eqn:quot} for tame finite group quotients.

\subsubsection*{Applications: gerbes, intersection cohomology \& Donaldson-Thomas theory}

To demonstrate the utility of this formula, we consider moduli problems $\Mc_{\Cc}$ of objects in abelian categories $\Cc$, which we assume to be of homological dimension one and to be \emph{symmetric}; a class of importance in Donaldson-Thomas theory \cite{MR19,Me15,BDNKP25}. Examples include representations of symmetric quivers, vector bundles on curves and (meromorphic) Higgs bundles on a curve. There are further assumptions needed for which we refer the reader to Section \ref{msds}.

We denote by $\alpha$ the gerbe measuring the obstruction to the existence of a universal family on $\Mc_{\Cc}$. Using the identification $\mathsf{Br}(F) = \Qb/\Zb$ given by the Hasse invariant and the Euler pairing $(\cdot,\cdot)$ on $\CM_{\Cc}$, one associates with $\alpha$ an almost everywhere defined function $f_{\tilde{\alpha}}$ taking values in complex roots of unity. The integral of this function can be explicitly evaluated in terms of intersection cohomology $IH^*$, leading to an important analogue of \eqref{eqn:quot}:

\begin{theorem}\eqref{icint}\label{thm:main2}
For any connected component $\CM_{\gamma}\subset \CM_{\Cc}$ such that the adequate moduli space $\CM_{\gamma} \to \M_\gamma$ is birational, we have
\[
- q_F^{-(\gamma,\gamma)+1} \int_{\M_{\gamma}(\Oc_F)} f_\alpha = \sum_{i \in \Zb} (-1)^i \Tr\big( \Fr, IH^i(\M_{\gamma},\Qb_{\ell})\big )
\]
\end{theorem}

The proof of this formula relies on Donaldson-Thomas theory, which identifies intersection cohomology with BPS-cohomology in this case \cite{Me15,DM20}. The latter satisfies by definition a recursive relation involving plethystic operations. Using Theorem \ref{inthm} we verify the same relation for the left-hand side in a $\lambda$-ring of \emph{counting functions} that is expressly introduced for this purpose in Section \ref{pikls}.

Combining Theorem \ref{thm:main2} with \cite{COW21}, which follows the strategy of \cite{GWZ20a} to prove the Hausel--Thaddeus Conjecture, one obtains a new proof of a $\chi$-independence result of Maulik--Shen for the moduli space $\M_{\beta,\chi}$ of $1$-dimensional sheaves on a del Pezzo surface with curve class $\beta$ and Euler characteristic $\chi$. 

\begin{theorem}\cite[Theorem 0.1]{MS20} \label{thm:intro-chi}For $k= \BC$ and $\chi,\chi' \in \BZ$ there exists an isomorphism of graded vector spaces between the intersection cohomology groups
\[ IH^*(\M_{\beta,\chi}) \cong IH^*(\M_{\beta,\chi'}),  \]
respecting the perverse filtrations. There also exists a possibly different isomorphism $IH^*(\M_{\beta,\chi}) \cong IH^*(\M_{\beta,\chi'})$ respecting the Hodge filtrations.  
\end{theorem}

If one applies Theorem \ref{thm:main2} to the moduli space of semi-stable vector bundles on a curve one recovers an observation made by Kiem-Li \cite{KL04} regarding the relation between intersection cohomology and stringy Hodge numbers, see \ref{KL}.

\subsubsection*{Methods: twisted $p$-adic points \& Bruhat-Tits buildings}

Let us recall the geometric setting \ref{sit:good} of a global quotient stack $\Xc = [U/G]$ with a morphism 
$\pi: \Xc \to X$,
to an algebraic space $X$ relative to $\Spec(\Oc_F)$. We further require the existence of an open dense $W \subset X$ such that the base change $\Xc_W = \Xc \times_X W \to W$ is an equivalence, and thus $\pi$ is reminiscent to the map to a coarse moduli space of semistable objects. In this analogy, the subset $W$ corresponds to a subset of the stable locus.

The proof of Formula \eqref{introeq} exploits structural properties of the map $\Xc(\Oc_F) \to X(\Oc_F)$. As will be stated in Theorem \ref{bapro} below, this map is almost everywhere a locally measure-preserving topological covering. In order to solve the over-counting problem caused by the finite fibres of this covering, it is necessary to embed $\Xc(\Oc_F)$ inside an even bigger covering of $r$-twisted points, $\Xc(\Dr) \to X(\Oc_F)$. Here, $\Dr$ is defined to be as the root stack of $\mathbb{D}_F=\Spec \Oc_F$, obtained by formally adjoining an $r$-th root to the divisor given by the ideal $\mathfrak{m}_F \subset \Oc_F$. This is a tame stack, containing $\Spec F$ as an open dense subscheme and $B\mu_r = [\Spec k / \mu_r]$ as a closed subscheme.

Under the assumptions above, there is a subset $\Xc(\Dr)^\sharp \subset \Xc(\Dr)$, $r \geq 1$, with the complement being of measure zero, admitting the structure of a $p$-adic manifold. We show the following fundamental properties:

\begin{theorem}[\ref{prop:r-finite}, \ref{conhul}]  \label{bapro}
\begin{enumerate}
\item[(a)]\label{parta} The map
\[\pi_r:\Xc(\Dr)^\sharp \to  X(\Oc_F)^\sharp, \]
is étale with finite fibres.
\item[(b)]\label{partb} Assume $\Xc / \Oc_F$ is $S$-complete and fix a presentation $\Xc = [U/G]$ with $G$ a general linear group. Then, for any $x \in  X(\Oc_F)^\sharp$ we can identify $ \cup_{r\geq 1} \pi_r^{-1}(x)$ with the rational points of a bounded convex region inside the Bruhat-Tits building of $G$. 
\end{enumerate}
\end{theorem}
\begin{figure} 
\centering
\includegraphics[height=3cm]{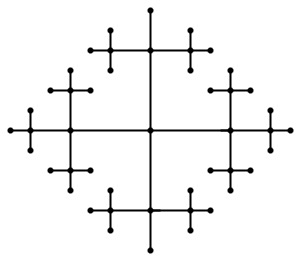}
\caption{The product of this graph with $\Rb$ is the extended Bruhat-Tits building of $\GL_2(\Qb_3)$.}
\end{figure}

The second assertion is crucial for solving the aforementioned over-counting problem. As shown by part (b) of the result above, the cardinality of a fibre $N_r = |\pi_r^{-1}(x)|$ agrees with the number of $\Zb[\frac{1}{r}]$-points in a convex polytope $P$ in the Bruhat-Tits building of $G/F$. The generating series 
\[
E_P(T)=\sum_{r \geq 1} N_r T^r
\]
is known as the Ehrhart series and can be shown to be the Taylor series expansion of a rational function $E_P(T) \in \Qb(T)$ with a single pole in $T=1$. Furthermore, it is a well-known principle in combinatorics that the limit $\lim_{T\to \infty} E_P(T)$ agrees with $-1$. Heuristically exchanging limits, analytic continuation and integration, one sees that the volume $X(\Oc_F)$ can be computed in the following way:
\[
\mu_{orb}\big(X(\Oc_F)\big) = - \lim_{T\to\infty} Q_X(T),\]
where $Q_X(T)$ is the analytic continuation of $ \sum_{r \geq 1} \mu_{orb}(\Xc(\Dr)) T^r$.

The proof of Theorem \ref{inthm} combines this strategy with a volume formula for Artin stacks \ref{orbiorbi}, to obtain a formula for the volume of $\Xc(\Dr)^\sharp$ for a fixed $r\geq 1$.   

\subsubsection*{Relation to previous work}

To the best of our knowledge, Theorem \ref{inthm} is the first result of its kind applying to moduli spaces of Artin stacks. We are not aware of any similar formulae in the p-adic or motivic integration literature. The overall shape of the expression, that is, evaluating the analytic continuation of a generating series at infinity to obtain a quantity of interest has appeared before, notably in connection with motivic vanishing cycles (see \cite{DL98,HL15}). It would be interesting to investigate a possible connection between these two, a priori unrelated, formulae.

The $r=1$ case of Theorem \ref{prop:r-finite}, that is the map $\pi_1\colon \Xc(\mathbb{D}_F) \to X(\Oc_F)$, was first studied in the motivic setting in \cite{SU22}, where it was interpreted as a measure of the non-separatedness of a stack. It should be pointed out that in this context, unlike (1) of Theorem \ref{prop:r-finite} for $p$-adic integration, $\pi_1$ can have infinite fibres, see Section 10 in \textit{loc. cit}. 

In Donaldson-Thomas theory, BPS-invariants and their refinements appear as integral invariants in the presence of strictly semi-stable objects in a given CY3 category \cite{JS12, KS08}. In \cite{COW21} it was observed that in the case of $1$-dimensional sheaves on local del Pezzo surfaces and meromorphic Higgs bundles on a curve, the alternating sum of Frobenius-traces on BPS-cohomology can be computed as a $p$-adic integral on the corresponding coarse moduli spaces. The proof in \textit{loc. cit.} is rather indirect, based on a $\chi$-independence result of Maulik-Shen \cite{MS20}, and the present article grew out of the desire to provide a more conceptual understanding from the point of view of $p$-adic integration. Our methods allowed us to give a new proof of $\chi$-independence (see Theorem \ref{thm:intro-chi} above), albeit $p$-adic integration only allows one to recover a slightly weaker statement.

\subsubsection*{Acknowledgements} We warmly thank Arthur Forey and François Loeser for sharing their intuition from motivic integration with us. We further thank Francesca Carocci, Lucien Hennecart, Young-Hoon Kiem, Matthew Satriano, Jeremy Usatine and Tanguy Vernet for interesting discussions on various aspects of the paper and Evan Sundbo for careful reading of an early draft and pointing out several inaccuracies.

\section{$p$-adic manifolds from linear quotient stacks}\label{mqs}

\subsection{Quotients of $F$-analytic manifolds}

In this subsection we refer to a group object in the category of $F$-analytic manifolds as a \emph{$p$-adic Lie group}, even for other non-archimedean local fields $F$. We refer to \cite{Se09} for a detailed account. 

Let $X$ be an $F$-analytic manifold endowed with an analytic action $a\colon G \times X \to X$ by a $p$-adic Lie group $G$. We assume that the action is proper, i.e., the map
$$(a,\id_X)\colon G \times X \to X \times X$$
is proper and also that it is free. We include the proof of the following proposition since we will refer to details of the construction at a later point.

\begin{proposition}\label{prop:quotient}\cite[Theorem II.III.12.1]{Se09}
 Under the assumptions above, the topological quotient $X/G$ can be endowed in a unique manner with the structure of an $F$-analytical manifold such that the quotient map $\pi\colon X \to X/G$ is an $F$-analytic submersion.  

In particular we have for $x \in X$ a short exact sequence
\[ 0 \to \mathfrak{g} \to T_xX \to T_{[x]}(X/G), \]
where $\mathfrak{g} \to T_xX$ denotes the injection given by the infinitesimal action. 
 
\end{proposition}
\begin{proof}
 It is clear that there can be at most one such structure of an $F$-analytic manifold, by virtue of the Implicit Function Theorem. 

 Thus, we turn to the proof of existence. Hausdorffness of the quotient follows from properness of the action. Second countability is inherited from $X$ to $G$ by virtue of openness of the projection map $\pi\colon X \to X/G$.
 
The last step will be to construct charts on $X$ which are adapted to the $G$-action.
 
 \begin{claim}[Adapted charts]\label{cl:adapted}
 For every point $x \in X$ there exists a chart $(W,\phi)$, where $W \subset X$ and $\phi\colon W \to \Oc_F^m \times \Oc_F^n$, such that the pre-image $\phi^{-1}(\Oc_F^m \times \{y\})$ agrees with the intersection of a $G$-orbit with $W$. 
\end{claim}

\begin{proof}
 Note that the orbit $Gx$ is a submanifold of $X$ since it equals the image of the proper injective immersion $G \to X$ sending $g$ to $gx$. 
 
 To construct a chart as above, we let $S$ be a slice to the orbit $Gx$, i.e., a submanifold which is transverse to the orbit $Gx \subset X$ and intersects each $G$-orbit at most once. The action map $G \times S \to X$ is by construction regular in $(1,x)$ and thus the Inverse Function Theorem yields a neighbourhood $U \times V$ of $(1,x)$ and $W$ of $x \in X$ such that the above map restricts to an isomorphism
 $$W \to U \times V.$$
 This concludes the proof of the claim.
 \end{proof}

The subset $\pi(W)\subset X/G$ is open since $\pi$ maps open subsets of $X$ to open subsets in the quotient. By design, the map $\mathrm{pr_2}\circ\phi\colon W \to \Oc_F^n$ descends to a continuous bijection $\phi'\colon \pi(W) \to \Oc_F^{n}$. Since the left-hand space is compact and the right-hand space is Hausdorff, this map is a homeomorphism. We have therefore succeeded in assigning a chart $(\pi(W),\phi')$ to an adapted chart $(W,\phi)$.

The collection of all charts of $X/G$ obtained in this manner from adapted charts of $X$ is an $F$-analytic atlas. To see this it suffices to prove that the change of coordinates map is $F$-analytic. Let $(W_1,\phi_1)$ and $(W_2,\phi_2)$ be two adapted charts of $X$. We then have the following commutative diagram

\[
\xymatrix{
\phi_2(W_1 \cap W_2) \ar[r]^{\phi_{12}} \ar[d]_{\mathrm{pr}_2} & \phi_1(W_1 \cap W_2) \ar[d]^{\mathrm{pr}_2} \\
\phi'_2(\pi(W_1\cap W_2)) \ar[r]^{\phi'_{12}} & \phi'_1(\pi(W_1\cap W_2)),
}
\]
where $\mathrm{pr}_2\colon \Oc_F^m \times \Oc_F^n \to \Oc_F^n$ denotes the projection to the second component.
Since the top row is an analytic map the same property holds for the bottom row.
\end{proof}

\subsection{$\Oc_F$-points of a linear stack}

In this subsection we study the structure of $\Xc(\Oc_F)$ where $\Xc$ is a $\Oc_F$-stack, which can be represented as a global quotient by a connected linear algebraic group.

\begin{definition}
A stack $\Xc$ over a base scheme $S$ is said to be a \emph{linear quotient stack} if there exists a smooth linear group scheme $G \to S$ and an algebraic $S$-space $U$ endowed with a relative group action $U \times_S G \to U$ such that there is an equivalence of stacks $[U/G] \simeq \Xc$.\end{definition}

By virtue of definition, there is a $G$-invariant map $U \to \Xc$ which is a $G$-torsor with representable total space. Existence of such a $G$-torsor leads to another equivalent way to phrase the definition of linear quotient stacks.

There is usually a plethora of ways to realize a given $\Xc$ as a linear quotient stack. For instance, given a presentation $\Xc \simeq [U/G]$ and an embedding of $G \hookrightarrow H$ we may write
$\Xc \simeq [U_H / H]$, where $U_H = (U \times H)/G$ denotes the quotient with respect to the anti-diagonal action, i.e., the induced $H$-torsor. For this reason, we may assume, without loss of generality that $G$ is a general linear group.

\begin{definition}\label{defi:hashtag}
Let $\Xc$ be a smooth linear $\Oc_F$-stack, together with a dense open smooth subspace $W \subset \Xc$. Then, we define $\Xc(\Oc_F)^{\sharp}$ to be the set given by the fibre product $\Xc(\Oc_F)^{\iso} \times_{\Xc(F)^{\iso}} W(F)$.
\end{definition}

\begin{proposition}\label{presin}
There exists a unique structure of an $F$-analytic manifold on $\Xc(\Oc_F)^{\sharp}$ such that for every presentation as a linear quotient stack $\Xc\simeq [U/G]$, the induced map 
$$U(\Oc_F) \to \Xc(\Oc_F)^{\iso}$$ 
is an $F$-analytic submersion over $\Xc(\Oc_F)^{\sharp}$.
\end{proposition}
\begin{proof}
For a given presentation as a quotient by the general linear group, as explained above, we can define an $F$-analytic manifold structure on $\Xc(\Oc_F)^{\sharp}$ by applying Proposition \ref{prop:quotient} to $U(\Oc_F)^{\sharp}$ acted on by $G(\Oc_F)$. Here, we denote by $U(\Oc_F)^{\sharp}$ the open submanifold given by the pre-image of $\Xc(\Oc_F)^{\sharp}$.

It remains to ensure that for an alternative choice of presentation $\Xc \simeq [V/H]$ we obtain an equivalent manifold structure. This follows from the following commutative diagram:
\[
\xymatrix{
(U \times_{\Xc} V)(\Oc_F)^{\sharp} \ar[r] \ar[d] & V(\Oc_F)^{\sharp} \ar[d] \\
U(\Oc_F)^{\sharp} \ar[r] & \Xc(\Oc_F)^{\sharp} 
}
\]
Since all morphisms in this diagram are quotients with respect to a free and proper $G(\Oc_F)$ or $H(\Oc_F)$-action, we can apply the uniqueness assertion of Proposition \ref{prop:quotient} to infer that the identity map $\id_{\Xc(\Oc_F)^{\sharp}}$ induces an $F$-analytic isomorphism between the quotients $U(\Oc_F)^{\sharp}/G(\Oc_F) \simeq V(\Oc_F)^{\sharp}/H(\Oc_F)$.
\end{proof}

\subsection{Proper and \'etale maps}

Recall the notion of a finite covering map $\pi\colon U \to V$ of topological spaces as map for which there exists an open covering $\{V_i\}_{i \in I}$ of $V$ such that each $\pi^{-1}(V_i) \xrightarrow{\pi} V_i$ is isomorphic to the projection map $V_i \times F_i \to V_i$ for some finite set $F_i$ endowed with the discrete topology. The purpose of this subsection is to prove the following auxiliary results:

\begin{proposition}
Let $\pi\colon U \to V$ be a proper and \'etale morphism of $F$-analytic manifolds. Such a $\pi$ is a covering map in the neighbourhood of any point $x \in V$ where the preimage $\pi^{-1}(x)$ is finite.
\end{proposition}
\begin{proof}
According to the Inverse Function Theorem, every $y \in \pi^{-1}(x)$ has an open neighbourhood $W_y$ such that $\pi(W_y) \subset V$ is open and the restriction 
$$\pi|_{W_y}\colon W_y \to \pi(W_y)$$
is an isomorphism of $F$-analytic manifolds.

Since there are only finitely many elements in $\pi^{-1}(x)$, the intersection
$N_x=\bigcap_{y \in \pi^{-1}(x)} \pi(W_y)$
is open. By construction, we have $\pi^{-1}(N_x) \supset \bigsqcup_{y \in \pi^{-1}(x)} W_y$, and thus we see that 
$$\pi^{-1}(N_x) \supset \bigsqcup_{y \in \pi^{-1}(x)} \underbrace{\big(\pi^{-1}(N_x) \cap W_y\big)}_{= L_y}.$$
We refer to $L_y$ as the \emph{leaf} of $y \in \pi^{-1}(x)$.

We intend to show that possibly after replacing $N_x$ by a smaller neighbourhood, the above inclusion is in fact an equality, i.e., that $\pi^{-1}(N_x)$ equals the disjoint union of leaves. Let $N_x^{(n)}$ be a descending sequence of compact-open neighbourhoods of $x$ contained in $N_x$ and forming a neighbourhood basis. We denote $\pi^{-1}(N_x^{(n)}) \cap L_y$ by $L_y^{(n)}$.

Assume by contradiction that for every $n \in \mathbb{N}$ there exists $z_n$ in $\pi^{-1}(N_x^{(n)})$, which does not lie in $\bigsqcup_{y \in \pi^{-1}(x)} L_y^{(n)}$. By properness of the map $f$, this sequences possesses a limit point $z$. By construction, this point $z$ lies in $\pi^{-1}(x)$ and must belong to a leaf $L_z$. Since $L_z$ is open in $U$, it follows that for some $n$ the point $z_n$ is contained in $L_z$ and hence in $L_z^{(n)}$. This contradiction concludes the proof. 
\end{proof}

\begin{corollary}\label{cor:loc-constant}
Assume that $\pi\colon U \to V$ is a proper and \'etale morphism of $F$-analytic manifolds with finite fibres. Then, the function 
$$\sep \colon V \to \mathbb{N}$$
assigning to $x \in V$ the cardinality $|\pi^{-1}(x)|$ is locally constant.
\end{corollary}
\begin{proof}
 By virtue of definition of a finite covering, every point $x \in V$ possesses an open neighbourhood $N_x$ such that there exists a finite discrete space $F_x$ and a homeomorphism $\pi^{-1}(N_x) = N_x \times F_x$ intertwining $\pi$ and the projection to the first component. For every $y \in N_s$ we therefore have $\sep(y) = |F_x|$, and thus the function $\sep$ is locally constant.  
\end{proof}

\subsection{$p$-adic properties of the map $\pi$}

In this subsection we will study the behaviour of $p$-adic points with respect to birational maps from algebraic stacks to algebraic spaces (or schemes).

\begin{situation}\label{sit:good}
Assume that $\Xc = [U/G]$ is a linear quotient stack over $\Oc_F$ where $U$ is separated. Let $\pi\colon \Xc \to X$ be a morphism to an algebraic $\Oc_F$-space such that there exists an smooth open-dense subspace $W \subset X$ with the base change $\Xc_W = \Xc \times_X W \to W$ being an equivalence. 

Among all such $W \subset X$, there exists a maximal one, and we will always take $W$ to be maximal.
\end{situation}

\begin{definition} Assume that we are in Situation \ref{sit:good}.  Similar to Definition \ref{defi:hashtag}, we let $X(\Oc_F)^{\sharp} = X(\Oc_F)\times_{X(F)} W(F)$.
\end{definition}

\begin{proposition}\label{prop:etale}
Assume that we are in Situation \ref{sit:good}. Then, the map 
$\pi\colon \Xc(\Oc_F)^{\sharp} \to X(\Oc_F)^{\sharp}$ is an \'etale $F$-analytic map. 
\end{proposition}
\begin{proof}
We use that $X(\Oc_F)^{\sharp} \subset W(F)$. Without loss of generality, $G$ is a general linear group, and we then have $\Xc(\Oc_F) = U(\Oc_F)/G(\Oc_F)$ and $W(F) = U_W(F)/G(F)$. For a point $[x] \in \Xc(\Oc_F)^{\sharp}$ we therefore have 
$T_{[x]}  \Xc(\Oc_F) \simeq T_x U(\Oc_F) / \mathfrak{g} \simeq T_{\pi([x])} X(\Oc_F)$, and likewise $T_{\pi([x])} \simeq T_x U(F) / \mathfrak{g}$.

Since $U(\Oc_F) \subset U(F)$ is open, we obtain that both tangent spaces are identical. This concludes the proof.
\end{proof}

\begin{proposition}\label{prop:proper}
Assume that we are in Situation \ref{sit:good}. Then, the map $\pi = \pi(\Oc_F)$ is a proper map between $p$-adic manifolds.    
\end{proposition}
\begin{proof}
 Let $\Xc \simeq [U/G]$ be a presentation of $\Xc$ as a linear quotient stack, without loss of generality we assume that $G$ is a general linear group. We then have a cartesian diagram of topological spaces
 \[
\xymatrix{
U(\Oc_F)^{\sharp} \ar[d] \ar[r]^g & \Xc(\Oc_F)^{\sharp} \ar[d] \ar[r]^{\pi} & X(F)^{\sharp} \ar[d]  \\
U(\Oc_F) \ar[r]^-{\alpha} & U(\Oc_F)/G(\Oc_F) \ar[r] & X(F),
}
 \]
i.e., $U(\Oc_F)^{\sharp} = U(\Oc_F) \times_{U(\Oc_F)/G(\Oc_F)} \Xc(\Oc_F)^{\sharp}$. The map $\alpha$ is a continuous map of compact Hausdorff spaces and thus is proper. 

Let $K \subset X(F)^{\sharp}$ be a compact subset. Then, it is also closed in $X(F)$, therefore the pre-image $g^{-1}(\pi^{-1}(K))$ is closed in $U(\Oc_F)^{\sharp}$.  We conclude that $\pi^{-1}(K)=g(g^{-1}(\pi^{-1}(K)))$ is proper too.
\end{proof}

\subsection{$p$-adic finiteness}

Continuing our analysis of the $p$-adic properties of $\pi$, we will prove in this subsection that the map $\pi$ even happens to be finite.

\begin{proposition}\label{prop:finite}
Assume that we are in Situation \ref{sit:good}. The map $\pi=\pi(\Oc_F)$ is a finite map between $p$-adic manifolds, i.e., every fibre $\pi^{-1}(x)$ is a finite set.
\end{proposition}

The proof of finiteness rests on the following (well-known) auxiliary result, for which a proof is include for the convenience of the reader.

\begin{lemma}\label{lemma:discrete}
Let $G/\Oc_F$ be a linear group scheme. The quotient topology on $G(F)/G(\Oc_F)$ is discrete.
\end{lemma}
\begin{proof}
 Recall that $\Oc_F \subset F$ is a compact-open subset. This implies that $G(\Oc_F)\subset G(F)$ is a compact-open subgroup. The quotient space $G(F)/G(\Oc_F)$ therefore has the property that its points (which correspond to $G(\Oc_F)$-orbits) are open. Hence, the quotient topology is discrete.    
\end{proof}

\begin{proof}[Proof Proposition \ref{prop:finite}]
Let $U_W$ be the fibre product $U \times_X W$ and $\widetilde{W(F)}$ the quotient $U_W(F)/G(\Oc_F)$. The map 
\begin{equation}\label{eqn:strassenzug}
\widetilde{X(F)}=U_W(F)/G(\Oc_F) \to U_W(F)/G(F)=W(F)\end{equation}
will be denoted by $c$.

There is a commutative diagram
\[
\xymatrix{
& \widetilde{W(F)} \ar[rd]^{c} & \\
\Xc(\Oc_F)^{\sharp} \ar[r]^{\pi} \ar@{-->}[ur]^b & X(\Oc_F)^{\sharp} \ar[r]^{\subset} & W(F),
}
\]
where the dashed arrow is given by the natural map 
$$\Xc(\Oc_F)^{\sharp}=U(\Oc_F)^{\sharp}/G(\Oc_F) \to U_W(F)/G(\Oc_F)=\widetilde{W(F)}.$$ 
Since $c$ has discrete fibres homeomorphic to $G(F)/G(\Oc_F)$ and $\pi$ is proper, we infer from injectivity of the dashed arrow that the fibres of $\pi$ are simultaneously compact and discrete spaces. Therefore, $\pi$ must be a finite map.
\end{proof}

\begin{corollary}
The map $\pi\colon \Xc(\Oc_F) \to X(\Oc_F)$ is a finite cover and the function $\sep\colon X(\Oc_F)^{\sharp} \to \mathbb{N},\; x \mapsto |\pi^{-1}(x)|$ is well-defined and locally constant.
\end{corollary}
\begin{proof}
It follows from Proposition \ref{prop:finite} that each fibre is finite. The counting function $\sep$ is therefore well-defined on all of $X(\Oc_F)^{\sharp}$. In Corollary \ref{cor:loc-constant} it is shown that this function is locally constant, provided that $\pi$ is \'etale and proper. Those two qualities were established in Propositions \ref{prop:etale} and \ref{prop:proper}.
\end{proof}

\section{Twisted $\Oc_F$-points and buildings}\label{tpb}

This section is devoted to incorporating twisted $\Oc_F$-points in the picture developed so far. Twisted $\Oc_F$-points have already played an important role in our previous work on $p$-adic integration for Hitchin systems (see \cite{GWZ20b,GWZ20a}). The key difference here is that we also require certain twists by $\mu_r$, where we no longer assume that the residual characteristic is prime to $r$.

\subsection{Definition and basic properties}
The most efficient option to introduce this concept of twisted $\Oc_F$-points is via the theory of root stacks introduced independently by Abramovich--Graber--Vistoli \cite{MR2450211} and Cadman \cite{1117817b-301b-381d-b9c7-cb40e4814623}.

\begin{definition}
For $r \in \mathbb{N} \setminus \{0\}$ we denote by $\Dr$ the root stack of order $r$ of $\Spec \Oc_F$ with respect to the Cartier divisor given by $\mathfrak{m}_F \subset \Oc_F$. To wit, we fix a generator $\varpi \in \mathfrak{m}$, which allows us to identify the root stack with the quotient stack of $\Spec \Oc_F[\lambda]/(\lambda^r-\varpi)$ with respect to the natural $\mu_r$-action. 
\end{definition}

The root stack $\Dr$ receives two immersions, the first being open and the second one being closed:

$$\Spec F \hookrightarrow \Dr \hookleftarrow  B\mu_{r,k},$$
where $B\mu_{r,k}=[\Spec k / \mu_{r,k}]$.

\begin{definition}
 Let $\Xc /\Oc_F$ be an algebraic stack.
 \begin{enumerate}
     \item For $r \geq 1$, we denote by $\Xc(\Dr)$ the set of isomorphism classes of maps $\Dr \to \Xc$. Restriction to the open substack $\Spec F$ induces a natural map $\pi_r\colon \Xc(\Dr) \to \Xc(F)$. 
     \item We also write $\Xc(\Dinfty)$ for the union $\cup_{r\geq 1} \Xc(\Dr)$.
 \end{enumerate}    
\end{definition}

The purpose of the next subsection will be to construct a natural structure of an $F$-analytic manifold on $\Xc(\Dr)$ and to show that the map $\pi_r$ is \'etale.

\subsection{$F$-analytic manifolds of twisted points}

\begin{definition}
 Let $L/F$ be a finite Galois extension with Galois group $\Gamma = \Gal(L/F)$, and $M$ an $L$-analytic manifold. A $\Gamma$-action on $M$ is said to be \emph{Galois} if every $x\in M$ is contained in a $\Gamma$-invariant chart $(U,\phi\colon U \hookrightarrow L^d)$ such that $\phi$ is a $\Gamma$-equivariant map.
\end{definition}

We first establish a few basic lemmas, relating $\Gamma$-fixpoints to $F$-analytic manifolds
\begin{lemma}\label{lemma:Galois-fix}
 Let $L/F$, $\Gamma$ and $M$ be as before. Then, the fixpoint set $M^{\Gamma}$ is naturally endowed with the structure of an $F$-analytic manifold.  
\end{lemma}
\begin{proof}
Since $M^{\Gamma} \subset M$ is a subset, it is Hausdorff and second countable. Let $x \in M^{\Gamma}$, and let $(U,\phi)$ a $\Gamma$-invariant chart containing $x$ as above. 
In particular, $\phi|_{U^{\Gamma}}$ yields an open embedding $U^{\Gamma} \hookrightarrow (L^d)^{\Gamma} \cong F^d$. We claim that the charts for $M^{\Gamma}$ obtained this way are compatible, i.e., change-of-coordinates maps are $F$-analytic. Let $\phi_{ij}\colon \Oc_L^d \to \Oc^d_L$ be a change-of-coordinates map (up to rescaling, it can be assumed to be of this form). In particular, $\phi_{ij}$ is a Galois-invariant map, locally given by a convergent power series with coefficients in $L$. After replacing $\Oc_L^d$ by a sufficiently small neighbourhood of $0$, and rescaling again, we may assume that $\phi_{ij}(z) = \sum_{n=0}^{\infty} a_n z^n$ with $a_n \in L$. Galois-invariance amounts to the condition $\phi_{ij}(\gamma \cdot{}z) = \gamma\cdot{} \phi_{ij}(z)$, which forces the coefficients $a_n$ to lie in $F$. In particular, we see that $\phi_{ij}|_{\Oc_F^d}$ is an $F$-analytic map, which concludes the proof.
 \end{proof}
 
\begin{lemma}\label{lemma:Gamma-etale}
Let $g\colon M \to N$ be a Galois-equivariant \'etale map between $L$-analytic manifolds. Then, the induced map $g|_{M^{\Gamma}} \colon M^{\Gamma} \to N^{\Gamma}$ is \'etale. 
\end{lemma}
\begin{proof}
For every $x \in M^{\Gamma}$ we have a natural isomorphism $(T_x M^{\Gamma}) \otimes_F L \simeq T_x M$, as follows from the chart-wise nature of the definition of a Galois action. Since the differential $dg \colon T_x M \to T_x N$ is an isomorphism, it follows from descent theory that its restriction $$dg|_{T_xM^{\Gamma}}\colon T_xM^{\Gamma}\to T_xN^{\Gamma}$$ must be an isomorphism too.
\end{proof}

\begin{lemma}\label{lemma:Galois}
Let $X/F$ be a smooth $F$-variety and $L/F$ be a finite Galois extension with Galois group $\Gamma$. Then, the natural $\Gamma$-action on the $L$-analytic manifold $X(L)$ is Galois.
\end{lemma}
\begin{proof}
For every $x \in X(F)$ we can find an \'etale chart, i.e., an open subscheme $V \subset X$ with an \'etale morphism of $F$-schemes $\phi\colon V \to \Ab_F^d$. The induced map $V(L) \to L^d$ is then a Galois-equivariant \'etale map. Within a sufficiently small neighbourhood of $x$ it will be a Galois-equivariant chart.
\end{proof}

\begin{lemma}\label{lemma:Galois-slice}
Let $i\colon M \hookrightarrow N$ be a $\Gamma$-equivariant immersion of $L$-analytic manifolds, i.e., $d_xi$ is assumed to be injective for all $x \in M$. Then, for every $p \in M$ there exists a $\Gamma$-equivariant submanifold $S \hookrightarrow N$, of complementary dimension to $M$, with $M \cap S =\{p\}$ and $T_p N = T_p M \oplus T_p S$.
\end{lemma}
\begin{proof}
Since this is a local statement, we may assume without loss of generality that $M=\Oc_L^m$ and $N=\Oc_L^n$. We are thus given a $\Gamma$-equivariant analytic map $i=(i_1,\dots,i_n)\colon \Oc_L^m \to \Oc_L^n$ with injective Jacobi matrix $d_xi$.  By further shrinking the neighbourhoods around $p$ and re-ordering coordinates if necessary, we achieve that the first $m\times m$ minor of $d_xI$ is invertible. Thus, the function $g\colon \Oc_L^m \to \Oc_L^m$ given by
$g(x,y) = (i_1(x),\dots,i_m(x))$ is \'etale and $\Gamma$-equivariant.

In particular, after further shrinking the neighbourhood of $p$, we may assume that $g^{-1}$ exists. We then have that $(i \circ g^{-1})(x) = (x,f(x))$, where $f$ is a $\Gamma$-equivariant function $f\colon \Oc_{L}^m \to \Oc_L^{n-m}$ by construction. The requisite transverse $S$ through $p$ can now be defined as the subset
$\{(p,z)|z\in \Oc_L^{n-m}\}.$
\end{proof}

\begin{proposition}\label{prop:Galois-quotient}
Let $M$ be an $L$-analytic manifold endowed with an analytic and free and proper $G$-action
$$m\colon G \times M \to M$$ where $G$ is an $L$-analytic Lie group. Assume furthermore that $G$ and $M$ are endowed with compatible Galois actions by $\Gamma = \Gal(L/F)$ such that the action map $m$ is $\Gamma$-equivariant. Then, the quotient manifold $M/G$ is endowed with a Galois action.
\end{proposition}
\begin{proof} 

We will freely use the notation and terminology from the proof of Proposition \ref{prop:quotient}. Recall that the key step for the construction of an $F$-analytic manifold structure for $M/G$ was the construction of \emph{adapted charts} for $M$ with respect to the group action $G$. This was achieved in Claim \ref{cl:adapted}. Note that we replace $F$ by $L$ for the purpose of the current argument. We must show that the adapted charts can be constructed in a Galois-invariant manner. For this purpose, we fix a point $x \in M^{\Gamma}$ and consider the orbit $G\cdot{} x$. We then choose a $\Gamma$-invariant transverse slice $S$ (see Lemma \ref{lemma:Galois-slice}) and conclude the proof in the same manner as the one of Proposition \ref{prop:quotient}.
\end{proof}

\begin{proposition}\label{prop:Gamma-finite}
Assume that we are in Situation \ref{sit:good}, and let $L/F$ be a finite Galois extension with Galois group $\Gamma$. Then, the map
$$\pi_{\Gamma}\colon \Xc(\Oc_L)^{\sharp,\Gamma} 
\to X(\Oc_L)^\Gamma=X(\Oc_F)$$
is proper and \'etale with finite fibres.
\end{proposition}
\begin{proof}
    We know this map to be finite \'etale by virtue of Proposition \ref{prop:finite}. The analogous property now follows from Lemma \ref{lemma:Gamma-etale}. Likewise, we obtain properness and finiteness, since these properties hold for the map $\pi \colon \Xc(\Oc_F)^{\#}\to X(\Oc_F)^{\#}$.
\end{proof}

Let us fix $r \in \mathbb{N} \setminus \{0\}$ and a smooth algebraic $\Dr$-stack $\Xc \to \Dr$.
We denote by $L$ the Galois extension $F(\mu_r(\bar F),\varpi^{1/r})$, where $\varpi$ is a uniformiser in $F$.

\begin{proposition}\label{defi:r-points}
Let $\Xc$ be a linear quotient stack over $\Oc_F$ as in Situation \ref{sit:good}. Then we have an identification $\Xc(\Dr)^{\sharp} =(\Xc(\Oc_L)^\sharp)^{\Gamma}= (U(\Oc_L)^\sharp/G(\Oc_L))^\Gamma$.

In particular, this endows $\Xc(\Dr)^{\sharp}$ with the structure of an $F$-analytic manifold independent of the presentation.  

Moreover, for any morphism $\Xc \to \Yc$ of such quotient stacks over $\Oc_F$, the resulting map $\Xc(\Dr)^{\sharp} \to \Xc(\Dr)^{\sharp}$ is analytic. 
\end{proposition}
\begin{proof}
Any $x \in \Xc(\Dr)^{\sharp}$ is given by a $\mu_r$-invariant $\Oc$-morphism $\Spec(\Oc_F(\varpi^{1/r})) \to \Xc$. By faithfully flat descent this is the same as a $\Gamma$-invariant morphism $\Spec(\Oc_L) \to \Xc$, which gives the desired identification. 

 By Proposition \ref{prop:Galois-quotient} and Lemmas \ref{lemma:Galois-fix} \& \ref{lemma:Galois}, we thus see that $\Xc(\Dr)^{\sharp}$ is the fixpoint set of a Galois action and thus an $F$-analytic manifold. The independence of presentation is proven as in Proposition \ref{presin}.    

 The last point by applying the preceding statement to the graph of such a morphism.
\end{proof}

Since the map $\pi_r \colon \Xc(\Dr)^{\sharp} \to X(\Oc_F)^{\sharp}$ is obtained by applying $\Gamma$-fixpoint loci to domain and co-domain of the morphism
$\pi\colon \Xc(\Oc_L)^{\sharp} \to X(\Oc_L)^{\sharp}$, we obtain:

\begin{proposition}\label{prop:r-finite}
Assume that we are in Situation \ref{sit:good}. Then, for any $r\geq 1$ the map
$$\pi_{r}\colon \Xc(\Dr)^{\sharp} 
\to X(\Oc_F)^\sharp$$
is proper and \'etale with finite fibres.
\end{proposition}
\begin{proof}
This is a direct consequence of Proposition \ref{prop:Gamma-finite}.
\end{proof}

\subsection{Buildings}
For a quotient stack $\CX=[U/G]$ with $G$ a reductive group over $\CO_F$, we will relate the set of twisted points $\CX(\Dinfty)$ to the (extended) Bruhat-Tits building $B(G_F)$ of $G_F$. This is a polysimplicial complex with a $G(F)$-action obtained by gluing affine spaces $A(T)\cong \BR^n$, the so-called apartments of $A(T)$ associated to the maximal split tori $T$ of $G_F$, along hyperplanes determined by the root system of $G$. It was constructed by Bruhat and Tits in \cite{BT1}.

We fix a hyperspecial point $\theta_0 \in B(G)$ whose parahoric integral model is equal to $G$. Then by the work of Landvogt \cite{MR1739403} (c.f. also \cite[Theorem 127]{MR4129311}) the pair $(B(G_F),\theta_0)$ is naturally functorial in the group scheme $G$ over $\CO_F$. That is, for any homomorphism $G \to H$ to a reductive group scheme $H$ over $\CO_F$ and any hyperspecial point $\theta'_0 \in B(H_F)$ with parahoric group $H$, there exists a unique $G(F) \to H(F)$-equivariant map $h\colon B(G_F) \to B(H_F)$ which sends $\theta_0$ to $\theta'_0$ and which is toral, that is for every maximal split torus $T$ of $G_F$ there exists a maximal split torus $T'$ of $H_F$ such that $T$ maps to $T'$ and such that $h$ restricts to an affine map $A(T) \to A(T')$.

\begin{definition} \label{RationalPtDef}
    Let $\Gamma \subset \BR$ be any additive subgroup containing $\BZ$. We let $B(G_F)_\Gamma \subset B(G_F)$ be the set of points $\theta \in B(G_F)$ which in some apartment $A(T)$ containing both $\theta_0$ and $\theta$ differ from $\theta_0$ by an element of $X_*(T)_\Gamma \subset X_*(T)_\BR$ (with respect to the structure of $A(T)$ as an affine $X_*(T)_\BR$-space).

    For any subset $S \subset B(G_F)$ we let $S_\Gamma \defeq S \cap B(G_F)_\Gamma.$
\end{definition}
\begin{rmk}
    In the situation of Definition \ref{RationalPtDef}, any two maximal split tori $T$ and $T'$ of $G_F$ whose apartments contain $\theta$ and $\theta_0$ are conjugate by an element of $G(F)$ which fixes both $\theta$ and $\theta_0$. Conjugation by such an element gives an affine isomorphism $A(T) \cong A(T')$ fixing $\theta$ and $\theta_0$. Hence the condition from Definition \ref{RationalPtDef} holds for one such $T$ if and only if it holds for all such $T$. 
\end{rmk}

The building of the group $\GL_n$ admits a concrete description, due to Goldman-Iwahori \cite{MR0144889}, as the set of non-archimedean norms $F \to \BR^{\geq 0}$ on the vector space $F^n$. Such norms are closely related to vector bundles on the spaces $\Dr$ by the following:

\begin{construction} \label{NormTOVB}
    Let $(V,|\cdot|)$ be a finite-dimensional normed vector space over $F$ for which $|V|$ is contained in $|F(\varpi^{1/r})|=\langle |\omega|^{-1/r} \rangle \subset \BR^{>0}$. To this we associate a vector bundle on $\Dr$ as follows: 
    By \cite[Prop. 1.1]{MR0144889} there exists a basis $(e_i)_{i=1}^n$ of $V$ such that $| \sum_i \lambda_i v_i|=\max_i |\lambda_i| |e_i|$ for all $(\lambda_i)_i \in F^n$. Then the $\Oc_{F(\varpi^{1/r})}$-lattice $M \subset V_{F(\varpi^{1/r})}$ given by those elements which have norm $\leq 1$ under the base change norm $|\cdot|_{F(\varpi^{1/r})}$ on $V_{F(\varpi^{1/r})}$ is given by
    \begin{equation*}
        M = \bigoplus_i {F(\varpi^{1/r})} \cdot (e_i \otimes \varpi^{w_i/r})
    \end{equation*}
    for suitable integers $w_i$ satisfying $|e_i|=|\varpi^{-w_i/r}|$. From this one sees that the ${F(\varpi^{1/r})}$-equivariant $\mu_r$-action on $V_{F(\varpi^{1/r})}$ restricts to a $\CO_{F(\varpi^{1/r})}$-equivariant $\mu_r$-action on $M$. This gives us the desired vector bundle on $\Dr.$

    Moreover one checks that this construction is functorial in $(V,|\cdot|)$ with respect to contractive morphisms of normed vector spaces.
\end{construction}

\begin{theorem} \label{NormVBEq}
    This construction gives an exact monoidal equivalence between the category of normed finite-dimensional vector spaces over $F$ with values in $\langle |\omega|^{-1/r} \rangle \subset \BR^{>0}$ and the category of vector bundles over $\Dr$.
\end{theorem}
\begin{proof}
    We first show essential surjectivity: A vector bundle over $\Dr$ is given by a finite free $\CO_{F(\varpi^{1/r})}$-module $M$ with an equivariant $\mu_r$-action. Then as an $\CO_F$-module we may decompose $M$ into its weight spaces $M^w$ for $w \in \BZ/r\BZ$:
    \begin{equation*}
        M=\bigoplus_{w \in \BZ/r\BZ} M^w
    \end{equation*}
    The fact that $\mu_r$ acts equivariantly over $\CO_{F(\varpi^{1/r})}$ translates to the fact that $\varpi^{1/r}$ maps each $M^w$ to $M^{w+1}$. We choose a graded basis $(\bar f_i)_i$ of the $\BZ/r\BZ$-graded $k_F$-vector space $\oplus_w M^w/\varpi^{1/r} M^{w-1}$. For each $i$ let $w_i \in \BZ/r\BZ$ be the degree of $\bar f_i$. Then if we lift each $\bar f_i$ to an element $f_i \in M^{w_i}$, by Nakayama's Lemma we obtain a basis $(f_i)_i$ of $M$. 

    The equivariant $\mu_r$-structure on the generic fibre $M_{F(\varpi^{1/r})}$ gives a descent datum to $F$ which yields an $F$-vector space $V$ over $F$ with an isomorphism $M_{F(\varpi^{1/r})} \cong V_{F(\varpi^{1/r})}$. Then for each $i$ the element $e_i \defeq \varpi^{-w_i/r} f_i$ is invariant under $\mu_r$ and hence contained in $V$. It follows from the above that these $e_i$ form a basis of $V$. If we endow $V$ with the norm given by
    \begin{equation*}
        |\sum_i \lambda_i e_i|=\max_i |\lambda_i| |\varpi^{-w_i/r}|
    \end{equation*}
    for all $(\lambda_i)_i \in F^n$, then by comparing with Construction \ref{NormTOVB} we see that we have constructed a preimage of $M$.

    Similarly one checks that the functor is fully faithful and exact.
\end{proof}

\begin{construction} \label{TorsEmb}
    By \cite[Theorem 4.16]{Ziegler22}, the above description of $B(\GL_n)$ as a space of norms generalizes to the following descriptions of the building $B(G_F)$ of a reductive group $G$ over $\CO_F$: Once we fix our base point $\theta_0 \in B(G)$, specifying a point of $B(G)$ amounts to specifying for any representation $\rho\colon G \to \operatorname{GL}(X)$ of $G$ on a finite free $\CO_F$-module $X$ a norm on $X_F$, in such a way that these norms are functorial and exact in $X$ as well as compatible with tensor products. This description of $B(G_F)$ is also functorial in the group $G$, and for any subgroup $\BZ \subset \Gamma \subset \BR$ an element of $B(G_F)$ is in $B(G_F)_\Gamma$ if and only if all the corresponding norms on the $X_F$ have values in $|\varpi|^\Gamma \subset \BR^{>0}$.

    Using this, we associate to a $G$-torsor $\mathcal{T}$ over $\Dr$ together with a trivialisation $\psi\colon \mathcal{T}_F \cong G_F$ of the generic fibre of $\mathcal{T}$ a point of $B(G_F)_{\frac{1}{r}\BZ}$: Let $\rho\colon G \to \operatorname{GL}(X)$ be any representation of $G$ on a finite free $\CO_F$-module $X$. By pushing out $\mathcal{T}$ along $\rho$ we obtain a vector bundle $M$ over $\Dr$ together with an isomorphism $M_F \cong X_F$. By Theorem \ref{NormVBEq} this vector bundle corresponds to a norm on $X_F$. These norms are functorial and exact in $X$, as well as compatible with tensor products and so give a point in $B(G_F)_{\frac{1}{r}\BZ}$.
\end{construction}

\begin{theorem}\label{mrdes}
    Construction \ref{TorsEmb} gives a bijection between the set of pairs $(\CT,\psi)$ consisting of a $G$-torsor $\CT$ over $\Dr$ together with a generic trivialization $\psi$ of $\CT$ and the set $B(G_F)_{\frac{1}{r}\BZ}$.

    In case $F$ contains all $r$-th roots of unity, this gives a canonical bijection
    \begin{equation*}
        (G({F(\varpi^{1/r})})/G(\CO_{F(\varpi^{1/r})}))^{\mu_r(F)} \cong B(G_F)_{\frac{1}{r}\BZ}.
    \end{equation*}

\end{theorem}
\begin{proof}
    The first claim follows from Theorem \ref{NormVBEq} and the fact that $G$-torsors over $\Dr$ correspond to $\CO_F$-linear exact tensor functors from the category of dualizable representations of $G$ over $\CO_F$ to the category of vector bundles over $\Dr$.

    For the second statement, we use the fact that $G$-torsors over $\Dr$ correspond to $\mu_r$-equivariant $G$-torsors over $\Spec(\CO_{F(\varpi^{1/r})})$. Since every $G$-torsor over $\Spec(\CO_{F(\varpi^{1/r})})$ is trivial, the set of $G$-torsors over $\Spec(\CO_{F(\varpi^{1/r})})$ together with a generic trivialization is in bijection with $G({F(\varpi^{1/r})})/G(\CO_{F(\varpi^{1/r})})$. Finally,  by the assumption on $F$ the $\mu_r$-equivariance of such a torsor is equivalent to the $\mu_r(F)$-invariance of the corresponding element of $G({F(\varpi^{1/r})})/G(\CO_{F(\varpi^{1/r})})$. 
\end{proof}
\subsection{Construction of the embedding}
 Let now $\CX=[U/G]$ be a quotient stack over $\CO_F$ with $U$ separated over $\CO_F$ and $G$ reductive over $\CO_F$. For a point $x \in \CX(F)$, we consider the fibre $\CX(\Dinfty)_x$ above $x$ of the functor $\CX(\Dinfty) \to \CX(F)$. That is $\CX(\Dinfty)_x$ is the set of isomorphism classes of pairs $(y,\psi)$ where $y$ is an object of $\CX(\Dinfty)$ and $\psi$ is an isomorphism $y|_F \isoto x$. We also denote by $\CX(\Dr)_x \subset \CX(\Dinfty)_x$ the subset given by those pairs with $y \in \CX(\Dr)$.
 
\begin{construction} \label{EmbCons}
    Let $x \in \CX(F)$ be such that the $G$-torsor $U|_x$ over $F$ is trivial and fix a point $\tilde x \in U|_x(F)$ of this torsor. We construct an embedding
    \begin{equation*}
        b_{\tilde x}\colon \CX(\Dinfty)_x \into B(G_F)_\BQ
    \end{equation*}
    as follows: For any $r \geq 1$, an element $y$ of $\CX(\Dr)$ is given by the $G$-torsor $\CT=U|_y$ over $\Dr$ together with a $G$-equivariant morphism $\CT \to U$. Via our chosen point $\tilde x$, an identification of $y|_F$ with $x$ gives identifications $\CT_F \cong U|_x \cong G_F$ of $G_F$-torsors. Hence via Construction \ref{TorsEmb} we obtain a point $b_{\tilde x}(y) \in B(G_F)_{\frac{1}{r}\BZ}$. This construction is functorial in $r$ and thus gives a map as desired.
\end{construction}

A different choice of base point $\tilde x$ is of the form $g\cdot \tilde x$ for some $g \in G(F)$. So the following describes the effects of a different choice of $\tilde x$:
\begin{lemma} \label{bBP}
   In the situation of Construction \ref{EmbCons}, for any $g \in G(F)$ and $y \in \CX(\Dr)$ we have $b_{g\cdot \tilde x}(y)=g\cdot b_{\tilde x}(y)$. 
\end{lemma}
\begin{proof}
    This follows from the construction of $b_{g \cdot \tilde x}$ using the fact that in \cite[Theorem 4.16]{Ziegler22}, replacing the base point $\theta_0$ by $g\cdot \theta_0$ changes the resulting bijection by the action of $g$ on $B(G_F)$.
\end{proof}
We first note the following functoriality property of this construction:
\begin{construction} \label{bFuncCons}
    Let $G' \to G$ be a homomorphism of reductive group schemes over $\CO_F$ and $\CX'$ the resulting quotient stack $[U/G']$ with the canonical morphism $\CX' \to \CX$. Let $x' \in \CX'(F)$ be a point with image $x \in \CX(F)$. There is a natural morphism of torsors $U|_{x'} \to U|_x$ which is equivariant with respect to $G' \to G$. We assume that the $G'$-torsor $U|_{x'}$ is trivial and choose a point $\tilde x' \in U|_{x'}(F)$. Let $\tilde x$ be the image of $\tilde x'$ in $U|_{x}(F)$. We obtain the following diagram:
       \begin{equation} \label{bFuncDiag}
        \xymatrix{
            \CX'(\Dinfty)_{x'} \ar[r]^{b_{\tilde x'}} \ar[d] & B(G'_F)_{\Qb} \ar[d]\\
            \CX(\Dinfty)_x \ar[r]_{b_{\tilde x}}  & B(G_F)_{\Qb}         
        }
    \end{equation}
\end{construction}
\begin{proposition} \label{bFunc}
     The diagram \eqref{bFuncDiag} is Cartesian.
\end{proposition}
\begin{proof}
    The commutativity follows from the various constructions. To see that the diagram is Cartesian, we argue as follows: Consider a point $y$ in $\CX(\Dinfty)_x$, the associated $G$-torsor $\CT=U|_y$ and the trivialization $\CT_F \cong U|x \cong G$ as in Construction \ref{EmbCons}. Then the point $b_{\tilde x}(y)$ encodes the the $G$-torsor $\CT$ together with this trivialization of $\CT_F$. A lift of $b_{\tilde x}(y)$ to $B(G'_F)_{\Qb}$ amounts to giving, after possibly enlarging $r$, a $G'$-torsor $\CT'$ together with an equivariant morphism $\CT'\to \CT$ and a compatible trivialization $\CT'_F \cong G'_F$. But such a datum amounts exactly to a lift of $y$ to $\CX'(\Dinfty)_{x'}$.
\end{proof}

From now on we fix a point $x \in \CX(F)$ as in Construction \ref{EmbCons} and assume in addition that the automorphism group of $x$ is trivial. Then the natural morphism $U|_x \to U$ is a closed immersion with image $c^{-1}(x)$, and so the choice of $\tilde x$ amounts to the choice of a point $\tilde x \in c^{-1}(x)$. We fix such a choice.

We give the following more explicit description of the map $b_{\tilde x}$. First we describe its restriction to $\CX(\CO_F)_x$: We note that by Lang's theorem $\CX(\CO_F)=G(\CO_F) \backslash U(\CO_F)$. So if $\CX(\CO_F)_x$ is non-empty, then there exists a point $u \in U(\CO_F)$ whose generic fibre lies over $x$. We fix such a $u$ and take $\tilde x$ to be $u|_{\Spec(F)}$. Then we find 
\begin{equation*}
    \CX(\CO_F)_x= G(\CO_F) \backslash (G(F)\cdot u \cap U(\CO_F)) \subset \CX(\CO_F)
\end{equation*}

\begin{lemma} \label{bDesc}
    The map $b_{\tilde x}$ sends the class of a point $g\cdot u \in U(\CO_F)$ for some $g \in G(F)$ the point $g \cdot \theta_0$.
\end{lemma}
\begin{proof}
    The image of $g \cdot u$ in $\CX(\CO_F)$ is given by the trivial $G$-torsor $\CT=G$ together with the morphism $G \to U, \; h \mapsto h\cdot u$. Hence the resulting isomorphism $\CT_F \to G$ from Construction \ref{EmbCons} is given by the action of $g$. 
\end{proof}

\begin{lemma} \label{bBC}
    The diagram
    \begin{equation*}
        \xymatrix{
            \CX(\Dr)_x \ar[r]^{b_{\tilde x}} \ar[d] & B(G_F) \ar[d] \\
            \CX(\CO_{F(\varpi^{1/r})})_x \ar[r]_{b'_{\tilde x}} & B(G_{F(\varpi^{1/r})})
        },
    \end{equation*}
    in which the map $b'_{\tilde x}$ is the one associated to the stack $\CX_{F(\varpi^{1/r})}$ over ${F(\varpi^{1/r})}$, is Cartesian.
\end{lemma}
\begin{proof}
    The commutativity of the diagram follows from the various constructions. The see that it is Cartesian, we argue as follows: A point $y$ of $\CX(\CO_{F(\varpi^{1/r})})_x$ is given by a $G$-torsor $\CT=U|_y$ over $\CO_{F(\varpi^{1/r})}$ together with a $G$-equivariant morphism $\CT \to U_{\CO_{F(\varpi^{1/r})}}$ the image of whose generic fibre is the $G$-orbit $c^{-1}(x)_{F(\varpi^{1/r})}$ of $\tilde x$ in $U_{F(\varpi^{1/r})}$. Then giving a lift of $b'_{\tilde x}(y)$ to $B(G_F)$ amounts to descending the torsor $\CT$ to $\Dr$, i.e. to giving a $\mu_r$-equivariant structure on $\CT$ which extends the one on $\CT$ coming from the given identification $\CT_F \cong c^{-1}(X)_{F(\varpi^{1/r})}$. By flatness, if such a $\mu_r$-equivariant structure exists, the morphism $\CT \to U_{\CO_{F(\varpi^{1/r})}}$ is automatically $\mu_r$-equivariant, since it is so in the generic fibre. Hence giving a lift of $b'_{\tilde x}(y) \to B(G_F)$ is equivalent to giving a lift of $y$ to $\CX(\Dr)_x$.   
\end{proof}

\subsection{Structure of the image of the embedding}
\begin{situation} \label{AffineCoverSit}
    We assume that $\CX$ satisfies the assumptions of Situation \ref{sit:good} and that $\CX$ admits a presentation $[U/G]$ with $G$ reductive over $\CO_F$ and $U$ separated and of finite type over $\Spec(\CO_F)$ such that moreover for some split subtorus $T \subset G$ over $\CO_F$ whose generic fibre is a maximal split torus of $G_F$, there exists a finite open covering $U=\cup_i U_i$ by affine $T$-invariant open subschemes $U_i \subset U$.
\end{situation}

\begin{rmk}
    This additional condition on the $G$-action on $U$ is satisfied in the following cases:
    \begin{enumerate}
        \item Tautologically: If $U$ is affine.
        \item By \cite[3.11]{MR0387294}: If $U$ is flat over $\Spec(\CO_F)$ and its special fibre is geometrically normal and geometrically integral. 
        \item As a special case of (ii): If $U$ is smooth over $\Spec(\CO_F)$.
    \end{enumerate}

    Furthermore, e.g. by \cite[3.18]{Ziegler22} any two split tori $T \subset G$ as in this condition are conjugate by an element of $G(\CO_F)$. Hence this conditions holds for one such $T$ if and only if it holds for all such $T$. 
\end{rmk}

The building $B(G_F)$ admits natural $G(F)$-invariant metrics whose restriction to each apartment are Euclidean and which make $B(G_F)$ into a CAT(0)-space, and so in particular into a geodesic space. Although these metrics are not unique, the resulting notion of the geodesic $[\theta,\theta']$ between two points $\theta, \theta' \in B(G_F)$ is. Namely, this geodesic is the straight line segment between $\theta$ and $\theta'$ in any apartment containing these two points. This gives a notion of a convex subset of $B(G_F)$ as a subset which contains the geodesic between any two of its elements. By a rational polytope in $B(G_F)$ we mean a subset which is the convex hull of finitely many points in $B(G_F)_\BQ$. We also note that this description of geodesics implies that toral maps of buildings send geodesics to geodesics.

 \begin{lemma} \label{VertexBound}
     Consider integers $n,m \geq 1$ as well as a matrix $A=(a_{ij}) \in \BQ^{m \times n}$. For any $y=(y_i) \in \BZ^m$ we consider the polyhedron
     \begin{equation*}
         P_y \defeq \{ (x_j) \in \BR^n \mid \forall i\colon \sum_j a_{ij} x_j \geq y_i \} \subset \BR^n.
     \end{equation*}

     Then there exists an integer $N \geq 1$ such that for all $y \in \BZ^m$ for which $P_y$ is bounded, all vertices of $P_y$ are contained in $(\frac{1}{N}\BZ)^n \subset \BR^n$
 \end{lemma}
 \begin{proof}
     Let $y \in \BZ^m$ be such that $P_y$ is bounded and $v$ a vertex of $P_y$. We consider all affine equations $\sum_j a_{ij} x_j = y_i$ which are satisfied by $v$ and claim that this system of equations has $v$ as its unique solution in $\BR^m$. Indeed, if the affine subspace $W$ cut out by these equations had positive dimension, then by considering the remaining inequalities defining $P_y$, it would follow that some neighbourhood of $v$ in $W$ is contained in $P_y$, which is not possible since $v$ is a vertex of $P_y$. 

     This shows that any $N$ for which $\frac{1}{N}\BZ$ contains all coefficients of $A$ and the inverses of all minors of $A$ has the required property.
 \end{proof}
\begin{theorem} \label{Polytopes}
    There exists an integer $N \geq 1$, such that for all $x \in \pi^{-1}(W)(F)$ and all lifts $\tilde x$ of $x$, there exist finitely many rational polytopes $P_1, \hdots, P_m \subset B(G_F)$, each of which is the convex hull of finitely many points in $b_{\tilde x}(\CX(\Dr[N])_x)$, such that $b_{\tilde x}(\CX(\Dinfty)_x)$ is equal to $\cup_i P_{i,\BQ}$. 
\end{theorem}
\begin{proof}
 We fix a presentation $\CX=[U/G]$ as in Situation \ref{AffineCoverSit} with $G$ a general linear group.

    Let $T \subset G$ be a split closed subtorus whose generic fibre is a maximal split torus of $G_F$. Then by Lemma \ref{bFunc} the intersection of $b_{\tilde x}(\CX(\Dinfty)_x)$ with the apartment $A(T)$ of $T$ is equal to the image of the map $b'_{\tilde x}\colon [U/T](\Dinfty)_{x'} \to B(T_F)$, and we first consider this set $b'_{\tilde x}([U/T](\Dinfty)_{x'})$.
    
     We cover $U$ by finitely many open affine $T$-invariant subschemes $U_i$. Then
     \begin{equation*}
         [U/T](\Dinfty)_{x'} = \cup_{i\colon \tilde x \in U_i(F)} [U_i/T](\Dinfty)_{x'},
     \end{equation*}
    
    and so we consider a single $U_i$ containing $\tilde x$. Then by choosing finitely many generators of the ring of global functions of $U_i$ which are homogenous with respect to the induced action of $T$, we can embed $U_i$ equivariantly into an affine $n$-spaces $\BA^n$ equipped with a diagonal $T$-action, i.e. such that $T$ acts by $t\cdot (u_j)_{1\leq j \leq n}=(\chi_j(t)u_j)$ for some homomorphisms $\chi_j\colon T \to \BG_m$. By Lemma \ref{bFunc} the image of $[U_i/T](\Dinfty)_{x'}$ in $B(T_F)$ is the same as the image of $[\mathbb{A}^n/T](\Dinfty)_{x'}$. 

    For any algebraic extension $L$ of $F$ and any section $t \in T(L)$, the element $t \cdot \tilde x$ is in $\BA^n(\CO_L)$ if and only if $v(\chi_j(t)) \geq -v(\tilde x_j)$. Using Lemmas \ref{bFunc} and \ref{bDesc} this implies that the image of $b_{\tilde x}$ consists of the rational points of the subset of $B(T_F)\cong X_*(T)_{\BR}$ defined by these linear inequalities $v(\chi_j) \geq -v(\tilde x_j)$. These inequalities are of the kind considered in Lemma \ref{VertexBound}, and we fix an $N \geq 1$ given by this lemma. Then to obtain the desired description of $b'_{\tilde x}([U_i/T](\Dinfty)_{x'})$, it suffices to see that the locus cut out by these inequalities is bounded. However, if it were unbounded, the set $b_{\tilde x}(\CX(\Dr[N])_x)$ would be infinite, which would contradict Proposition \ref{prop:r-finite}.

    We apply the above to all the $U_i$ and choose a single $N$ large enough to apply to all of them. Then we know that for any single $T \subset G$ as above, the intersection $A(T) \cap b_{\tilde x}(\CX(\Dinfty)_x)$ has the desired description. Any other such torus $T' \subset G$ is of the form $^{g}T$ for some $g \in G(\CO_F)$. Hence the action of $g^{-1}$ on $U$ gives an isomorphism $[U/T'] \isoto [U/T]$ which by Lemma \ref{bFunc} maps $A(T') \cap b_{\tilde x}(\CX(\Dinfty)_x)$ to $A(T) \cap b_{g^{-1}\cdot \tilde x}(\CX(\Dinfty)_x)$. So, up to the integral isomorphism $X_*(T') \cong X_*(T)$ given by $g$, the intersection $A(T') \cap b_{\tilde x}(\CX(\Dinfty)_x)$ is described on each $U_i$ by the linear inequalities $v(\chi_j) \geq -v((g^{-1}\cdot \tilde x)_j)$ as above. So we obtain the desired description for $A(T) \cap b_{\tilde x}(\CX(\Dinfty)_x)$ for all such $T$ with respect to our chosen $N$. The collection of all apartments $A(T)$ for such $T$ covers $B(G_F)$. 

    Next we claim that $b_{\tilde x}(\CX(\Dinfty)_x)$ is bounded. If it were not, the above would imply that there are infinitely many points in $b_{\tilde x}(\CX(\Dr[N])_x)$, which would again contradict Proposition \ref{prop:r-finite}. Hence, by the local finiteness of the apartments in $B(G_F)$, there exists finitely many $T \subset G$ as above whose apartments cover $b_{\tilde x}(\CX(\Dinfty)_x)$. So by using the above for these finitely many apartments we obtain the claim.
\end{proof}

\begin{corollary} \label{PiImageBounded}
    There exists an integer $N\geq 1$ for which the maps $\CX(\Dinfty)^\sharp \to X(\CO_F)^\sharp$ and $\CX(\Dr[N])^\sharp \to X(\CO_F)^\sharp$ have the same image.
\end{corollary}
\begin{proof}
    If we take $N$ as in Theorem \ref{Polytopes}, then we see that if $\CX(\Dinfty)_x$ is non-empty, any of the $P_i$ appearing must have a point coming from $\CX(\Dr[N])_x$. So such an $N$ has the required property.
\end{proof}

\subsection{Convexity of the image}
Next we want to show that in case the stack $\CX$ is S-complete over $\CO_F$ in the sense of \cite[Definition 3.38]{AHH23}, the image of $b_{\tilde x}$ is the set of rational points of a single convex bounded subset of $B(G_F)$. For this we consider the stack $$\overline{ST}_{\CO_F}=[\Spec(\CO_F[s,t]/(st-\varpi))/\BG_m],$$ where the $\BG_m$-action has weight $1$ on $s$ and weight $-1$ on $t$. This stack has a canonical stacky closed point $0 \colon B_{k_F} \BG_m \into \overline{ST}_{\CO_F}(k_F)$ whose complement is the scheme obtained by gluing two copies of $\Spec(\CO_F)$ along their generic fibre $\Spec(F)$. We denote the unique $F$-point of $\overline{ST}_{\CO_F}$ by $\ast$ and by $V$ the scheme $\Spec(\CO_F[s,t]/(st-\varpi))$ with the above $\BG_m$-action.

\begin{lemma}
    For any $\tilde x \in V(F)$ the image of $b_{\tilde x} \colon \overline{ST}_{\CO_F}(\Dinfty)_x \to B(\BG_m) \cong \BR$ is an interval $[n, n+1]$ for some $n \in \BZ$, whose end-points are the images of the two $\CO_F$-points of $\overline{ST}_{\CO_F}$.
\end{lemma}
\begin{proof}
    If we take for example $\tilde x =(1,\varpi) \in V(\CO_F) \subset V(F)$, then for any $r\geq 1$ the set $V(\Dr)$ consists of the points $t \cdot \tilde x$ for all $t \in \BG_m(F(\varpi^{1/r}))$ satisfying $|\varpi \leq |t| \leq 1$. Using Lemmas \ref{bDesc} and \ref{bBC} we find $b_{\tilde x}(t\cdot \tilde x)=-v(t) \in [-1,0]_{\BQ}$. (Here the minus sign comes from the fact that in Bruhat-Tits theory the group $\BG_m(F)$ acts on $B(\BG_m) \cong \mathbb{R}$ through translation by minus the valuation of an element.)
\end{proof}

    The scheme $V^\circ\defeq V \setminus \{0\}=\Spec((\CO_F[s,t]/(st-\pi))_{(s,t,\pi)})$ admits two open immersions $\BG_{m,\CO_F} \into V^\circ$, given by the ring homomorphisms $(\CO_F[s,t]/(st-\pi))_{(s,t,\pi)} \to \CO_F[s^\pm],\; s\mapsto s, t\mapsto \pi/s$ and $(\CO_F[s,t]/(st-\pi))_{(s,t,\pi)} \to \CO_F[t^\pm],\; s \mapsto \pi/t,\; t \mapsto t$. These two open subschemes cover $V^\circ$, and their intersection is given by the two generic fibres $\BG_{m,F}$ glued via $s \mapsto \pi/t$.

\begin{lemma}\label{SomeHomCons}
    For any two sections $u, v \in U(\CO_F)$ whose generic fibres are conjugate under $G(F)$, there exists a homomorphism $\chi\colon \BG_{m} \to G$ over $\CO_F$ and a $\chi$-equivariant morphism $V^\circ \to U$ such that the images of the compositions with the two open immersions $\BG_{m,\CO_F} \into V^\circ$ land in the $G$-orbits of $x$ and $y$ respectively.
\end{lemma}
\begin{proof}
    Using the Cartan decomposition of $G(F)$, after replacing $u$ and $v$ by some $G(\CO_F)$-conjugates, there exists a homomorphism $\chi\colon \BG_{m,\CO_F} \to G$ such that $u|_F=\chi(\varpi)\cdot v|_F$. Then the two morphisms $\BG_{m,\CO_F} \to U, h \mapsto \chi(h)\cdot u$ and $\BG_{m,\CO_F} \to U, h \mapsto \chi(h)^{-1}\cdot v$ glue to a morphism $V^\circ \to U$ as desired.
\end{proof}

\begin{theorem}\label{conhul}
    If $\CX$ is S-complete, then there exists an integer $r \geq 1$, such that for all $x \in \pi^{-1}(W)(F)$ and all lifts $\tilde x$ of $x$, there exists a rational polytope $P \subset B(G_F)$, which is the convex hull of finitely many points in $b_{\tilde x}(\CX(\Dr)_x)$, such that $b_{\tilde x}(\CX(\Dinfty)_x)$ is equal to $P_\BQ$. 
\end{theorem}
\begin{proof}
    By Theorem \ref{Polytopes}, it suffices to show, that for any two points $y_1$ and $y_2$ in $\CX(\Dinfty)_x$, any rational point on the geodesic from $b_{\tilde x}(y_1)$ to $b_{\tilde x}(y_2)$ is again in $b_{\tilde x}(\CX(\Dinfty)_x)$. Using Lemma \ref{bBC}, we may replace $F$ by $F(\varpi^{1/r})$ for some suitable $r$ and hence may assume that the $y_i$ are elements of $\CX(\CO_F)$. We choose lifts $u$ and $v$ of the $y_i$ and obtain a $\chi$-equivariant morphism $V^\circ \to U$ as in Lemma \ref{SomeHomCons}. Such a morphism induces a morphism $\overline{ST}_{\CO_F} \setminus \{0\} \to \CX$ which sends the two $\CO_F$-points of $\overline{ST}_{\CO_F}$ to $y_1$ and $y_2$ respectively. By S-completeness this morphism extends uniquely to a morphism $\overline{ST}_{\CO_F} \to \CX$. 

    We claim that the morphism $V^\circ \to U$ extends to a $\chi$-equivariant morphism $V \to U$ which lies over the morphism $\overline{ST}_{\CO_F} \to \CX$. Finding such an extension amounts to finding a splitting of the $G$-torsor on $V$ obtained by pulling back the $G$-torsor $U \to \CX$ along the composition $V \to \overline{ST}_{\CO_F} \to \CX$ which extends the given splitting of this torsor on $V^\circ$. But since $V$ is an affine regular scheme and $\{0\} \subset V$ has codimension $2$, the existence of such a splitting follows from Hartogs's extension theorem.

    By Lemma \ref{bFunc}, the existence of this morphism $V \to U$ gives us a commutative diagram
    \begin{equation*}
        \xymatrix{
            \overline{ST}_{\CO_F}(\Dinfty)_{\ast} \ar[r] \ar[d] & B(\BG_{m,F}) \ar[d] \\
            \CX(\Dinfty)_x \ar[r] & B(G_F),
        }
    \end{equation*}
    in which the right vertical arrow is induced by $\chi$ and hence sends geodesics to geodesics. Hence this arrow sends the interval $b_*(\overline{ST}_{\CO_F})$ to the geodesic from $b_{\tilde x}(y_1)$ to $b_{\tilde x}(y_2)$. This proves the claim.
\end{proof}

\section{$p$-adic integration}\label{pins}

\subsection{Volume formula}\label{vol-for}

In this subsection we give a formula for the $p$-adic volume of $\Dr[r]$-points of a smooth quotient stack as in Situation \ref{sit:good}. 

We start by introducing a natural \textit{weight function}. For this let $\Xc$ be a smooth finite type Artin stack over a perfect field $k$. The cyclotomic inertia stack $I_{\hat{\mu}}\Xc$ is defined as
\[I_{\hat{\mu}}\Xc = \varprojlim I_{\mu_r} \Xc,\]
where $I_{\mu_r} \Xc = \Hom(B{\mu_r}, \Xc)$.
A point in $I_{\hat{\mu}}\Xc(k)$ is given by a pair $(x,\phi)$, where $x \in \Xc(k)$ and $\phi:\hat{\mu} \to \Aut(x)$. The weight $w(x,\phi) \in \BQ$ is defined as in the case of Deligne-Mumford stacks \cite[Definition 2.20]{GWZ20b}, with the difference, that instead of a tangent space of $\Xc$ at $x$, we must work with a tangent complex, a two-term complex of $k$-vector spaces $T_x\Xc = [V_{-1} \to V_0]$, well defined up to quasi-isomorphism. 

Furthermore, $T_x\Xc$ can be realised as a complex of vector spaces endowed with actions of the stabiliser group $\Aut(x)$. This is a consequence of the formalism of (co)tangent complexes, according to which the tangent complex $T\Xc$ can be defined as an object of the derived category $D(\Xc)$ of quasi-coherent sheaves on $\Xc$. In particular, we may consider the pullback $T\Xc|_{B\Aut(x)}$ along the morphism $B\Aut(x) \to \Xc$, which yields an object in $D(B\Aut(x))$. Quasi-coherent sheaves on the classifying stack $BG$ of a group $G$ correspond to $G$-representations, and we therefore obtain the requisite $\Aut(x)$-action on $T_x\Xc$.

For every pair $(x,\phi)$ as above, we therefore obtain a $\hat{\mu}$-action on $T_x\Xc$ by restricting the natural $\Aut(x)$-action along $\phi$.

\begin{definition}\label{weightf}
    \begin{enumerate}
    
        \item For a $\hat\mu$-representation $V$ over $k$ with character decomposition $V=\bigoplus_{\chi \in \hat\mu^*}{V^{\oplus n\chi}_\chi}$, we define 
    \[w(V) =  \sum_{\chi \in  \hat\mu^* } n_\chi \cdot \chi \in \BQ,\]
    where we make the somewhat unconventional identification
    \[\hat\mu^* = \BQ/\BZ \cong (0,1] \cap \BQ.\]
    \item For any $(x,\phi) \in I_{\hat{\mu}}\Xc(k)$ with tangent complex $[V_{-1} \to V_0]$ we define 
    \[w(x,\phi) = w(V_0)-w(V_{-1}).\]
    \end{enumerate}
\end{definition}

In the case of a global quotient stack $\Xc=[U/G]$ we have $T_x\Xc = [\Lie(G) \to T_{\tilde{x}}U]$ with $\tilde{x} \in U$ any lift of $x$. Thus, $w(x,\phi) = w(T_{\tilde x}U) - w(\Lie(G))$, where $T_{\tilde x}U$ is a $\hat \mu$-representation via $\hat \mu \to \Aut(x) \cong \stab(\tilde{x}) \subset G$ and $\Lie(G)$ via the adjoint action.

Let $\Xc = [U/G]$ be a linear quotient stack over $\Oc_F$ as in Situation \ref{sit:good}. We assume further that $\Xc$ is smooth over $\Oc_F$ and that its canonical bundle $\omega_{\Xc}$, i.e. the determinant of its cotangent complex, is Zariski-locally trivial. For any $r\geq 1$ we define a natural measure on $\Xc(\Dr)$ as follows:

Choose an open cover $\Xc = \bigcup_i \Xc_i$ such that $\omega_{\Xc|\Xc_i}$ admits a trivializing section $\omega_i$. Then $\omega_i$ defines a volume form, i.e. a non-vanishing top-form, on $\Xc_i \cap \Xc_W \cong \pi(\Xc_i) \cap W$, where $\pi: \Xc \to X$. Integrating the absolute value $|\omega_i|$ defines a Borel measure $\mu_i$ on $(\pi(\Xc_i) \cap W)(F)$ and the same proof as \cite[Proposition 3.1.2]{COW21} gives:

\begin{lemma} The measures $\mu_i$ glue to a Borel measure $\mu_{orb}$ on $X^\sharp = X(\Oc_F) \cap W(F)$ which is independent of the choice of trivializing sections $\omega_i$.    
\end{lemma}

By Proposition \ref{prop:Gamma-finite}, the map 
\[ \pi_r: \Xc(\Dr)^\sharp \to X^\sharp   \]
is étale and we thus obtain a measure on $\Xc(\Dr)^\sharp$, still denoted by $\mu_{orb}$, by pulling back $\mu_{orb}$ along $\pi_r$. 

To compute the volume of $\Xc(\Dr)^\sharp$ with respect to $\mu_{orb}$ we consider the specialization map 
\begin{equation}\label{ersp} e_r: \Xc(\Dr)^\sharp \to I_{\mu_r} \Xc(k) = \Hom(B_{\mu_r}, \Xc)(k),  \end{equation}
given by restriction to $B\mu_r \subset \Dr$.

\begin{theorem}\label{orbiorbi} Assume that $F$ does not contain any $p$-power roots of unity. Then for any $y \in I_{\mu_r} \Xc(k)$ we have 
\[ \mu_{orb}( e_r^{-1}(y)) = \frac{q^{-w(y)}}{|\Aut(y)(k)|}. \] \end{theorem}

\begin{corollary}\label{finvol} Assume that $\pi_\infty: \Xc(\Dinfty)^\sharp \to X^\sharp $ is surjective. Then the $F$-analytic manifold $X^\sharp$ has finite volume with respect to $\mu_{orb}$. 
\end{corollary}
\begin{proof}
By Corollary \ref{PiImageBounded}, there exists an integer $s\geq 1$ for which $\pi_s: \Xc(\Dr[s])^\sharp \to X^\sharp $ is surjective. By Proposition \ref{prop:r-finite}, the map $\pi_s$ is étale with finite fibres and thus
\[\int_{\Xc(\Dr[s])^\sharp}\mu_{orb} \geq \int_{X(\Oc_F)^\sharp}\mu_{orb}.
  \]
But Theorem \ref{orbiorbi} shows in particular that $\int_{\Xc(\Dr[s])^\sharp}\mu_{orb}$ is finite and the corollary follows.    
\end{proof}

\subsection{Proof of Theorem \ref{orbiorbi}}

We start with an important special case.
\begin{lemma}\label{locc} Let $\Xc = [U/\mu_r]$ be a generically representable smooth quotient stack over $\Oc_F$. Then for any $x\in \Xc(k)$ with $\Aut(x) = \mu_r$ we have
\[  \mu_{orb}( e_r^{-1}(x,\id))  = \frac{q^{-w(y)}}{|\Aut(x,\id)(k)|} = \frac{q^{-w(y)}}{|\mu_r(k)|} ,  \]
where $\id$ denotes the identity on $\mu_r$.
\end{lemma}
\begin{proof} The proof is essentially the same as in \cite[Lemma 2.22]{GWZ20b}, which in turn is taken from \cite{DL02}. Up to replacing $U$ by an unramified $\mu_r$-twist, we may assume that $x$ admits a lift $\tilde{x} \in U^{\mu_r}(k)$. Then by Lemma \ref{diagdiag} below, there exists a $\mu_r$-invariant Zariski-open neighborhood $V$ of $\tilde x$ and a $\mu_r$-equivariant étale morphism $V \to \BA^d$ sending $\tilde{x}$ to $0$. By construction the orbifold measure is compatible with equivariant étale morphisms and thus it suffices to show the lemma for $U=\BA^d$ with a linear $\mu_r$-action and $x=0$. Since $\mu_r$ is diagonalisable, we may further assume that the action is diagonal with characters $\chi_1,\dots,\chi_d$.

In this situation we can describe $ e_r^{-1}(0,\id)$ explicitly as follows. For $1 \leq i \leq d$, let $1\leq c_i \leq r$ be such that $\chi_i = c_i/r$ under the identification $\mu_r^* \cong \frac{1}{r} \BZ \cap (0,1]$. Consider the map
\begin{align*}
    \lambda: \BA^d(\Oc_F) &\rightarrow [\BA^d/\mu_r](\Dr)\\
    (x_1,\dots,x_d) &\mapsto (\varpi^{c_1/r}x_1,\dots,\varpi^{c_d/r}x_d).
\end{align*} 

By \ref{defi:r-points} the map $\lambda$ is an analytic map of $F$-analytic manifolds. By looking at the restriction to the generic point $\Spec(F)$, we see that $\lambda$ is $|\mu_r(F)|$-to-$1$ outside a set of measure $0$. Finally essentially by construction, see also \cite[(2.3.4)]{DL02}, the image of $\lambda$ is $e_r^{-1}(0,\id)$. Hence we get
\[ \int_{e_r^{-1}(0,\id)} \mu_{orb} = \frac{1}{|\mu_r(F)|}\int_{\Oc_F^d}\lambda^*\mu_{orb} =  \frac{1}{|\mu_r(F)|}\int_{\Oc_F^d} |\varpi^{\sum c_i}|^{1/r}dx_1\dots dx_r = \frac{q^{-w(0,\id)}}{|\mu_r(F)|}.  \]
Finally, by assumption, $F$ does not contain any $p$-power roots of unity and hence $|\mu_r(F) | =| \mu_r(k) |$.
\end{proof}

\begin{lemma}\label{diagdiag}
Let $G$ be a finite diagonalisable group scheme acting on a smooth $\Oc_F$-scheme $U$ of relative dimension $d$. Assume that $x \in U^G(k_F)$. Then, there exists a Zariski-open neighbourhood $V \subset U$ of $x$ and a $G$-equivariant \'etale morphism $g\colon V\to \Ab_{\Oc_F}^d$, where $g(x)=0$ and the affine space is endowed with the infinitesimal $G$-action on $T_xU$.     
\end{lemma}
\begin{proof}
    Let $A$ be the character group of $G$. Then giving a representation of $G$ on some module $M$ is equivalent to giving a weight decomposition $M=\oplus_{a\in A} M^a$. We call the elements of $M^a$ homogeneous of weight $a$.
    
    Let $\mathfrak{m}_x$ be the maximal ideal of the local ring $\Oc_{U,x}$ of $U$ at $x$. Since $x$ is fixed by $G$, the group $G$ acts on these objects. We choose a basis $u_1,\hdots,u_d$ of the cotangent space $\mathfrak{m}_x/\mathfrak{m}^2_x$ of $U$ at $x$ consisting of elements $u_i$ which are homogeneous of weight $a_i$. The ring $\Oc_{U,x}$ is the colimit of rings $\Oc_U(V)$, where $V$ runs over Zariski-open affine neighbourhoods of $x$. We may furthermore assume $V$ to be preserved by $G$ since $x$ is a fixpoint and $G$ is finite. So for some such $G$-invariant $V$ there exist elements $f_i \in \Oc_U(V)$ mapping to the $u_i$. The group $G$ acts on $\Oc_U(V)$ and the homomorphism $\Oc_U(V) \to \Oc_{U,x}/\mathfrak{m}_x^2$ is $G$-equivariant. Hence after throwing away the homogeneous components of each $f_i$ which are not of weight $a_i$, we can assume that each $f_i$ is homogeneous of weight $a_i$. Then the $f_i$ define a morphism $V \to \Ab^d$ which is equivariant if we let $G$ act on the $i$-th component of $\Ab^d$ with weight $a_i$. By construction this morphism is \'etale at $x$.
    
\end{proof}

For the general case let us fix a presentation $\Xc=[U/G]$ such that $G/\Oc_F$ is fibrewise connected and $H^1(F,G) = \{0\}$. Let $y = (x,\phi) \in I_{\mu_r} \Xc(k)$ with $x\in \Xc(k)$ and $\phi:\mu_r \to \Aut(x)$. Then given $x\in \Xc(k)$, since every $G_k$-torsor over $k$ is trivial by Lang's theorem, we may pick a lift $\tilde{x} \in U(k)$, which induces an identification $\Aut(x) \cong \stab(\tilde{x}) \subset G_k$. We still write $\phi$ for the induced homomorphism $\mu_r\to  \stab(\tilde{x})$ and furthermore fix a lift 
\[\tilde\phi: \mu_r \to G\]
of this homomorphism over $\Spec(\Oc_F)$, which exists e.g. by \cite[3.14]{Ziegler22}. By means of $\tilde\phi$ we obtain an action of $\mu_r$ on $U$ and we write $\widetilde\Xc = [U/\mu_r]$. We denote by
\[p: \widetilde\Xc \to \Xc\]
the resulting quotient morphism.

If we write $\tilde e_r$ for the specialization morphism of $\widetilde\Xc$ we have a commutative diagram
\begin{equation}\label{spdiag}
\xymatrix{
 \widetilde\Xc(\Dr)^\sharp \ar[r]^{\tilde e_r} \ar[d]^{p} & I_{\mu_r} \widetilde\Xc(k)  \ar[d]^{I_{\mu_r}p} \\
\Xc(\Dr)^\sharp \ar[r]^{e_r} & I_{\mu_r} \Xc(k).
}
\end{equation}
Our goal is to describe $e_r^{-1}(x,\phi)$ by means of $\tilde{e}_r$ and $p$. For this recall first from \ref{defi:r-points} that 
\begin{equation}\label{galid}
    \Xc(\Dr)^\sharp  = \left(U(\Oc_L)^\sharp/G(\Oc_L)\right)^\Gamma,\end{equation}
where $L/F$ is a Galois-closure of $F(\varpi^{1/r})=F[\lambda]/(\lambda^r-\varpi)$ and $\Gamma = \Gal(L/F)$. We write $\Oc_r$ for the ring of integers of $F(\varpi^{1/r})$.

The group $\mu_r(L)$ can be identified with the subgroup of $\Gamma$ corresponding to the subfield $F(\mu_r(L)) \subset L$. With this in mind we find:

\begin{lemma}\label{efib}
Under the identification \eqref{galid} we have
\[e_r^{-1}(x,\phi) = U_{x}^{\tilde{\phi}}(\Oc_r)^\sharp / G_{x}^{\tilde{\phi}}(\Oc_r),\]
where
\begin{align*}  
U_{x}^{\tilde{\phi}}(\Oc_r)^\sharp &=  \{ u \in U(\Oc_r)^\sharp \ |\ \forall \zeta \in \mu_r(L)\colon u^\zeta = u \tilde{\phi}(\zeta)   \text{ and } u_{|k} = \tilde{x}\}, \\
    G_{ x}^{\tilde{\phi}}(\Oc_r) &= \{g \in G(\Oc_r)\ | \ \forall \zeta \in \mu_r(L)\colon g^{\zeta} = \tilde{\phi}(\zeta)^{-1}g\tilde{\phi}(\zeta) \  \text{ and } g_{|k} \in \stab(\tilde{x})\}.\end{align*}
\end{lemma}
\begin{proof} Let $u \in U_{x}^{\tilde{\phi}}(\Oc_r)^\sharp$. Then the condition $u^\zeta = u \tilde{\phi}(\zeta)$ for all $\zeta \in \mu_r(L)$ implies, that the composition $\Spec(\Oc_r) \to U \to \Xc$ is $\mu_r$-invariant and that the induced morphism between automorphism groups of closed points is given by $\phi$. Thus, we have a map
\begin{align*}U_{x}^{\tilde{\phi}}(\Oc_r)^\sharp &\to e_r^{-1}(x,\phi).
\end{align*}
Conversely, given $[u] \in e^{-1}(x,\phi) \subset (U(\Oc_L)/G(\Oc_L))^\Gamma$, we may pick a representative $u \in U(\Oc_r)$, which defines a 1-cocycle $\psi: \mu_r(L) \to G(\Oc_L)$ by the requirement
$u^\zeta = u \psi(\zeta)$ for all $\zeta \in \mu_r(L)$. Next we use that we have a bijection in group cohomology
\[ H^1(\mu_r(L), G(\Oc_L)) = \Hom(\mu_r(L),G(\Oc_{F(\mu_r(L))}) )/ G(\Oc_{F(\mu_r(L))}). \]

Indeed, by \cite[Proposition I.36]{Se07} we have an exact sequence of pointed sets

\[1 \to G(\Oc_{F(\mu_r(L))}) \to G(F(\mu_r(L))) \to  \left(G(L)/G(\Oc_L) \right)^{\mu_r(L)} \to H^1(\mu_r(L), G(\Oc_L)) \to 1,\]
and by Theorem \ref{mrdes} we can identify $\left(G(L)/G(\Oc_L) \right)^{\mu_r(L)}$ with $B(G_{F(\mu_r(L))})_{\frac{1}{r}\BZ}$. Taking the quotient by $G(F(\mu_r(L)))$ identifies this subset with the Weyl-group orbits of the $\frac{1}{r}\BZ$-points of a fixed apartment, which are in bijection with $\Hom(\mu_r(L),G(\Oc_{F(\mu_r(L))}) )/ G(\Oc_{F(\mu_r(L))})$.

Thus, we may assume that $\psi$ is an actual group homomorphism. Its restriction to the special fibre is $G(k)$-conjugate to $\phi$ by assumption, which by \cite[XI.5.2]{SGA3II} implies that $\psi$ and $\tilde\phi$ are $G(\Oc_{F(\mu_r(L))})$-conjugate. So we may assume that $\psi = \tilde \phi$, which shows surjectivity of $U_{x}^{\tilde{\phi}}(\Oc_r)^\sharp \to e_r^{-1}(x,\phi)$. 
The identification of the fibre with $ G_x^{\tilde{\phi}}(\Oc_r)$ is a direct computation using the fact that $G(\Oc_L)$ acts freely on $U(\Oc_L)^\sharp$. 
\end{proof}

By construction the measure $\mu_{orb}$ on $ \Xc(\Dr)^\sharp $ is given by integrating local forms $|\omega_i|$ on an open covering of the schematic locus $W(F)$. Pulling back the $\omega_i$ along the $G$-torsor $U \to \Xc$ and wedging with a translation-invariant form $\omega_G$ on $G$ we obtain $G$-invariant forms $\tilde{\omega}_i$ locally on $U$. In particular integrating locally against the $\tilde{\omega}_i$ and $\omega_G$ defines gives the orbifold measures on $U/\mu_r$ and $G/\mu_r$ respectively. Here we consider the action of $\mu_r$ on $G$ by conjugation via $\tilde \phi$. 

By Fubini and Lemma \ref{efib} we thus have
\[ \int_{U_{ x}^{\tilde{\phi}}(\Oc_r)^\sharp } \mu_{orb} = \int_{e_r^{-1}(x,\phi) }\mu_{orb} \int_{G_{x}^{\tilde{\phi}}(\Oc_r)} \mu_{orb}. \]

Finally we notice that we can compute the volume of $U_{x}^{\tilde{\phi}} $ using Lemma \ref{locc}, since as in Lemma \ref{efib} there is an identification $U_{x}^{\tilde{\phi}}(\Oc_r)^\sharp /\mu_r(F)  = \tilde{e}^{-1}_r (\tilde{x},\id)$. Since the action of $\mu_r(F)$ is free we get 
\[ \int_{U_{ x}^{\tilde{\phi}}(\Oc_r)^\sharp } \mu_{orb} = q^{-w(T_{\tilde x}U)}.  \] 
The same applies to $G_{x}^{\tilde{\phi}}(\Oc_r)$, which specializes by definition onto $\stab(\phi)(k) \cap \stab(\tilde{x})(k)$. Thus, we finally obtain
 \[\mu_{orb}(e_r^{-1}(x,\phi)) =  \frac{ q^{-w(x,\phi)} }{|\stab(\phi)(k) \cap \stab(\tilde{x})(k)|}.\] 

This finishes the proof of Theorem \ref{orbiorbi} since $|\Aut(x,\phi)(k)| = |\stab(\phi)(k) \cap \stab(\tilde{x})(k)|$, see for example \cite[Lemma 2.10]{GWZ20b}.

\subsection{Integration on the base}\label{iob}

Let $\Xc$ be a smooth linear quotient stack over $\Oc_F$ and $\Xc \to X$ as in Situation \ref{sit:good}.

We furthermore assume:
\begin{enumerate}
\item $\Xc \to \Spec(\Oc_F)$ is $S$-complete.
\item The map 
\[ \pi_{\infty}\colon \Xc(\Dinfty)^\sharp \to X(\Oc_F)^\sharp\]
is surjective.
\item The canonical bundle $K_{\Xc}$ is Zariski-locally trivial.     
\end{enumerate}

Our goal is to express integrals of certain functions $f:X(\Oc_F)^\sharp \to \BC$ in terms of the twisted inertia stack $I_{\hat \mu}\Xc$. We note that the maps $e_r\colon \Xc(\Dr) \to I_{\mu_r} \Xc(k)$ fit together to a map $e\colon \Xc(\Dinfty) \to I_{\hat \mu} \Xc(k).$

\begin{definition} A function $f:X(\Oc_F)^\sharp \to \BC$ is admissible, if it is integrable with respect to $\mu_{orb}$ and the composition $f  \circ \pi_{\infty}$ factors through the specialization map $e: \Xc(\Dinfty)^\sharp \to I_{\hat \mu}\Xc(k)$ i.e. if there exists a function $\overline{f}: I_{\hat \mu}\Xc(k) \to \BC$ such that $f  \circ \pi_{\infty} = \overline{f} \circ e$.
\end{definition}

\begin{theorem}\label{inbase} Let $f:X(\Oc_F)^\sharp \to \BC$ be an admissible function and $\overline{f}:I_{\hat \mu}\Xc(k) \to \BC$ the induced function. Then for any $x \in X(k)$ the generating series $\sum_{r\geq 1} \sum_{y \in I_{\mu_r}\Xc_x(k)}\overline{f}(y)\frac{q^{-w(y)}}{|\Aut(y)(k)|}T^r$ is a rational function $Q_{x,f}(T)$ in $T$ of degree $0$ and we have
\[ \int_{B(x)}f \mu_{orb} = -\lim_{T\to \infty}Q_{x,f}(T),  \]
    where $B(x) = \{ \tilde x \in X(\Oc_F)^\sharp \ |\ \tilde{x}_{|k} = x\}$ and $\Xc_x \subset \Xc$ is the fibre over $x$. In particular,
    \[\int_{X(\Oc_F)^\sharp} f \mu_{orb} = -\lim_{t\to \infty}Q_f(T),\]
    where $Q_f(T) = \sum_{x \in X(k)} Q_{x,T}(T)$.
    
\end{theorem}

\begin{rmk}\label{gerf} The main example of  admissible functions we have in mind are functions coming from $\BG_m$-gerbes on $\Xc$. Given such a gerbe $\alpha \in H^2(\Xc,\BG_m)$ and a point $x\in  X(\Oc_F)^\sharp$ we may consider the generic point $x_{|F}$ as an element of $\Xc(F)$ and then the pullback $x_{|F}^*\alpha$ can be identified with an element in $\BQ/\BZ$ via the isomorphism $H^2(F,\BG_m) \cong \BQ/\BZ$ given be the Hasse invariant. Then the composition with $\pi_\infty$ of the function given by $f_\alpha(x) = e^{2 \pi i x_{|F}^*\alpha}$ factors through the specialization map $e$ by\cite[Lemma 3.10]{GWZ20b}.
\end{rmk}
 
\begin{proof} Consider the projections $\pi_r: \Xc(\Dr)^\sharp \to X(\Oc_F)^\sharp$ and the generating series
\begin{equation}\label{genser} E(T) = \sum_{r \geq 1} \int_{\pi_r^{-1}B(x)} f \circ \pi_r \mu_{orb} T^r= \sum_{r\geq 1} \sum_{y \in I_{\mu_r}\Xc_x(k)}\overline{f}(y)\frac{q^{-w(y)}}{|\Aut(y)(k)|}T^r,\end{equation}
where we used Theorem \ref{orbiorbi} for the equality. 

By Lemma \ref{locco} below there is an open cover $B(x) = \sqcup_{i \geq 1} U_i$ and rational functions $J_i(T) = \frac{g_i(T)}{(1-T^N)^{d+1}} \in \BZ(T)$ of degree $0$ with $\lim_{T \to \infty} J_i(T) = -1$ such that for all $i$ and $z \in U_i$ we have
\[ \sum_{r \geq 1}  |\pi_r^{-1}(z)| T^r = J_i(T).\] 
Since the projections $\pi_r$ are étale by Proposition \ref{prop:r-finite} we find for every $i$ the equality
\[  \sum_{r \geq 1} \left( \int_{\pi_r^{-1}U_i} f \circ \pi_r \mu_{orb} \right) T^r  =  J_i(T) \int_{U_i} f \mu_{orb}.\]

Now, the partial sums $P_k(T)= \sum_{i=1}^k J_i(T)\int_{U_i}f \mu_{orb}$ converge coefficient by coefficient to $E(T)$. On the other hand, $\left(\int_{U_i}f \mu_{orb}\right)J_i(T)= \left(\int_{U_i}f \mu_{orb}\right)\frac{g_i(T)}{(1-T^{N})^{d+1}}$, and thus the $P_k(T)$ also converge coefficient by coefficient to a rational function of degree $0$. Thus, $E(T)$ is indeed rational and taking $\lim_{T \to \infty}$ on both sides of \eqref{genser} gives the desired formula.
\end{proof}

\begin{lemma}\label{locco} There exist an integer $N\geq 1$, a countable open cover $X(\Oc_F)^\sharp = \sqcup_i U_i$ and rational functions $J_i(T) = \frac{g_i(T)}{(1-T^N)^{d+1}} \in \BZ(T)$ of degree $0$ with $\lim_{T \to \infty} J_i(T) = -1$ such that for all $i$ and $x \in U_i$ we have
\[ \sum_{r \geq 1}  |\pi_r^{-1}(x)| T^r = J_i(T).\] 
\end{lemma}
\begin{proof}
    We choose a presentation $\CX=[U/G]$ with $G=\GL_n$ for some $n$, so that we may apply Construction \ref{EmbCons} to all points of $\CX(F)$.

A theorem of Stanley \cite{St80} states that for any rational convex polytope $P \subset \BR^d$, the series

\[J_P(T) = \sum_{r \geq 1} |\frac{1}{r}\BZ^d \cap P|T^r, \]
is a rational function of the form $\frac{g(T)}{(1-T^\delta)^{d+1}}$, where $\delta \geq 1$ is an integer such that $\delta P$ is an integral polytope. Furthermore $J_P(T)$ is of degree $0$ and $\lim_{T \to \infty} J_P(T) = -1$. 

Now, given $x \in X(\Oc_F)^\sharp$, and any lift $\tilde x$ of $x|_F$ to $U$, it follows from Theorem \ref{conhul} and an inclusion-exclusion argument that the same holds for the series
\[J_x(T) = \sum_{r \geq 1} |B(G_F)_{\frac{1}{r}\BZ^d} \cap b_{\tilde x}(\Dinfty)|T^r= \sum_{r \geq 1} |\pi_r^{-1}(x)| T^r.\]
Furthermore the denominator of $J_x(T)$ may be chosen to be $(1-T^N)^{d+1}$, with $N$ independent of $x$ as in Theorem \ref{conhul}. Thus, $J_x(T)$ is determined by a finite number of its coefficients, and that number is independent of $x$. Since for every $r \geq 1$ the function $x \mapsto |\pi_r^{-1}(x)|$ is locally constant by Proposition \ref{prop:r-finite} and $ X(\Oc_F)^\sharp$ is second-countable and totally disconnected we obtain the claim.
\end{proof}

We give two examples on how to compute the formulas in Theorem \ref{inbase}.

\begin{example}\label{DMcase} Let $\Xc$ be a smooth and tame Deligne-Mumford stack over $\Oc_F$ (Zariski-locally) of the form $\Xc \cong [U/\Gamma]$ with $\Gamma$ a finite étale group scheme. Then we can apply Theorem \ref{inbase} to the coarse moduli space map $\Xc \to X$. Indeed conditions (i) and (iii) are automatic and (ii) follows since any $X(\Oc_F)^\sharp$ lifts after a tame extension and up to an unramified extension they are all obtained by adjoining a root of a uniformiser. 

Then there are only a finite numbers of isomorphism classes $y=(x,\phi) \in I_{\hat{\mu}}\Xc(k)$, each of which contributing $\frac{q^{-w(y)}}{|\Aut(y)(k)|}$ to the coefficient of $T^r$, whenever the order of $\phi$ divides $r$. Theorem \ref{inbase} thus gives
\[ \int_{X(\Oc_F)^\sharp}\mu_{orb} =  -\lim_{T\to \infty}  \sum_{y = (x,\phi) \in I_{\mu_r}\Xc(k)} \frac{q^{-w(y)}}{|\Aut(y)(k)|}\sum_{r' \geq 1} T^{\ord(\phi)r'} = \sum_{y  \in I_{\mu_r}\Xc(k)} \frac{q^{-w(y)}}{|\Aut(y)(k)|}, \]
in agreement with the orbifold formula for DM-stacks as for example in \cite[Theorem 2.21]{GWZ20b}.
\end{example}

\begin{example} Consider $\Xc = [\BA^2/\BG_m]$ where $\BG_m$ acts on $\BA^2$ with weights $(1,-1)$. Then $\pi:\Xc \to X = \BA^1$ defined by $\pi(x,y) = xy$ satisfies all conditions (i)-(iii). The fibre $\Xc_0(k)$ consists of three classes corresponding to the orbits $(1,0), (0,1)$ and $(0,0)$ and only the last one has non-trivial automorphism group $\BG_m$. For any $\phi:\mu_r \to \BG_m$ one computes $w(0,\phi) = 1$ if $\phi \equiv 1$ and $w(0,\phi)  = 0$ if $\phi$ is non-trivial. We thus get 
\[ \int_{B(0)}  \mu_{orb} = 2q^{-1}  -\lim_{T\to \infty}\sum_{r\geq 1}  \frac{q^{-1}}{q-1} + (r-1) \frac{1}{q-1}T^r = 2q^{-1} +\frac{q^{-1}}{q-1} - \frac{1}{q-1} = q^{-1}, \]
in agreement with the standard measure of a $1$-dimensional ball around $0$. 
\end{example}

\section{Plethystic identities for $k$-linear stacks}\label{pikls}

\subsection{$\lambda$-rings of counting functions}

We follow essentially \cite{Mo19}. Let $k \cong \BF_q$ be a fixed finite field and write $k_n$ for the degree $n$ extension of $k$ in a fixed algebraic closure $\bar{k}$. Let $\sigma$ be a generator of the absolute Galois group of $k$. In the following we consider a set $S$ endowed with a continuous $\widehat{\mathbb{Z}} = \Gal(\bar{k}/k)$-action. We denote for any subset $X \subset S$ as above the fixed-point set $X^{\sigma^n}$ by $X_n$.
We call $S$ as above $\sigma$-finite if the set $S_n = S^{\sigma^n}$ is finite for every integer $n \geq 1$.

For example, if $\BM /k$ is an Artin stack over $k$, we have an action of $\widehat{\mathbb{Z}}$ on the set of isomorphism classes $\BM(\bar{k}).$\footnote{By abuse of notation we will write $\BM(T)$, where $T/k$ is any scheme, both for the groupoid and the set of isomorphism classes in the groupoid.} 
As above, we denote by $\BM(\bar{k})_n$ the set of $\sigma^n$-fixpoints in $\BM(\bar{k})^{\iso}$. 
Note that there is a surjective map:
$$ \BM(k_n) \to \BM(\bar{k})_n =\BM(\bar{k})^{\sigma^n}.$$
\begin{rmk}\label{rmk:bij}
For moduli stacks of objects in abelian categories, bijectivity of this map follows from \cite[Lemmas 3.4 \& 3.5]{Mo19}.    
\end{rmk}

In order to define a suitable ring of functions on a $\widehat{\mathbb{Z}}$-set $S$, recall the volume ring introduced in \cite[Section 2.2]{Mo19} 
\[ \CV = \prod_{n\geq 1} \overline{\BQ}.\]
We equip $\CV$ with Adams operations $\psi_m:\CV \to \CV$ by means of the formula $\psi_m(v_n)_{n\geq 1} = (v_{mn})_{n\geq 1}$. By \cite[Remark 2.1]{Mo19} this defines a $\lambda$-ring structure on $\CV$.

\begin{definition} For $S$ as above the ring of counting functions $\mcf(S)$ on $S$ is the set of maps 
\[ \mcf(S) = \{f: S \to \CV \mid \text{For all } n\geq 1\text{ and } x \in S: f(x)_n=0 \text{ unless } x^{\sigma^n}=x\},  \]
equipped with component-wise addition $+$ and multiplication $*$. For any equivariant morphism of $\phi: S_1 \to S_2$ of $\widehat{\mathbb{Z}}$-sets we get a pullback 
\[ \phi^*:   \mcf(S_2) \rightarrow  \mcf(S_1), \ \ \ \ f \mapsto f \circ \phi.\]
If furthermore $\phi$ has $\sigma$-finite fibres we also have a proper pushforward 
\[ \phi_!:  \mcf(S_1) \rightarrow  \mcf(S_2), \ \ \ \ \phi_!(f)(x) = \sum_{y \in \phi^{-1}(x)} f(y). \]
Both $\phi^*$ and $\phi_!$ are group homomorphism with respect to $+$, but only $\phi^*$ is a ring homomorphism for $*$.
\end{definition}

The functions we are interested in will usually be motivic in the following sense. Let $\BM$ be an Artin stack as above and $K_0(\Var/\BM)$ be the Grothendieck ring of varieties over $\BM$ generated by representable finite type morphisms $\CX \to \BM$ subject to the usual cut-and-paste relations, see \cite[Section 6.1]{MR19} for a precise definition. We denote the ring of counting functions of $\BM(\bar{k})$ by $\mcf(\BM)$. There is a ring homomorphism
\[K_0(\Var/\BM) \rightarrow \mcf(\BM)\]
defined as follows. Given $\pi: \CX \to \BM$ we may consider the function 
\begin{align*} |\CX|: \BM(\bar{k}) &\rightarrow \CV \\
  x &\mapsto (|\pi^{-1}(x)(\bar{k})_n|)_{n\geq 1}. 
\end{align*}
Notice that by definition $\pi^{-1}(x)$ is a finite type scheme and $|\pi^{-1}(x)(\bar{k})_n|$ denotes the cardinality of $|(\pi^{-1}(x)(\bar{k}))^{\sigma^n}|$ if $x$ is fixed by $\sigma^n$ and $0$ else. We write $\BL \in \mcf(\BM)$ for $|\BA^1\times \BM|$.

\begin{rmk} The reason we work with $\mcf(\BM)$ instead of $K_0(\Var/\BM)$ comes from the simple observation that, by definition, counting functions on $\BM$ agree if they agree on all $\bar{k}$-points of $\BM$ while elements $K_0(\Var/\BM)$ only agree if they agree after pullback to every schematic point \cite[Lemma 1.1.8]{CL16}. Indeed, an interesting class of examples of counting functions stems from lisse $\ell$-adic sheaves $\CL$ on an Artin stack $\BM/k$. To $\CL$ we associate the function 
$$f_{\CL}\in \mcf(\BM)$$
which assigns to $x \in \BM(\bar{k})_n$ the trace of $\sigma^n \in \Gal(\bar{k}/k)=\pi_1^{\text{\'et}}(\Spec k)$ acting on the stalk of $\CL_x$. However, the ring of counting functions also contains many examples, which do not arise in this way. These exotic counting functions are needed in the present work as the roots $\BL^{\frac{1}{2}}$ of $\BL$ introduced in Subsection \ref{pls} are of this type.
\end{rmk}

In order to define a $\lambda$-ring structure on $\mcf(S)$ we assume that $S$ is a commutative torsion-free monoid, i.e., that there is a $\widehat{\Zb}$-equivariant morphism
\[\oplus: S \times S \rightarrow S   \]
satisfying the natural compatibiliies, as well as a zero element $0 \in S^{\sigma}$. We also assume that $\oplus$ has $\sigma$-finite fibres. All monoids considered below will be assumed to satisfy these assumptions (in particular, torsion freeness). In this case, there is an other multiplication on $\mcf(S)$ given by convolution \cite[Section 4]{Me17b}, i.e., for $f,g \in \mcf(S)$ we define
\[ (f\cdot g)(x)_n = (fg)(x)_n = \sum_{(x',x'') \in \oplus^{-1}(x)(\bar{k})_n} f(x')_ng(x'')_n. \]
From now on we will consider $\mcf(S)$ as a ring with respect to the convolution product with unit given by the characteristic function of $0$. 

\begin{definition}
For positive integers $m | n$ we define the trace map 
\begin{equation*}
    \Tr_{n/m} \colon S_n \to S_m,\; x \mapsto \bigoplus_{i=1}^{n/m} x^{\sigma^{im}} 
\end{equation*}

The Adams operations $(\psi_m)_{  m \geq 1}$ on  $\mcf(S)$ are defined by 
\[ \psi_m(f)(x)_n =
\begin{cases}
    \sum_{y \in S_{nm}\colon \Tr_{mn/n}(y)=x} f(y)_{nm} & \text{ if } x \in S_n \\
    0 & \text{else}
\end{cases}   \]
\end{definition}

\begin{lemma}
There is a unique $\lambda$-ring structure on $\mcf(S)$ with these Adams operations.
\end{lemma}
\begin{proof}
    First we check that the $\psi_m$ are ring homomorphisms (with respect to the convolution product): 
    For $x \in S_n$ we find
    \begin{equation*}
        \psi_m(fg)(x)_n=\sum_{y \in \Tr_{mn/n}^{-1}(x)} (fg)(y)_{mn} =\sum_{y \in \Tr_{m/n}^{-1}(x)} \sum_{u,v \in S_{mn}\colon u+v=y } f(u)_{mn} g(v)_{mn}
    \end{equation*}
    and
    \begin{align*}
        (\psi_m(f)\psi_m(g))(x)_n &=\sum_{u', v'\in S_n\colon u'+v'=x}\psi_m(f)(u')_n\psi_m(g)(v')_n\\ &=\sum_{u', v'\in S_n\colon u'+v'=x} \sum_{u \in \Tr_{mn/n}^{-1}(u'), v \in \Tr_{mn/n}^{-1}(v') }f(u)_{mn} g(v)_{mn}.
    \end{align*}
    These sums are equal since they are both indexed by the elements $u, v \in S_{mn}$ satisfying $\Tr_{mn/n}(u+v)=x$.

    Hence by \cite[Section 2.1]{Mo19} the $\psi_m$ define a $\lambda$-ring structure on $\mcf(S)$ if in addition $\psi_1=\id$ and $\psi_m\circ\psi_{m'}=\psi_{m m'}$ for all $m,m' \geq 1$. The first statement follows from $\Tr_{n/n}=\id$, and to check the second we compute as follows:
    \begin{align*}
        \psi_m(\psi_{m'})(f)(x)_n&=\sum_{y \in \Tr_{mn/n}^{-1}(x)} \psi_{m'}(f)(y)_{nm} = \sum_{y \in \Tr_{mn/n}^{-1}(x)} \sum_{y'\in \Tr_{mm'n/mn}^{-1}(y')} f(y')_{nmm'}\\
        &= \sum_{y' \in \Tr_{mm'n/n}^{-1}(x)}f(y')_{nmm'}=\psi_{mm'}(f)(x)_n
    \end{align*}
    
\end{proof}
Next, let $\mcf(S)_0 \subset \mcf(S)$ be the ideal of functions that vanish at $0$ and $1+ \mcf(S)_0 \subset  \mcf(S)$ the multiplicative subgroup of functions with value $1=(1)_{n\geq 1 }$ at $0$. The usual formulae for exponential and logarithm define mutually inverse group homomorphisms
 \begin{align*}\exp&: \mcf(S)_0 \rightarrow 1 + \mcf(S)_0, \ \ \ \  f \mapsto \sum_{n \geq 0} \frac{f^n}{n!}.\\
\log&: 1+ \mcf(S)_0 \rightarrow \mcf(S)_0, \ \ \ \ 1+ f \mapsto \sum_{n\geq 1} (-1)^{n-1} \frac{f^n}{n}. \end{align*}
Notice that for any $x \in S$  only finitely many terms in $ \sum_{n \geq 0} \frac{f^n}{n!}$ and $\sum_{n\geq 1} (-1)^{n-1} \frac{f^n}{n}$ contribute to $\exp(f)(x)$ and $\log(1+f)(x)$ respectively, since $f^n$ is computed with respect to the convolution product. 

Similarly we have the plethystic exponential and logarithm
\begin{align*} \Sym&: \mcf(S)_0 \rightarrow 1 + \mcf(S)_0,  \ \ \ \ f \mapsto \exp \left(\sum_{n \geq 1} \frac{\psi_n(f)}{n}  \right)\\
\Log&: 1 + \mcf(S)_0 \rightarrow \mcf(S)_0, \ \ \ \  1+ f \mapsto \sum_{n \geq 1} \frac{\mu(n)}{n} \psi_n(\log(1+f)),
  \end{align*}
where $\mu$ denotes the Möbius function on $\BN$.

\begin{example}[{\cite[Example 4.2]{Me17b}}]\label{dimon} Consider the discrete monoid $\BM = \BN^r$ endowed with the trivial Galois action. Then, $ \mcf(\BM)$ can be identified with $\CV[[X_1,\dots,X_r]]$. Under this isomorphism, the pointwise product $*$ corresponds to the Hadamard product and the convolution product $\cdot$ to the usual product of power series. The $\lambda$-ring structure on  $ \mcf(\BM)$  induced by the $\psi_n$ agrees with the usual one on $\CV[[X_1,\dots,X_r]]$.
\end{example}

\begin{lemma}\label{mfunc} Let $\phi: S_1 \to S_2$ be a morphism of monoids.
\begin{enumerate}
 \item If $\phi$ has $\sigma$-finite fibres, then $\phi_!$ is a homomorphism of $\lambda$-rings. 
\item\label{disrec}Assume $\phi$ is injective and its image $\im(\phi) \subset S_2$ is a full submonoid i.e. if $x \in \im(\phi)$ and $x=x'+x''$ for objects $x', x'' \in S_2$, then $x',x'' \in \im(\phi) $. Then, $\phi^*$ is a homomorphism of $\lambda$-rings. 
\end{enumerate}

 In particular. in each of these cases  $\phi_!$ resp. $\phi^*$ commutes with the usual and plethystic exponential and logarithm on $\mcf(S_1)$ and $ \mcf(S_2)$.
\end{lemma}
\begin{proof} Assume first that $\phi$ has $\sigma$-finite fibres. Then $\phi_!$ respects the convolution product: For $f,g \in \mcf(S_1)$ and $x \in S_2$ we find
\begin{align*} \phi_!(fg)(x) = \sum_{y \in \phi^{-1}(x)(\bar{k})} \sum_{y'+y'' = y} f(y')g(y'') &= \sum_{x' +x'' = x}\left(\sum_{y' \in \phi^{-1}(x')(\bar{k})} f(y') \right)\left(\sum_{y'' \in \phi^{-1}(x'')(\bar{k})} f(y'') \right) \\ &= \phi_!(f)\phi_!(g) (x).
\end{align*}
Next, let $f \in \mcf(S_1)$, let $x \in S_2$ and $n\geq 1$. If there is no $x' \in  S_2$ such that $x = n x'$, then there doesn't exist any $y \in \phi^{-1}(x)(\bar{k})$ such that $y = n y'$ and thus $\phi_!(\psi_n(f))(x) = 0 = \psi_n(\phi_{!}(f))(x)$. If $x= n x'$, then
\begin{align*}\phi_!(\psi_n(f))(x)  &=  \sum_{y \in \phi^{-1}(x)} \psi_n(f)(y) = \sum_{y' \in \phi^{-1}(x')} \psi_n(f(y')) \\    
&= \psi_n \left(\sum_{y' \in \phi^{-1}(x')} f(y')   \right) = \psi_n \left( \phi_!(f)  \right)(x). \end{align*}

This shows that $\phi_!$ is a $\lambda$-ring homomorphism.

Now, assume $\phi$ is as in \ref{disrec}. Then we have for $f,g \in \mcf(S_2)$ and $y \in S_1$

\begin{align*} \phi^*(fg)(y) = fg(\phi(y)) = \sum_{x'+x''=\phi(y)} f(x')g(x'') = \sum_{y'+y'' = y} f(\phi(y'))g(\phi(y'')) = \phi^*(f)\phi^*(g)(y).
\end{align*}
Finally, let $f \in  \mcf(S_2)$, $y \in S_1$ and $n\geq 1$. As above we argue that if there is no $y' \in  S_1$ such that $y = n y'$, then the same holds for $\phi(y)$ by assumption, and thus $\phi^* (\psi_n(f))(y) = 0 = \psi_n(\phi^*(f))(y)$. If $y = n y'$ we have 
\[ \phi^* (\psi_n(f))(y) = \psi_n(f)(\phi(y)) = \psi_n(f(\phi(y'))) = \psi_n(\phi^*(f))(y).  \]

 The final statement follows since (plethystic) exponential and logarithm are defined using only $\lambda$-ring operations.
\end{proof}

For later use let us record the following explicit formula:

\begin{proposition}
For any $f \in  \mcf(S)_0$ and $n\geq 1$ the following holds:
    \begin{equation}\label{logeq}
        \Log(1+f)(x)_n = \sum_{m,s \geq 1} \frac{(-1)^{s-1} \mu(m)}{ms}\sum_{(y_1,\hdots,y_s) \in S_{nm}^s\colon \; x=\sum_{i=1}^s \Tr_{nm/n}(y_i)} f(y_1)_{nm} \cdots f(y_s)_{nm}
    \end{equation}
\end{proposition}

\begin{proof}
\begin{align*}
\begin{split}
    \Log(1+f)(x)_n &= \sum_{m \geq 1} \frac{\mu(m)}{m} \psi_m(\log(1+f))(x)_n \\
    &= \sum_{m \geq 1} \frac{\mu(m)}{m} \sum_{y \in \Tr^{-1}(x)}\log(1+f)(y)_{nm} \\
    &= \sum_{m \geq 1} \frac{\mu(m)}{m} \sum_{y \in \Tr^{-1}(x)} \sum_{s\geq 1} \frac{(-1)^{s-1}}{s} f^s(y)_{nm} \\
    &= \sum_{m \geq 1} \frac{\mu(m)}{m} \sum_{y \in \Tr^{-1}(x)} \sum_{s\geq 1} \frac{(-1)^{s-1}}{s} \sum_{(y_1,\hdots,y_s) \in S_{nm}^s: y=y_1+\hdots+y_s} f(y_1)_{nm}\cdots f(y_s)_{nm} \\
    &= \sum_{m,s \geq 1} \frac{(-1)^{s-1} \mu(m)}{ms}\sum_{(y_1,\hdots,y_s) \in S_{nm}^s\colon \; x=\sum_{i=1}^s \Tr_{nm/n}(y_i)} f(y_1)_{nm} \cdots f(y_s)_{nm}
    \end{split}
\end{align*}
\end{proof}

As an immediate consequence we have 
\begin{corollary} Let $g \in \mcf(S)_0$ be a counting function satisfying the following two conditions for all $n,m \geq 1$, all $x_1,x_2 \in S_n$ and all $y \in S_{nm}$:
\begin{enumerate}
    \item $g(x_1+x_2)=g(x_1)g(x_2)$,
    \item $g(\Tr_{mn/n}(y))_{n}=g(y)_{nm}$.
\end{enumerate}
Then, \eqref{logeq} implies
\begin{equation}
    \Log(1+fg)=g\Log(1+f)
\end{equation}
for any $f \in \mcf(S)_0$.
\end{corollary}

\subsection{$k$-linear stacks}\label{axms}

Following an idea of Artin--Zhang \cite{AZ01}, moduli stacks for finitely presented objects in an $R$-linear abelian category are studied by Alper--Halpern-Leistner--Heinloth in \cite[Section 7]{AHH23}. The abelian category is assumed to be co-complete and locally noetherian. No restrictions other than commutativity and unitality are imposed on the base ring $R$ (which is denoted $k$ in \emph{loc. cit.}). We fix such a co-complete category $\Bc$, and denote the stack of finitely presented objects in $\Bc$ by $\BM_{\Bc}$.

The definition of the stack $\BM_{\Bc}$ hinges on the insight of Artin--Zhang that in the context above, there is a natural (exterior) tensor product 
$$\otimes_R\colon \Mod(R) \times \Bc \to \Bc,$$
which allows one to define the notion of an $S$-family of objects in $\Bc$ for every commutative $R$-algebra $S$. We denote by $\Bc_S$ the $S$-linear abelian category obtained by this base-change procedure from $\Bc$. For an object $X\in \Bc$ we write $X_S$ for $X \otimes_R S$. Faithfully flat descent theory is satisfied in this general setting, which implies the statement that $\BM_{\Bc}$ is a stack. 

Henceforth, we assume that our base ring $R$ is a field $k$ with algebraic closure $\bar{k}$.
We also fix a full sub-category $\overline{\Cc} \subset \Bc^{\rm fp}_{\bar{k}}$, which we assume to be an extension-closed abelian sub-category of the category of finitely presented objects in $\Bc_{\bar{k}}$ defined over $\bar{k}$. 

We then obtain an abelian subcategory $\Cc \subset \Bc^{\rm fp}$, which we define to be the full sub-category of objects $X$ in $\Bc^{\rm fp}$ satisfying $X_{\bar{k}} \in \overline{\Cc}$.

Furthermore, it gives rise to a sub-stack $\BM = \BM_{\Cc} \hookrightarrow \BM_{\Bc}$. Here, an $S$-family of objects $X_S$ in $\Bc^{\rm fp}_S$ is assumed to belong to $\BM_{\Cc}(S)$ if and only if for every $z\colon \Spec \bar{k} \to S$ the base change $z^*X_S \in \overline{\Cc}$.

In applications, the objects in $\Cc$ are of principal interest and $\Bc$ plays an auxiliary role, as for instance, higher Ext-groups $\Ext^i_{\Bc}(X,Y)$ for two objects in $\Cc$ are defined with respect to the ambient category $\Bc$. This set-up is necessary to ensure that the higher Ext-groups reflect the geometry of interest.

\begin{example}
Let $\Bc$ be the category of sheaves of complex vector spaces on a connected topological manifold $X$. We denote by $\Cc$ the category $\Loc(X)$ of finite-dimensional complex local systems on $X$, that is, locally constant sheaves of finite-dimensional complex vector spaces on $X$. The category $\Cc$ is equivalent to $\Loc(B\pi_1(X))$, that is, the category of local systems on the classifying space of $\pi_1(X)$. In particular, one has $\Ext^i_{\Cc}(\underline{\Cb}_X,\underline{\Cb}_X) = H^i(B\pi_1(X),\Cb)$, whereas $\Ext^i_{\Bc}(\underline{\Cb}_X,\underline{\Cb}_X) = H^i(X,\Cb)$.
\end{example}

The higher Ext-groups $\Ext^i_{\Bc}$ satisfy the following flat base-change relation, which we record for later use.

\begin{lemma}\label{lemma:base-change}
Let $R \to S$ be a flat ring homomorphism and $\Bc$ be a co-complete and locally noetherian $R$-linear abelian category. There is a natural isomorphism $\Ext^i_{\Bc}(X,Y) \otimes_R S \simeq \Ext^i_{\Bc_S}(X_S,Y_S)$ for every pair of objects $X,Y \in \Bc$ and every $i \in \Nb$.
\end{lemma}
\begin{proof}
This is Proposition C3.1(ii) in \cite{AZ01}.
\end{proof}

Following Mozgovoy \cite{Mo19}, we will refer to $\BM=\BM_{\Cc}$ as a $k$-linear stack. Note that in \emph{loc. cit.} this refers to a sheaf of abelian categories on the small \'etale site of a field $k$. The general theory of tensor products and descent theory for $\Bc$ recalled above always yields a sheaf of abelian categories on the big \'etale site, which we restrict to the small \'etale site when citing results from \emph{loc. cit.} 

In addition to the assumptions fixed above we will work with categories satisfying the following:

\begin{enumerate}
    \item There exists a positive integer $\delta$ such that for any finite extension $k'/k$, the category $\Cc_{k'}$ is abelian of finite global dimension $\delta$, i.e., such that $\Ext^i(\;,\;)$ vanishes for $i>\delta$, and with finite-dimensional $\Ext^i$-spaces for $0 \leq i \leq \delta$.
        
\end{enumerate}
This  implies in particular, that $\Cc_{k'}$ is a Krull-Schmidt category. 

We will consider $\BM$ as a monoid with respect to the direct sum $\oplus$ in these abelian categories. The Krull-Schmidt property implies that $\oplus$ has finite fibres. 

Given extensions $k''/k'/k$ and $x \in \BM(k')$ we write $x_{k''}$ for the pullback of $x$ in $\BM(k'')$.

\begin{enumerate}[resume]
    \item\label{smst} The map 
    \begin{align*} (\cdot,\cdot): K_0(\Cc) \times K_0(\Cc) &\to \BZ \\ x,y &\mapsto \sum_{i=0}^\delta  (-1)^i\dim\Ext^i_{\Bc}(x,y)
        \end{align*}
        is symmetric.
\end{enumerate}

Notice that for any $x \in \BM_{\Bc}$ the integer $-(x,x)$ equals the Euler characteristic of the complex $R\Hom(x,x)[1]$, which is the tangent complex $T^{\bullet}_x\BM_{\Bc}$.
If $\Cc$ is hereditary, i.e., if $\delta \leq 1$, Axiom \ref{smst} is the central assumption in \cite{Me15}.

\subsection{$\BG_m$-rigidification and the gerbe function}\label{gmgr}

We continue with the setup from the previous subsection and assume in addition that the field $k$ is finite. For $x \in \BM$, the automorphism group scheme $\Aut(x)$ contains a distinguished copy of $\BG_m$ coming from scalar automorphisms. We write $\uAut(x)$ for the quotient of $\Aut(x)$ by this $\BG_m$ and $\CM$ for the corresponding $\BG_m$-rigidification of $\BM$. So $\BM \to \CM$ is a $\BG_m$-gerbe. By \cite[Def. 3.6]{GWZ20b} this $\BG_m$-gerbe induces a function
\begin{equation*}
    f_{\BM}\colon I_{\hat\mu}\CM(k)\to \BQ/\BZ,
\end{equation*}
whose definition we now recall: 

First, for any $r \geq 1$, there is a canonical identification of flat cohomology $H^2_{\rm fl}(B_k \mu_r,\BG_m) \cong \BZ/r \BZ$. Concretely this can be seen as in \cite[Lemma 3.4]{GWZ20b} as follows: any $\BG_m$-gerbe $\CG \in H^2_{\rm fl}(B_k \mu_r,\BG_m)$ becomes trivial after base change to $\bar k$. The pullback of any splitting $s \in \CG(B_{\bar k} \mu_r)$ under the Frobenius $\phi \in \Gal(\bar k / k)$ gives another splitting $\phi(s)$ which differs from $s$ by a $\BG_m$-torsor $\phi(s)-s \in H^1_{\rm fl}(B_{\bar k} \mu_r,\BG_m)\cong \Hom(\mu_r,\BG_m)(\bar k)\cong \BZ/r \BZ$. Since every class in $H^1_{\rm fl}(B_{\bar k} \mu_r,\BG_m)$ already comes from $H^1_{\rm fl}(B_k \mu_r,\BG_m)$ this element $\phi(s)/s$ is independent of the choice of $s$. This gives us the desired isomorphism
\begin{equation*}
    H^2_{\rm fl}(B_k \mu_r,\BG_m) \cong \BZ/r \BZ, \;\CG \mapsto \phi(s)-s.
\end{equation*}
    
Then any element $y \in I_{\hat\mu}\CM(k)$ can be represented by a morphism $\tilde y\colon B_k \mu_r \to \CM$ for some $r \geq 1$, and its image under $f_\BM$ is given by the element
\begin{equation*}
    f_\BM(y) \defeq \tilde y^*(\BM) \in H^2_{\rm fl}(B_k \mu_r,\BG_m) \cong \BZ/r\BZ \cong (\BQ/\BZ)[r] \subset \BQ/\BZ.
\end{equation*}
The fact that this does not depend on the choice of $r$ follows from the functoriality statement in \cite[Lemma 3.4]{GWZ20b}. We apply the same construction for any finite extension $k'/k$ to the pullback $\BM_{k'}$ to obtain a map $f_{\BM_{k'}}\colon I_{\hat\mu}\CM(k')\to \BQ/\BZ$.

Finally, we define the \textit{gerbe function}
\[ \alpha: I_{\hat{\mu}}\CM(\bar{k}) \rightarrow \prod_{n\geq 1}\BQ/\BZ,\]
as follows: Given $y \in I_{\hat\mu} \CM(\bar k)$ and $n\geq 1$, if $y \in I_{\hat\mu} \CM(k_n)$, then the component $\alpha(y)_n$ is equal to the image of $y$ under the map $$f_{\BM_{k_n}}\colon I_{\hat\mu}\CM(k_n) \to \BQ/\BZ,$$  and otherwise $\alpha(y)_n=0$.

\subsection{Plethystic logarithm of the shifted identity function}\label{pls}

In order to define our counting function of interest we need to introduce a square root $\BL^{\frac{1}{2}}$ of $\BL \in \CV$. We choose $\BL^{\frac{1}{2}}$ such that there exist constants $b_1,b_2 \in \BF_2$ such that for all $m \geq 1$:
\begin{equation*}
    \psi_m(\BL^{\frac{1}{2}})=(-1)^{b_1+b_2m}(\BL^{\frac{1}{2}})^m.
\end{equation*}
The shifted identity function we're interested in is defined as
\begin{align} \label{shid} \frac{\BL^{\frac{(,)}{2}}}{|\Aut(\cdot)|}\colon \BM &\rightarrow \CV \\
x &\mapsto \frac{\BL^{\frac{(x,x)}{2}}}{|\Aut(x)|}, \nonumber   
\end{align}
where $\BM$ is a $k$-linear stack $\BM$ as in Section \ref{axms}.

We will express the plethystic logarithm of $ \frac{\BL^{\frac{(,)}{2}}}{|\Aut(\cdot)|}$ in terms of $I_{\hat{\mu}}\CM$ and two functions on this space. The first is the weight function 
\begin{align*}\Bw\colon I_{\hat{\mu}}\CM(\bar{k}) &\rightarrow \BQ \\
(x,\phi) &\mapsto \sum_{i=0}^\delta (-1)^{i+1} w(\Ext^i(x,x)),
\end{align*}
where we consider $\Ext^i(x,x)$ as a $\hat{\mu}$-representation via $\phi$ and its weight $w(\Ext^i(x,x))$ is defined as in Definition \ref{weightf}(i). Notice that in the case when $\delta \leq 1$, $\CM$ is smooth and the weight function $w$ on $I_{\hat{\mu}}\CM$ defined in \ref{weightf}(ii) is related to $\Bw$ by $w = \Bw +1$, as $w$ takes the rigidification into account.

The second function is a slight modification of the gerbe function
\[ \alpha\colon  I_{\hat{\mu}}\CM(\bar{k}) \rightarrow \prod_{n\geq 1}\BQ/\BZ,\]
from Section \ref{gmgr}. First, associated with $\alpha$ is the order function
\[o_\alpha\colon  I_{\hat{\mu}}\CM(\bar{k}) \rightarrow \prod_{n\geq 1}\BN, \]
defined by 
\[o_\alpha(y)_n = \text{ order of } \alpha(y)_n.\]

If we write $y = (x,\phi)$ with $x\in \CM(\bar{k})$ and $\phi: \hat{\mu} \to \uAut(x)$, we will show in Lemma \ref{gerby} below that $o_\alpha(y)_n^2 |  (x,x)$ and using this we introduce the modified gerbe function

\[ \tilde\alpha: I_{\hat{\mu}}\CM(\bar{k}) \rightarrow \prod_{n\geq 1}\BQ/\BZ, \; (x,\phi) \mapsto (\alpha(x,\phi)_n+ \frac{b_1(x,x)}{o_\alpha(x,\phi)_n^2}\cdot \frac{1}{2})_{n\geq 1},\]

with $b_1$ as in the choice of $\BL^{\frac{1}{2}}$. Notice that 
\begin{equation}\label{modal} e^{2\pi i \tilde\alpha(x,\phi)_n} = (-1)^{\frac{b_1(x,x)}{o_\alpha(x,\phi)_n^2}}e^{2\pi i \alpha(x,\phi)_n}.\end{equation}

Finally for any subset of isomorphism classes $Z \subset I_{\hat{\mu}}\CM(\bar{k})$ we define the associated weighted point count $ \#^{\Bw,\tilde\alpha}Z \in \CV$ by the formula

\[  (\#^{\Bw,\tilde\alpha}Z)_n = \sum_{z \in Z(k_n)} \frac{e^{2\pi i \tilde\alpha(z)_n} \BL_n^{-\Bw(z)}}{|\Aut(z)(k_n)|}.\]

\begin{theorem}\label{plid} For any $x \in \BM$ the generating series
\[ \sum_{r \geq 1} \#^{\Bw,\tilde\alpha} I_{\mu_r}B\uAut(x) T^r \in \CV[[T]] \]
is a rational function in $T$ of degree $0$. The counting function $\MF \in \mcf(\BM)_0$ defined by 
\begin{equation}\label{limfor} \MF(x) = -(-1)^{b_2(x,x)}\BL^{\frac{-(x,x)-1}{2}}\lim_{T\to \infty}\sum_{r \geq 1} \#^{\Bw,\tilde\alpha} I_{\mu_r}B\uAut(x) T^r   \end{equation}
satisfies the following identity in $\mcf(\BM)$:

\begin{equation}\label{plelog}    
\frac{\MF}{\BL^{\frac{1}{2}} - \BL^{-\frac{1}{2}}} = \Log\left( \frac{\BL^{\frac{(,)}{2}}}{|\Aut(\cdot)|}  \right).\end{equation}
\end{theorem}

The proof of this identity will be given in the following subsection.

\subsection{Proof of Theorem \ref{plid}}
We fix an integer $n\geq 1$ and an object $x \in \BM(k_n)=\BM(\bar{k})_n$, where we have used that $\BM$ is a $k$-linear stack and thus Remark \ref{rmk:bij} applies. For simplicity we'll write $\sigma$ for the Frobenius over $k_n$ for the rest of this section. 
\begin{definition}
    For an integers $m,s \geq 1$ we consider the trace map
    \begin{equation*}
        \Tr_{m,s}\colon \BM(k_{nm})^s \to M(k_n), \; (y_1,\hdots,y_s) \mapsto \sum_{i=1}^s \Tr_{nm/n}(y_i).
    \end{equation*}

    Given an element $y=(y_1,\hdots,y_s) \in \Tr_{m,s}^{-1}(x)$, the shift $(y_2,\hdots,y_s,\sigma(y_1))$ lies again in $\Tr^{-1}_{m,s}(x)$. We call the elements of $\Tr^{-1}_{m,s}(x)$ obtained by iterating this operation cyclic shifts of $y$.
\end{definition}

Let $\chi\colon \mu_r \to \uAut(x)$ be a homomorphism. 

\begin{lemma}
    The homomorphism $\chi_{\bar k}\colon  \mu_r \to \uAut(x)_{\bar k}$ can be lifted to a homomorphism $\tilde\chi\colon  \mu_r \to \Aut(x)_{\bar k}$.
\end{lemma}
\begin{proof}
    By \cite[Prop. 3.15]{Ziegler22} the homomorphism $\chi_{\bar k}$ factors through a maximal torus $T$ of $\uAut(x)_{\bar k}$. If $T'$ is the preimage of $T$ in $\Aut(x)_{\bar k}$, then the projection $T'\to T$ has kernel $\BG_m$ and hence splits as $T'=T \times \BG_m$. Using this one obtains a lift $\tilde\chi$.
\end{proof}
Let $\tilde\chi\colon \mu_r \to \Aut(x)_{\bar k}$ be such a lift. This defines a weight decomposition 
\[x=\bigoplus_{w \in (\BQ/\BZ)[r]} x^w\] in $\BM(\bar k)$: The action of $\mu_r$ on $x$ via $\tilde\chi$ induces an action of $\mu_r$ on $\End(x)_{\bar k}$, and we let $\id=\sum_{w \in (\BQ/\BZ)[r]} \id^w$ be the associated decomposition of the identity element into homogeneous components. For each $w$, let $x^w \subset x$ be the image of $\id^w\colon x \to x$. Then the inclusions $x^w \into x$ and the morphisms $\id^w$ give mutually inverse isomorphisms $\oplus_w x^w \rightleftarrows x$ in $\BM(\bar k)$.

Since $x$ is defined over $k_n$, the Frobenius generator $\sigma$ of $\Gal(\bar k/ k_n)$ induces an identification $\sigma^*:x^\sigma \xrightarrow{\sim} x$. Under this identification, the homomorphism $\sigma^*(\tilde\chi)\colon \mu_r \to \Aut(x)^\sigma_{\bar k}$ is another lift of $\chi$ and so differs from $\tilde\chi$ by a homomorphism 
\begin{equation}\label{lll}\lambda\colon \mu_r \to \BG_m. \end{equation} So $\lambda$ is given by a number $w_\lambda \in (\BQ/\BZ)[r]$ for which $\sigma^*\colon x^\sigma \isoto x$ sends $(x^w)^\sigma$ to $x^{w+w_\lambda}$. 

We lift the weights appearing in the decomposition of $x$ to unique rational numbers $0 < w_1 < w_2 < \hdots < w_r \leq 1$ and let $x_i \defeq x^{w_i}$. We write $w_\lambda=d/m$ for coprime integers $d,m \geq 1$. Since the set of weights of $x$ is invariant by addition by $d/m$ and hence by $1/m$, we see that if we choose $s\geq 1$ such that $w_s \leq 1/m < w_{s+1}$, the $w_i$ are all of the form $w_j+k/m$ for some $1\leq j\leq s$ and some $1\leq k \leq m$. In particular $r=ms$. Then the isomorphism $\sigma^*\colon x^\sigma \to x$ sends $x_i^\sigma$ to $x_{i+ds \pmod {ms}}$ and so $(x_1,\hdots,x_s) \in \Tr_{m,s}^{-1}(x)$. 

Conversely, let $(y_1,\hdots,y_s) \in \Tr_{m,s}^{-1}(x)$ for given integers $m,s\geq 1$, let $w_1,\hdots,w_s$ be elements of $(\BQ/\BZ)[r]$ satisfying $0 < w_1 < w_2 < \hdots < w_s \leq 1/m$ and let $d$ be an integer coprime with $m$. Let $x \defeq \Tr_{m,s}(y_1,\hdots,y_s) \in \BM(k_n)$. Since $y_{k_{nm}}\cong x_{k_{nm}}$, by \cite[3.4]{Mo19} there exists an isomorphism $y\cong x$ over $k_n$.

Such an isomorphism $y \cong x$ over $k_n$ induces a homomorphism $\tilde \chi\colon \mu_r \to \Aut(x)_{\bar k}$, which acts on each factor $y_i^{\sigma^j}$ of $x$ with weight $w_i+\frac{jd}{m}$. By construction $\sigma^*(\tilde\chi)$ is equal to $\tilde\chi+\lambda$ for the homomorphism $\lambda\colon \mu_r\to \BG_m$ with weight $d/m$ and hence the composition $\chi$ of $\tilde\chi$ with the projection $\Aut(x)_{\bar k} \to \uAut(x)_{\bar k}$ is defined over $k_n$. A different choice of isomorphism $y \cong x$ over $k_n$ replaces the resulting $\tilde\chi$ by a $\Aut(x)(k_n)$-conjugate and hence does not change the isomorphism class of the resulting homomorphism $\chi \in I_{\mu_r}B\uAut(x)(k_n)$.

So in this way we can parametrize isomorphism classes in  $I_{\mu_r}B\uAut(x)(k_n)$ by integers $m,s,d \geq 1$ such that $(m,d)=1$, elements $(y_1,\hdots,y_s) \in \Tr^{-1}_{m,s}(x)$ and $r$-torsion elements $w_1,\hdots, w_s$ of the set
\begin{equation*}
    \Delta_{m,s}\defeq \{(w_1,\hdots,w_s) \in \BQ \mid 0< w_1 < w_2 < \hdots < w_s \leq 1/m\}.
\end{equation*}
Two homomorphisms $\tilde \chi\colon \mu_r \to \Aut(x)_{\bar k}$ give the same homomorphism $\chi$ if and only if they differ by an element of $\Hom(\mu_r, \BG_m) \cong \BZ/r \BZ$. In our parametrization, this amounts to replacing $(y_1,\hdots,y_s)$ by a cyclic shift $(y_k,\hdots, y_s, \sigma(y_1), \hdots, \sigma(y_{k-1}))$ and the weights $w_1,\hdots,w_s$ by the weights $w_i+w \pmod{ 1/m}$ for some $w \in (\BQ/\BZ)[r]$ accordingly. This defines an action of $(\BQ/\BZ)[r]$ on $\Delta_{m,s}[r]$ and we denote the resulting quotient by $\overline{\Delta_{m,s}[r]}$.

Since $H^1(k_n,\BG_m)=0$, the homomorphism $\Aut(x)(k_n) \to \uAut(x)(k_n)$ is surjective. Hence if two different choices of homomorphism $\tilde\chi\colon \mu_r \to \Aut(x)$ yield isomorphic objects $\chi \in I_{\mu_r}B\uAut(x)(k_n)$, then up to a cyclic shift of the weights $(w_i)_i$ as above the homomorphisms $\tilde \chi$ are already conjugate over $k_n$ and so come from different choices of isomorphism $x\cong \Tr_{m,s}(y_1,\hdots,y_2)$. This shows that the only ambiguity in our parametrization of the isomorphism classes in $I_{\mu_r}B\uAut(x)(k_n)$ is described by the above cyclic shifts.

    Next we wish to determine the group $\Cent_{\uAut(x)}(\chi)$ in terms of this parametrisation, since this is exactly the automorphism group of the pair $(x,\chi)\in I_{\mu_r}\CM$. The projection $\pi\colon \Aut(x) \to \uAut(x)$ restricts to a homomorphism $\Cent_{\Aut(x)_{\bar k}}(\tilde \chi) \to \Cent_{\uAut(x)}(\chi)$. We denote by $\Cent_{\uAut(x)}(\chi)^\circ \subset \Cent_{\uAut(x)}(\chi)$ the connected component of the identity and by $$\pi_0(\chi) \defeq \Cent_{\uAut(x)}(\chi)/\Cent_{\uAut(x)}(\chi)^\circ$$ the group of connected components of the centralizer of $\chi$.

    \begin{lemma} \label{Pi0Lemma1}
    \begin{enumerate}
        \item  The group scheme $\Cent_{\Aut(x)_{\bar k}}(\tilde \chi)$ is connected and surjects onto $\Cent_{\uAut(x)}(\chi)^\circ$.
   \item The finite group scheme $\pi_0(\chi)$ is constant over $k_n$.
        \item $|\Cent_{\uAut(x)}(\chi)(k_n)|=|\pi_0(\chi)| \prod_{i=1}^s |\Aut(y_i)(k_{nm})|/|\BG_m(k_n)|$
    \end{enumerate}       
    \end{lemma}
\begin{proof}
    (i) A decomposition of $x$ into indecomposable elements in $\BM(\bar k)$ identifies $\Aut(x)_{\bar k}$ with an extension of a product of copies of the group $\GL_{n,\bar k}$ (for varying $n$) by a unipotent group. Hence the claim about the centralizer of $\tilde\chi$ follows from the fact that the centralizer of every homomorphism $\mu_r \to \GL_n$ is smooth and connected.

    Hence $\pi$ factors through a morphism $\tilde\pi\colon \Cent_{\Aut(x)_{\bar k}}(\tilde\chi) \to \Cent_{\uAut(x)}(\chi)^\circ$. By \cite[A.8.10]{CGP} the Lie algebra of $\Cent_{\Aut(x)_{\bar k}}(\tilde\chi)$ (resp. $\Cent_{\uAut(x)}(\chi)^\circ$) is equal to the subspace of the Lie algebra of $\Aut(x)_{\bar k}$ (resp. $\uAut(x)_{\bar k}$) fixed by $\mu_r$ under the natural $\mu_r$-action via $\tilde\chi$ (resp. $\chi$) and the adjoint action. This implies that $\tilde\chi$ induces a surjective map of Lie algebras. Since both these centralizers are smooth and connected this implies that $\tilde\pi$ is surjective.
    
    (ii) Every connected component of $\Cent_{\uAut(x)(\chi)}$ is a torsor under $\Cent_{\uAut(x)}(\chi)^\circ$ and hence has a $k_n$-point by Lang's theorem.

    (iii) We start with the identity
    \begin{equation*}
        |\Cent_{\uAut(x)}(\chi)(k_n)|=|\pi_0(\chi)||\Cent_{\uAut(x)}(\chi)^\circ(k_{n})|.
    \end{equation*}
        Since $\tilde\chi^\sigma$ differs from $\tilde\chi$ by a central cocharacter, the centralizer $\Cent_{\Aut(x)}(\tilde\chi)$ canonically descends to $k_n$. Since we chose our isomorphism $x\cong \Tr_{m,s}(y_1,\hdots,y_s)$ over $k_n$, the surjection $\Cent_{\Aut(x)}(\tilde\chi) \to \Cent_{\uAut(x)}(\chi)$ descends to a homomorphism over $k_n$ with kernel $\BG_m$. Using this we find
        \begin{equation*}
            |\Cent_{\uAut(x)}(\chi)^\circ(k_{n})|=|\BG_m(k_n)||\Cent_{\Aut(x)}(\tilde\chi)(k_{nm})^\sigma|=|\BG_m(k_n)| \prod_i |\Aut(y_i)(k_{nm})|,
        \end{equation*}
        which concludes the proof.
\end{proof}

     It remains to describe $\pi_0(\chi)$:

    \begin{lemma} \label{Pi0ChiDesc}
        \begin{enumerate}
            \item Sending a class $g \Cent_{\uAut(x)}(\chi)^\circ \in \pi_0(\chi)$ to the homomorphism ${}^{\tilde{g}}{\tilde\chi} - \tilde\chi \colon \mu_r \to \BG_m$ associated to a lift $\tilde g \in \Aut(x)(\bar k)$ of $g$ gives a well-defined injection
            \begin{equation*}
                \iota\colon \pi_0(\chi) \into \Hom(\mu_r, \BG_m).
            \end{equation*}
            \item A homomorphism $\tau\colon \mu_r \to \BG_m$ is in the image of $\iota$ if and only if $\tilde\chi+\tau$ is $\Aut(x)(\bar k)$-conjugate to $\lambda$.
                 
        \end{enumerate}
    \end{lemma}
\begin{proof}
    (i) A different lift $\tilde g$ differs from the given one by an element in the center of $\Aut(x)$ and so doesn't change the homomorphism ${}^{\tilde g} \tilde \chi - \chi$. Since every element of $\Cent_{\uAut(x)}(\chi)^\circ \in \pi_0(\chi)(\bar k)$ can be lifted to $\Cent_{\Aut(x)_{\bar k}}(\tilde \chi)$ by Lemma \ref{Pi0Lemma1}, the choice of $g$ also doesn't affect this homomorphism. Furthermore, by construction the composition of ${}^{\tilde g}\tilde\chi-\tilde\chi$ with $\pi\colon \Aut(x) \to \uAut(x)$ is the trivial homomorphism, and so ${}^{\tilde g}\tilde\chi-\tilde\chi$ factors through $\BG_m$ as claimed. So $\iota$ is well-defined.

    The fact that
    \begin{equation*}
        {}^{\tilde g \tilde h}\tilde\chi - \tilde \chi={}^{\tilde g}({}^{\tilde h}\tilde\chi - \tilde\chi) + ({}^{\tilde g}\tilde\chi - \tilde \chi)=({}^{\tilde h}\tilde\chi - \tilde\chi) + ({}^{\tilde g}\tilde\chi - \tilde \chi)
    \end{equation*}
    shows that $\iota$ is a homomorphism.

    Finally if $\iota(\chi)$ is trivial, then $\tilde g$ centralizes $\tilde\chi$ and hence $g \in \Cent_{\Aut(x)}(\bar k)$ by Lemma \ref{Pi0Lemma1}. This shows that $\iota$ is injective.

    (ii) If $\tilde\chi+\tau={}^{\tilde g}\tilde\chi$ for some $\tilde g \in \Aut(x)(\bar k)$ with image $\bar g$ in $\pi_0(\chi)(\bar k)$, then $\tau=\iota(\bar g)$, and conversely.

\end{proof}

This implies:
\begin{lemma}
    In our parametrization of the isomorphism classes of $I_{\mu_r}B\uAut(x)(k_n)$, each $\chi \colon \mu_r \to \uAut(x)$ appears $ms/|\pi_0(\chi)|$ times.
\end{lemma}
\begin{proof}
    As noted above, the only ambiguity in our parametrization comes from lifts $\tilde\chi'$ of $\chi$ differing from a given lift $\tilde\chi$ by a homomorphism $\tau\colon \mu_r \to \BG_m$, for which the resulting element of $\Tr^{-1}_{m,s}(x)$ is a cyclic shift of $(y_1,\hdots,y_r)$, with the weights $(w_1,\hdots,w_r)$ shifted accordingly and increased by the weight of $\tau$. This cyclic shift is equal to the original datum if and only if $\tilde\chi+\tau$ is $\Aut(x)(\bar k)$-conjugate to $\tilde\chi$. So by Lemma \ref{Pi0ChiDesc} each $\chi$ appears in our parametrisation $ms/|\pi_0(\chi)|$ times.
\end{proof}

 The following describes the Hasse invariant in terms of our parametrisation:
\begin{lemma}\label{gerby}
    The Hasse invariant of the gerbe $\alpha$ at $\chi$ is given by $\alpha(\chi)_n=w_\lambda = \frac{d}{m} \in \BQ/\BZ$, with $\lambda$ as in \eqref{lll}.
    
    If $x= \oplus_i x_i$ is the decomposition of $x$ into  indecomposables, where $x_i$ has splitting field of degree $d_i$, then the order of $\alpha(\chi)_n$ for any $\chi\colon \mu_r \to \uAut(x)$ divides the gcd of the $d_i$. 
\end{lemma}
\begin{proof}
    By definition, the element $\alpha(\chi)_n$ is obtained by pulling back $\BM \to \CM$ along $\chi\colon B_k \mu_r \to \CM$, choosing a section $s \in \chi^*\BM(B_{\bar k} \mu_r)$ and considering the object
    \begin{equation*}
        \alpha(\chi)_n=\phi(s)-s \in H^1(B_{\bar k} \mu_r,\BG_m)\cong (\BQ/\BZ)[r].
    \end{equation*}
    But in our situation, such a section $s$ is the same as a lift $\lambda\colon \mu_r \to \Aut(x)_{\bar k}$, and $\phi(s)-s$ is given by the homomorphism $\lambda\colon \mu_r \to \BG_m$ which corresponds to $w_\lambda = \frac{d}{m}$ in $\BQ/\BZ$.

    The fact that the order of $\alpha(\chi)_n$ has to divide each $d_i$ follows from the definition of $w_\lambda$. 
\end{proof}
The following describes the weight in terms of our parametrisation:
\begin{lemma} Let $x =\Tr_{m,s} (y_1,\hdots,y_s)$ be the decomposition corresponding to $\tilde\chi: \mu_r \to \uAut(x)$. Then we have
\[ \Bw(\chi) + \frac{1}{2}(x,x)= -\frac{m}{2} \sum_{1\leq i\leq s} (y_i,y_i).\]

\end{lemma}

\begin{proof} We work over $\bar{k}$ since the weight is insensitive to base change. For every $0\leq i \leq \delta$ we have 
\[\Ext^i(x,x) = \ \bigoplus_{1\leq i_1, i_2 \leq s, 1\leq j_1, j_2\leq m} \Ext^i(y_{i_1}^{\sigma^{j_1}},y_{i_2}^{\sigma^{j_2}}) .   \]

Now, the $\mu_r$-action on each $\Ext^i(y_i^{\sigma^j},y_i^{\sigma^j})$ is trivial, while its weight on each $\Ext^i(y_{i_1}^{\sigma^{j_1}},y_{i_2}^{\sigma^{j_2}}) $ with $(i_1,j_1) \neq (i_2,j_2)$ is exactly  the negative of the one on $ \Ext^i(y_{i_2}^{\sigma^{j_2}},y_{i_1}^{\sigma^{j_1}})$. Using that $(\cdot,\cdot)$ is bilinear and by assumption also symmetric, we deduce 
\[ \Bw(\chi) +  \frac{1}{2}(x,x) = -\frac{1}{2}\sum_{1\leq i\leq s, 1\leq j\leq m}(y_i^{\sigma^j},y_i^{\sigma^j}) = -\frac{m}{2} \sum_{1\leq i\leq s} (y_i,y_i).  \]
\end{proof}

    Altogether we find the following formula for $\MF(x)_n$:
    \begin{align*}
        (&\#^{\Bw,\tilde\alpha} I_{\mu_r}B\uAut(x))_n  \\ &= \sum_{m,s \geq 1} \sum_{(d,m)=1} \sum_{(y_1,\hdots,y_s) \in \Tr^{-1}_{m,s}(x)}  |\overline{\Delta_{m,s}[r]}|\frac{|\pi_0(\chi)|}{ms} (-1)^{\frac{b_1(x,x)}{m^2}}e^{2\pi i \frac{d}{m}}  \frac{(\BL_n-1)\BL_n^{\frac{1}{2}(x,x)+\frac{m}{2}\sum_{i=1}^s (y_i,y_i)}}{|\pi_0(\chi)|\prod_{i=1}^s |\Aut(y_i)(k_{nm)}|}  \\
        &= \BL_n^{\frac{1}{2}(x,x)} \sum_{m,s \geq 1}    \sum_{(d,m)=1} \sum_{(y_1,\hdots,y_s) \in \Tr^{-1}_{m,s}(x)} \\
        & \hspace{100pt}(-1)^{(b_1+b_2 m)\sum_{i=1}^s(y_i,y_i)} |\overline{\Delta_{m,s}[r]}|\frac{1}{ms} (-1)^{\frac{b_1(x,x)}{m^2}}e^{2\pi i \frac{d}{m}}  \frac{(\BL_n-1)\BL_{nm}^{\frac{1}{2}\sum_{i=1}^s (y_i,y_i)}}{\prod_{i=1}^s |\Aut(y_i)(k_{nm)}|}  \\
        &= (-1)^{b_2(x,x)} \BL_n^{\frac{1}{2}(x,x)} \sum_{m,s \geq 1}    \sum_{(d,m)=1} \sum_{(y_1,\hdots,y_s) \in \Tr^{-1}_{m,s}(x)}  |\overline{\Delta_{m,s}[r]}|\frac{1}{ms} e^{2\pi i \frac{d}{m}}  \frac{(\BL_n-1)\BL_{nm}^{\frac{1}{2}\sum_{i=1}^s (y_i,y_i)}}{\prod_{i=1}^s |\Aut(y_i)(k_{nm)}|} 
    \end{align*} 

Here in the first equality we also used \eqref{modal}. Notice that the only dependence on $d$ is in the factor $e^{2\pi i \frac{d}{m}}$ and hence the sum over all $d$ coprime with $m$ can be replaced by a factor $\mu(m)$. This gives

    \begin{multline}\label{muli}
       (-1)^{b_2(x,x)} \frac{\BL_n^{-\frac{1}{2}(x,x)}}{\BL_n-1} (\#^{\Bw,\tilde\alpha} I_{\mu_r}B\uAut(x))_n  \\ =
          \sum_{m,s \geq 1}  \frac{\mu(m)}{ms} \sum_{(y_1,\hdots,y_s) \in \Tr^{-1}_{m,s}(x)}  |\overline{\Delta_{m,s}[r]}|  \frac{\BL_{nm}^{\frac{1}{2}\sum_{i=1}^s (y_i,y_i)}}{\prod_{i=1}^s |\Aut(y_i)(k_{nm)}|} 
    \end{multline}

Notice that since $x$ is fixed, both sums on the right hand side of \eqref{muli} are finite and the only dependence on $r$ comes from $|\overline{\Delta_{m,s}[r]}|$. So by Lemma \ref{esch} below we deduce rationality and the formula

    \begin{multline*}
       \MF(x) =  -(-1)^{b_2(x,x)} \frac{\BL_n^{\frac{-(x,x)-1}{2}}}{\BL^{\frac{1}{2}}_n-\BL^{-\frac{1}{2}}_n} \lim_{T\to \infty}\sum_{r \geq 1}  \#^{\Bw,\tilde\alpha} I_{\mu_r}B\uAut(x) T^r  \\
       = \sum_{m,s \geq 1} \frac{\mu(m)}{ms} \sum_{(y_1,\hdots,y_s) \in \Tr^{-1}_{m,s}(x)}  \frac{\BL_{nm}^{\frac{1}{2}\sum_{i=1}^s (y_i,y_i)}}{\prod_{i=1}^s |\Aut(y_i)(k_{nm)}|} \lim_{T\to \infty}\sum_{r \geq 1}  |\overline{\Delta_{m,s}[r]}| T^r\\
        = \sum_{m,s \geq 1}(-1)^{s-1} \frac{\mu(m)}{ms} \sum_{(y_1,\hdots,y_s) \in \Tr^{-1}_{m,s}(x)}  \frac{\BL_{nm}^{\frac{1}{2}\sum_{i=1}^s (y_i,y_i)}}{\prod_{i=1}^s |\Aut(y_i)(k_{nm)}|}.
    \end{multline*}

Theorem \ref{plid} now follows directly from \eqref{logeq}.

    \begin{lemma}\label{esch}
The series $\sum_{r \geq 1} |\overline{\Delta_{m,s}[r]}| T^r$ is a rational function in $T$ of degree $0$ and
       \[ \lim_{T\to \infty}\sum_{r \geq 1} |\overline{\Delta_{m,s}[r]}| T^r=(-1)^{s-1}.\]
    \end{lemma}
\begin{proof}
    The class of an element $(w_1,\hdots,w_s) \in \Delta_{m,s}[r]$ in $\overline{\Delta_{m,s}[r]}$ is completely described by the differences $d_i=w_{i+1}-w_i$ for $1\leq i \leq s-1$, subject to the condition that $\sum_{i=1}^{s-1}d_i < 1/m$. So by \cite[Lemma 8.5.2]{HL15} the limit $ \lim_{T\to \infty}\sum_{r \geq 1} |\overline{\Delta_{m,s}[r]}| T^r$ computes the compactly supported Euler characteristic of 
    \begin{equation*}
        \{(d_i) \in (0,1)^{s-1} \mid \sum_i d_i < 1/m \},
    \end{equation*}
    which is equal to $(-1)^{s-1}$.
\end{proof}

\section{Applications to smooth linear stacks}\label{msds}
\subsection{Integrals and intersection cohomology}\label{ii}
Throughout this section we let $R$ be a finitely generated $\BZ$-algebra and $T = \Spec(R)$. 

We continue to consider an $R$-linear category $\Cc$ as in Section \ref{axms} and assume further that $\Cc$ is hereditary i.e. $\delta = 1$. The vanishing of the higher $\Ext$-groups implies in particular that $ \BM_\Cc$ is smooth. We consider an open substack $\BM \subset \BM_\Cc$ stable under extensions. 

In order to apply our theory we need to assume that $\BM$ admits a reasonable notion of moduli space.

\begin{enumerate}
\item In the decomposition 
\[\BM = \sqcup_{\gamma \in \pi_0(\BM)} \BM_\gamma,\]
each $\BM_\gamma$ is a finite type $R$-linear quotient stack.

\item For each $\gamma \in \pi_0(\BM)$ there exists an adequate moduli space $\tilde{\pi}_\gamma \colon \BM_\gamma \to \M_\gamma$.
\end{enumerate}

We write $\CM_\Cc = \sqcup_{\gamma \in  \pi_0(\BM)} \CM_\gamma$ for the $\BG_m$-rigidification along the scalar automorphisms and $\pi_\gamma\colon \CM_\gamma \to \M_\gamma$ for the induced morphisms. We write $(\gamma,\gamma)$ for the value of the Euler pairing on $\BM_\gamma$. By definition we have  $(\gamma,\gamma) = - \dim_R \BM_\gamma$. The set of connected components $\pi_0(\BM)$ inherits a commutative monoid structure from the direct sum operation $\oplus$.
For any $\phi:S \to T$ we write $\pi_{\gamma,\phi} \colon \CM_{\gamma,\phi} \to \M_{\gamma,\phi}$ for the base change to $S$.

In order to apply our integration theory we also need Conditions $(i) - (iii)$ from Section \ref{iob} to be verified. While $S$-completeness follows from \cite[Proposition 3.44]{AHH23} the other two conditions are not automatically satisfied.  

Fix a $\gamma \in \pi_0(\BM)$ such that there exists an open-dense subspace $W \subset \M_\gamma$ such that $\CM_\gamma \times_{\M_\gamma} W \to W$ is an equivalence.  Let $F$ be a mixed characteristic local field containing no $p$-power roots of $1$ and $\Oc_F \subset F$ its ring of integers. Then for any morphism $\phi: \Oo_F \to T$ we further assume:
\begin{enumerate}

\item[(iii)] The map 
\[ \pi_{\infty}\colon \CM_{\gamma,\phi}(\Dinfty)^\sharp \to \M_{\gamma,\phi}(\Oc_F)^\sharp\]
is surjective.
\item[(iv)] The canonical bundle of $\CM_{\gamma}$ is Zariski-locally trivial.     
\end{enumerate}

Under these assumptions, Theorem \ref{inbase} equips $\M_{\gamma,\phi}(\Oc_F)^\sharp$ with a canonical measure $\mu_{can}$ and furthermore, rigidification $\BM_\gamma \to \CM_\gamma$ equips $\M_{\gamma,\phi}(\Oc_F)^\sharp$ with a modified gerbe function $\tilde{\alpha}$ as in Section \ref{pls} and Remark \ref{gerf}.

\begin{theorem}\label{icint} There exists an open subscheme $T_\gamma \subset T$ such that for every morphism $\phi: \Spec(\Oo_F) \to T_\gamma$, where $F$ is a mixed characteristic local field containing no $p$-power roots of $1$ and $\Oc_F \subset F$ its ring of integers, and every $x \in \M_{\gamma,\phi}(k_F)$ we have

\begin{equation}\label{eqicint} \Tr(\Fr,\left(\CIC_{\M_{\gamma,\phi}}\right)_x) = - q_F^{-(\gamma,\gamma)+1} \int_{B(x)} e^{2\pi i\tilde{\alpha}} \ d\mu_{can}.  \end{equation}
Here $k_F$ denote the residue field of $F$, $q_F$ its cardinality and
 $B(x) = \{y \in \M_{\gamma,\phi}(\Oc_F)^\sharp \ |\ y_{|\Spec(k)} = x \}$.
\end{theorem}

The proof of Theorem \ref{icint} proceeds by showing that both sides satisfy a recursive relation coming from Donaldson-Thomas theory and identifies both sides of \eqref{eqicint} with the BPS-invariants of $\BM_\gamma$. 

Following  \cite[Section 6.5]{Me15} we define the motivic BPS-invariants of $\BM$ as a class $\BPS$ in the relative localized Grothendieck ring of varieties $K_0(\Var/\M)[\BL^{-1/2},(\BL^N-1)^{-1} : N \geq 1]$ over $\M$, defined by the equation 

\begin{equation}\label{groexp}[\tilde{\pi}_!\BL^{\frac{(\cdot,\cdot)}{2}}] = \Sym \left( \frac{\BPS}{\BL^{1/2}-\BL^{-1/2}}  \right), \end{equation}
where $\BL^{\frac{(\cdot,\cdot)}{2}}$ is the constant function given on each connected component $\BM_{\gamma}$ by $\BL^{- \dim \BM_{\beta,\chi}/2} = \BL^{(\gamma,\gamma)/2}$. 

The convention in \cite{Me15} for the square root $\BL^{1/2}$ is specified by the requirement $\sigma_m (\BL^{1/2}) = 0$ for all $m\geq 2$, which is equivalent to 
\begin{equation}\label{sqconv} \psi_m(\BL^{1/2}) = (-1)^{m+1} (\BL^{1/2})^m. \end{equation}

We write $\BPS_\gamma$ for the restriction of $\BPS$ to $\M_\gamma$. Then, as in \cite[Appendix]{MR2453601} we obtain for every finite field $k$ and $\overline{\phi}: k \to \Spec(T)$ a counting function $\BPS_{\gamma,\overline{\phi}}\in \mcf(\M_{\gamma,\overline{\phi}})$ by counting points in the fibers of the motivic class $\BPS_{\gamma,\overline{\phi}}$. Notice that the convention \eqref{sqconv} for the square root $\BL^{1/2}$ corresponds to $b_0 = b_1 = 1$ in the notation of Section \ref{pls}. 

The following lemma relates the right-hand side of \eqref{eqicint} to these BPS-invariants:

\begin{lemma}\label{intbps} Let $\phi: \Spec(\Oo_F) \to T$, where $F$ is a mixed characteristic local field containing no $p$-power roots of $1$, and $\overline{\phi}:\Spec(k_F) \to T$ its restriction to the residue field. Then we have for every $x \in \M_{\gamma,\overline{\phi}}(k_F)$ the equality

\begin{equation}\label{intid} \BPS_{\gamma,\overline{\phi}}(x) =  (-1)^{(\gamma,\gamma)}q_F^{\frac{-(\gamma,\gamma)+1}{2}} \int_{B(x)} e^{2\pi i\tilde{\alpha}} \ d\mu_{can}. \end{equation}
\end{lemma}
\begin{proof} First  we deduce from Theorem \ref{plid} 
\[\BPS_{\gamma,\overline{\phi}} = \pi_! \MF,\]
with $\MF$ defined as in \eqref{limfor} for $x' \in \BM_{\gamma,\overline{\phi}}(k_F)$ by
\[\MF(x) = -(-1)^{(x',x')}\BL^{\frac{-(x',x')-1}{2}}\lim_{T\to \infty}\sum_{r \geq 1} \#^{\Bw,\tilde\alpha} I_{\mu_r}B\uAut(x) T^r.\]
By all the assumption we made at the beginning of this section we may now apply Theorem \ref{inbase} to get for any $x \in \M_{\gamma,\overline{\phi}}(k_F)$ 
\[ \pi_!\MF(x) = (-1)^{(\gamma,\gamma)}\BL^{\frac{-(\gamma,\gamma)+1}{2}}\int_{B(x)} e^{2\pi i\tilde{\alpha}} \ d\mu_{can}.\]
\end{proof}

In order to relate $\BPS$ to intersection cohomology we first consider the base change to $\BC$. Then \cite[Theorem 5.6]{Me15} shows that the image of $\BPS_{\gamma,\BC}$ in the Grothendieck ring of mixed Hodge modules is given by the class of the (normalized) intersection complex $\CIC_{\M_{\gamma,\BC}}(\frac{(\gamma,\gamma)-1}{2})$, where $(\cdot)$ denotes the Tate twist and $\CIC_{\M_{\gamma,\BC}}$ is normalized such that its restriction to the smooth locus is given by $\BQ[-(\gamma,\gamma)+1]$. For the purpose of Proposition \ref{bpsic} below we will ignore Hodge structures and work instead with constructible $\BQ_\ell$-sheaves.

\begin{proposition}\label{bpsic} There exists an open subscheme $T_\gamma \subset T$ such that for every finite field $k$ with cardinality $q$, every morphism $\Spec(k) \to T_\gamma$ and every $ x \in \M_{\gamma,\phi}(k)$ we have

\[\Tr(\Fr,\left(\CIC_{\M_{\gamma,\phi}}\right)_x) = (-1)^{-(\gamma,\gamma)+1} q^{\frac{-(\gamma,\gamma)+1}{2}} \BPS_{\gamma,\phi}(x).  \]
\end{proposition}

\begin{proof} If $\gamma$ cannot be decomposed non-trivially in the mononoid $(\pi_0(\Mc),\oplus)$, neither can any object in $\BM_{\gamma}$. Thus $\M_\gamma \cong \CM_\gamma$ is smooth and both sides of the equality are equal to $(-1)^{-(\gamma,\gamma)+1}$. For the left hand side this follows since $\CIC_{\M_{\gamma,\phi}} = \BQ_\ell[-(\gamma,\gamma)+1]$ and for the right hand by definition \eqref{groexp}.

In general, we show that $\Tr(\Fr,\CIC_{\M_{\gamma,\phi}}) $ satisfies the same defining equation \eqref{groexp} as $\BPS_{\gamma,\phi}$. Restricted to $\M_{\gamma,\phi}$ this reads as 

\[ \BL^{-(\gamma,\gamma)/2}\tilde{\pi}_{!}\frac{1}{|\Aut(\cdot)|} = \sum_{\substack{\alpha_i,\gamma_i\\ \sum_i \alpha_i \gamma_i = \gamma}} \prod_i \Sym_{\alpha_i}\left(\frac{ \BPS_{\gamma_i}}{\BL^{1/2}-\BL^{-1/2}} \right) . \]

It is convenient rewrite this as follows: 
\begin{equation}\label{finrec} \tilde{\pi}_{!}\frac{1}{|\Aut(\cdot)|} = \sum_{\substack{\alpha_i,\gamma_i\\ \sum_i \alpha_i \gamma_i =\gamma}} \BL^{\frac{1}{2}\left((\gamma,\gamma) - \sum_i \alpha_i (\gamma_i,\gamma_i)  \right)}  \prod_i \Sym_{\alpha_i}\left(\frac{\BL^{\frac{(\gamma_i,\gamma_i)-1}{2}} \BPS_{\gamma_i,x}}{\BL-1} \right) . \end{equation}

Notice that $(\gamma,\gamma) - \sum_i \alpha_i (\gamma_i,\gamma_i)  $ is even and thus by induction we deduce that $\BL^{\frac{(\gamma_i,\gamma_i)-1}{2}} \BPS_{\gamma_i,\phi}$ is a counting function not involving any roots of $\BL$. This allows us to avoid discussing the spreading out of $\BL^{1/2}$ below. 
 
For the class of the intersection complex in the Grothendieck ring of constructible $\BQ_\ell$-sheaves we have an analogous formula to \eqref{finrec} over $\BC$ by \cite[Theorem 5.6]{Me15}, see also \cite[Theorem 10.2.7]{BDNKP25}:

\begin{equation} \tilde{\pi}_{!}[\BQ_{\ell,\BM_{\gamma,\BC}}]= \sum_{\substack{\alpha_i,\gamma_i\\ \sum_i \alpha_i\gamma_i = \gamma}} \BL^{\frac{1}{2}\left((\gamma,\gamma) - \sum_i \alpha_i (\gamma_i,\gamma_i) \right)}  \prod_i \Sym_{\alpha_i}\left(\frac{ (-1)^{-(\gamma_i,\gamma_i)+1}\BL^{(\gamma_i,\gamma_i)-1}[\CIC_{\M_{\gamma_i,\BC}}]}{\BL-1} \right) . \end{equation}

 One can further upgrade this equality to an isomorphism in the derived category $D^{\geq}(\M_{\gamma,\BC},\BQ_\ell)$ by an explicit isomorphism as in \cite[Theorem C]{DM20}. As will be explained below, this isomorphism is geometric in nature and thus spreads out over some open $T_\gamma \subset T$. The construction of the isomorphism in \textit{loc. cit.} relies on the fact, that $\tilde \pi$ can be approximated by proper maps and thus the decomposition theorem applies. Since furthermore $\tilde \pi$ restricted to the open dense locus of stable sheaves is a $\BG_m$-gerbe, one obtains for every $\gamma$ a morphism
\[ \CIC_{\M_{\gamma,\BC}}[(\gamma,\gamma)-1] \otimes H^*(B\BG_m) \to  \tilde{\pi}_* \BQ_{\ell,\BM_{\gamma,\BC}}.  \]
Using the relative t-structure on $D^{\geq}(\M_{(\beta,\chi),T},\BQ_\ell)$ constructed in \cite{HS23} and in particular Theorem 1.10 in \textit{loc. cit.} we can extend this morphism to 
\[ \CIC_{\M_{\gamma,T'}}[(\gamma,\gamma)-1] \otimes H^*(B\BG_m) \to  \tilde{\pi}_* \BQ_{\ell,\BM_{\gamma,T'}},  \]
for some dense open $T' \subset T$. From there one may use the usual pull-push operations used in the construction of cohomological Hall algebras to obtain a morphism over $T'$ which is the aforementioned isomorphism over $\BC$. Hence we deduce that we have an isomorphism over some $T_\gamma \subset T'$ and passing to the trace of Frobenius we deduce the proposition.
\end{proof}

Putting Lemma \ref{intbps} and Proposition \ref{bpsic} together, we get Theorem \ref{icint}.

\subsection{Moduli of sheaves on del Pezzo surfaces}

In this final section we apply our integration theory to moduli of 1-dimensional sheaves on a del Pezzo surface. In particular we give a new proof of a theorem of Maulik-Shen \cite{MS20} on $\chi$-independence of BPS-invariants.

 Let $S_{\BC}$ be a complex del Pezzo surface, $\beta_\BC$ an ample, base-point free curve class on $S_\BC$ and $\chi \in \BZ$. We write $\BM_{\beta,\chi,\BC}$ for the moduli stack of pure $1$-dimensional Gieseker-semistable sheaves on $S$ with support of class $\beta_\BC$ and Euler-characteristic $\chi$. Here, semistability is with respect to a fixed polarization $L_\BC$ and the slope function
\[\mu(\CF) = \frac{\chi(\CF)}{c_1(\CF) \cdot L_\BC}.\]
For $\tau \in \BQ$, we also write $\BM_{\tau,\BC}$ for the disjoint union of the moduli stacks $\BM_{\beta,\chi,\BC}$ with fixed slope $\tau$. 

Choosing a spreading out $S$ of $S_\BC$ over $T=\Spec(R)$ for a finitely generated $\BZ$-algebra $R$, together with its polarization $L$ and curve class $\beta$, we obtain a family $\BM_{\tau}/T$ and GIT-quotients $\tilde{\pi}_{\beta,\chi}: \BM_{\beta,\chi} \to \M_{\beta,\chi}$ satisfying conditions $(i)$ and $(ii)$ above, see for example \cite[Section 2.1]{COW21}.

\begin{lemma}\label{conc} For any pair $(\beta,\chi)$ the rigidification $\pi_{\beta,\chi}:\CM_{\beta,\chi} \to \M_{\beta,\chi}$ is generically an equivalence. Furthermore for any non-archimedean local field $F$ and any morphism $\phi:\Spec(\Oo_F) \to T$ the pullback $\pi_{\beta,\chi,\phi}\colon \CM_{\beta,\chi,\phi} \to \M_{\beta,\chi,\phi}$  satisfies conditions $(iii) - (iv)$ from Section \ref{ii}.
\end{lemma}

\begin{proof}

By \cite[Proposition 2.2.1]{COW21} the complement of the open locus $U\subset \BM_{\beta,\chi}$ of line bundles supported on curves in $S$ has codimension $\geq 2$. Since the canonical bundle restricted to $U$ is Zariski-locally trivial, these local trivialisations extend to all of $\BM_{\beta,\chi}$ and thus $(iv)$ is satisfied. Furthermore $\pi_{\beta,\chi}$ is an equivalence if we further restrict to the locus of stable line-bundles.

It remains to prove the lifting property (iv). Fix $x \in X_{\beta,\chi}(\Oc_F)^\sharp$. We may consider its generic fibre point $x_F$ as a point in $\CM_{\beta,\chi}(F)$. The obstruction of lifting $x_F$ to $\BM_{\beta,\chi}$ is given by the pullback of the $\BG_m$-gerbe $\BM_{\beta,\chi} \to \CM_{\beta,\chi}$ along $x_F$, which is an element in $\alpha_x \in H^2(F,\BG_m) \cong \BQ/\BZ$ and thus in particular torsion. Thus, there exists $r \geq 1$ such that $\alpha_x$ extends to a gerbe $\beta$ on $\Dr$. The element $x_F$ corresponds to an $1 \boxtimes \beta$-twisted sheaf $\F_{\eta}$ on the generic fibre of $S \times_{\Oc_F} \Dr$. A similar argument to the one used in \cite[Lemma 3.1.3.1]{Lieblich_2008} shows that $\F_{\eta}$ can be extended to a twisted coherent sheaf $\F$ on $S \otimes_{\Oc_F} \Dr$ (\emph{loc. cit.} only considers twisted sheaves on schemes, but the extension to the tame stack appearing above is straightforward). 

Pulling back along a cyclic extension $\Spec \Oc_L \to \Dr$ which trivialises $\beta$ we obtain a family of coherent sheaves $\F$ on $S \times_{\Oc_F} \Oc_L$ satisfying a twisted equivariance condition corresponding to the gerbe $\beta$. We apply Langton's algorithm to this family to obtain a semistable extension satisfying the same equivariance condition, which amounts to a morphism $\Dr \to \Mc_{\beta,\chi}$ extending $x_F$.
\end{proof}

We thus have established all properties of $\BM_{\tau}\to T$ in order to apply the results of Section \ref{ii}. First we deduce an arithmetic $\chi$-independence. For this consider the Hilbert-Chow morphism 
\[h_{\beta,\chi}: \M_{\beta,\chi} \to B=\BP H^0(S,\CO_S(\beta)),\] 
sending a $1$-dimensional sheaf to its fitting support. Then, we have the following:

\begin{theorem}\label{thm:chi-ind}For any finite field $k$ and any $\overline{\phi}: k \to T$ the counting function $h_{\beta,\chi!}\BPS_{\beta,\chi,\overline{\phi}}$ on $B_{\overline{\phi}}$ is independent of $\chi$, i.e.,
\[h_{\beta,\chi!}\BPS_{\beta,\chi,\overline{\phi}} = h_{\beta,\chi'!}\BPS_{\beta,\chi',\overline{\phi}},  \]
for any two $\chi,\chi' \in \BZ$.    
\end{theorem}
\begin{proof} First we notice that for any $x \in \BM_{\beta,\chi}$ 
\[(x,x) = - \dim\BM_{\beta,\chi} = -\beta^2, \]
by \cite[Lemma 2.5]{MS20}, in particular the pairing is independent of $\chi$.

Now let $\Oo_F$ be the ring of Witt-vectors of the finite field $k$. Then the quotient field $F$ contains no $p$-th roots of $1$ and $\overline{\phi}$ lifts to a $\phi:\Oo_F \to T$. By Lemma \ref{intbps} it is thus sufficient to show 

\[\int_{h_{\beta,\chi,\phi}^{-1}B(x)} e^{2\pi i\tilde{\alpha}} \ d\mu_{orb} = \int_{h_{\beta,\chi',\phi}^{-1}B(x)} e^{2\pi i\tilde{\alpha}} \ d\mu_{orb}.  \]

Now, \cite[Theorem 1.2.2]{COW21} states that \[\int_{h_{\beta,\chi,\phi}^{-1}B(x)} e^{2\pi i g_\beta\alpha} \ d\mu_{orb}\]
is independent of $\chi$, where $g_\beta = \frac{1}{2} \beta(\beta +K_S) +1$ is the arithmetic genus of a curve in the linear system $|\beta|$. 

We claim that $e^{2\pi i g_\beta \alpha(x)_n}$ and $e^{2\pi i\tilde{\alpha}(x)_n}$ differ by the constant factor $(-1)^{\beta^2}$, from which the theorem follows.

To see this, write $\beta = r \cdot \gamma$ for $r \geq 1$ and $\gamma$ a reduced curve class. Then, the order of $\alpha$ restricted to $M_{\beta,\chi}$ divides $r$ by Lemma \ref{gerby}. In particular we have for any $y \in h_{\beta,\chi}^{-1}B(x)$
\[e^{2\pi i g_\beta \alpha(y)_n}= e^{2\pi i \alpha} e^{2\pi i \left(\frac{1}{2} r\gamma(r\gamma +K_S)\alpha(y)_n\right)} = e^{2\pi i \alpha} e^{2\pi i \left(\frac{1}{2} r(r-1)\gamma^2\alpha(y)_n\right)}= (-1)^{\frac{r(r-1) \gamma^2}{o_\alpha(y)_n}} e^{2\pi i \alpha(y)_n}. \]
For the second equality we used $\frac{1}{2} \gamma (\gamma +K_S) \in \BZ$ and $r \alpha(y)_n = 0 \in \BQ/\BZ$.
On the other hand by \eqref{modal}

\[e^{2\pi i\tilde{\alpha}(y)_n} = (-1)^{-\frac{r^2\gamma^2}{o_\alpha(y)_n^2}}e^{2\pi i \alpha(y)_n} = (-1)^{-\frac{r}{o_\alpha(y)_n}\gamma^2}e^{2\pi i \alpha(y)_n}   \]

Therefore, $e^{2\pi i g_\beta \alpha(y)_n}$ and $e^{2\pi i\tilde{\alpha}(y)_n}$ differ by the factor $(-1)^{\frac{r^2 \gamma^2}{o_\alpha(y)}} = (-1)^{\beta^2}$ as claimed.
\end{proof}

Finally, we use our arithmetic methods to recover the $\chi$-independence of sheaf-theoretic BPS-invariants for del Pezzo surfaces proven in \cite{MS20}. It should be noted, however, that the original argument yields an isomorphism compatible with perverse and Hodge filtrations at the same time. To the best of our knowledge, this cannot be obtained from point-counting methods, the principal obstructions being that there is currently no analogue of $p$-adic Hodge theory for intersection cohomology.

\begin{theorem}\cite[Theorem 0.1]{MS20}\label{mscor} For $k= \BC$ and $\chi,\chi' \in \BZ$ there exists an isomorphism of graded vector spaces between the intersection cohomology groups
\[ IH^*(\M_{\beta,\chi}) \cong IH^*(\M_{\beta,\chi'}),  \]
respecting the perverse filtrations. There also exists a possibly different isomorphism $IH^*(\M_{\beta,\chi}) \cong IH^*(\M_{\beta,\chi'})$ respecting the Hodge filtrations. 
\end{theorem}
\begin{proof}

We will infer first the statement from Theorem \ref{thm:chi-ind} and Proposition \ref{bpsic} using the function-sheaf dictionary and Chebotarev Density. Choosing a spreading out over $T$ as above, we see from  Proposition \ref{bpsic} that there exists a $T_{\beta,\chi} \subset T$ such that for every morphism $\phi:\Spec(k) \to T_{\beta,\chi}$ we have

\[\Tr(\Fr,\CIC_{\M_{\beta,\chi,\phi}}) = (-1)^{\beta^2+1} \BL^{\frac{\beta^2+1}{2}} \BPS_{\beta,\chi,\phi},  \]
and similarly for $\chi'$.

Now, purity of the complex of lisse sheaves $\CIC_{\M_{\beta,\chi,\phi}}$ is shown in \cite[Corollaire 5.3.4]{beilinson2018faisceaux}. Applying Deligne's \cite[Proposition 6.2.6]{deligne1980conjecture} to this pure complex we obtain purity of the derived pushforward of intersection cohomology $R(h_{\beta,\chi})_*\CIC_{\M_{\beta,\chi,\phi}}$ and similarly for $\chi'$. Furthermore, we know from the Decomposition Theorem that the derived pushforward is semi-simple.
The functions associated to these derived pushforwards agree by virtue of $\chi$-independence of $h_{\beta,\chi!}\BPS_{\beta,\chi,\phi}$, Theorem \ref{thm:chi-ind}.

From Chebotarev's Density Theorem we infer that the isomorphism class of $R(h_{\beta,\chi})_*\CIC_{\M_{\beta,\chi,\BC}}$ in the derived category is independent of $\chi$. In particular, there exists an isomorphism in the derived category
$$R(h_{\beta,\chi})_*\CIC_{\M_{\beta,\chi,\BC}} \simeq R(h_{\beta,\chi'})_*\CIC_{\M_{\beta,\chi',\BC}}.$$
Applying the perverse truncation functors and derived global sections we obtain an isomorphism
\begin{eqnarray*}\image \big(\BH^*(B,\tau^p_{\leq i}R(h_{\beta,\chi})_*\CIC_{\M_{\beta,\chi,\BC}}) \to \BH^*(B,R(h_{\beta,\chi})_*\CIC_{\M_{\beta,\chi,\BC}})\big) \cong \\ \cong \image \big(\BH^*(B,\tau^p_{\leq i}R(h_{\beta,\chi'})_*\CIC_{\M_{\beta,\chi',\BC}}) \to \BH^*(B,R(h_{\beta,\chi'})_*\CIC_{\M_{\beta,\chi',\BC}})\big).\end{eqnarray*}
In particular, there exists an isomorphism of intersection cohomology $IH^*(\M_{\beta,\chi,\BC}) \cong IH^*(\M_{\beta,\chi',\BC})$
respecting the perverse filtration.

To obtain the statement about the Hodge filtration we argue in a similar way. By purity it is enough to show that the weight-polynomials of the Hodge structures $IH^*(\M_{\beta,\chi,\BC})$ and $ IH^*(\M_{\beta,\chi',\BC})$ agree. By \cite[Theorem 5.6]{Me15} these agree up to a prefactor with the weight (or E-) polynomial of the motivic classes $s_{\chi!}\BPS_{\beta,\chi,\BC}$ and $s_{\chi'!}\BPS_{\beta,\chi',\BC}$, where $s_{\chi!}$ and $s_{\chi'!}$ denote the structure morphisms to $\Spec(\BC)$. Their equality now follows from Katz's theorem \cite[Theorem 6.1.2]{MR2453601} together with the $\chi$-independence of the counting functions, Theorem \ref{thm:chi-ind}. Here we may argue as in the proof of Proposition \ref{bpsic} to see that $\BPS_{\beta,\chi,\BC}$ is in the image of 
\[K_0(\Var/\M_{\beta,\chi,\BC})[\BL^{-1},(\BL^N-1)^{-1} : N \geq 1] \to K_0(\Var/\M_{\beta,\chi,\BC})[\BL^{-1/2},(\BL^N-1)^{-1} : N \geq 1],\]
to avoid discussing an extension of Katz's theorem to fractional powers of $\BL$.
\end{proof}

As in \cite{MS20} one can obtain an analogous result for meromorphic Higgs bundles.

\begin{rmk}\label{KL} The results of of Section \ref{ii} also apply to other moduli problems, for example to moduli of semi-stable vector bundles of fixed slope on a smooth projective curve. In particular if  $\M_{2,0}(C)$ denotes the moduli space of semi-stable rank $2$, degree $0$ vector bundles on a genus $g$ curve $C$ it was already observed in \cite[Remark 6.2]{KL04}, that the stringy $E$-polynomial of $\M_{2,0}$ agrees with its intersection $E$-polynomial if $g$ is even. Now, the $p$-adic analogue of the stringy $E$-polynomial of a variety is its volume and thus this matches our computations as $\dim \BM_{2,0}= 4(g-1)$, from which one can deduce that $e^{2\pi i\tilde{\alpha}} \equiv 1$ on $\M_{2,0}$. If $g$ is odd, $e^{2\pi i\tilde{\alpha}}$ is non-constant and indeed the authors in \textit{loc. cit.} also observe that in this case stringy and intersection $E$-polynomial are different. 
\end{rmk}

\bibliographystyle{amsalpha}
\bibliography{master}

\newcommand{\etalchar}[1]{$^{#1}$}
\providecommand{\bysame}{\leavevmode\hbox to3em{\hrulefill}\thinspace}
\providecommand{\MR}{\relax\ifhmode\unskip\space\fi MR }
\providecommand{\MRhref}[2]{%
  \href{http://www.ams.org/mathscinet-getitem?mr=#1}{#2}
}
\providecommand{\href}[2]{#2}
\begin{thebibliography}{BDN{\etalchar{+}}25}

\bibitem[AGV08]{MR2450211}
Dan Abramovich, Tom Graber, and Angelo Vistoli, \emph{Gromov-{W}itten theory of
  {D}eligne-{M}umford stacks}, Amer. J. Math. \textbf{130} (2008), no.~5,
  1337--1398. \MR{2450211 (2009k:14108)}

\bibitem[AHLH23]{AHH23}
Jarod Alper, Daniel Halpern-Leistner, and Jochen Heinloth, \emph{Existence of
  moduli spaces for algebraic stacks}, Inventiones mathematicae \textbf{234}
  (2023), no.~3, 949--1038.

\bibitem[AZ01]{AZ01}
Michael Artin and James~J Zhang, \emph{Abstract hilbert schemes}, Algebras and
  representation theory \textbf{4} (2001), 305--394.

\bibitem[BBDG18]{beilinson2018faisceaux}
Alexander Beilinson, Joseph Bernstein, Pierre Deligne, and Ofer Gabber,
  \emph{Faisceaux pervers}, Soci{\'e}t{\'e} math{\'e}matique de France, 2018.

\bibitem[BDN{\etalchar{+}}25]{BDNKP25}
Chenjing Bu, Ben Davison, Andr{\'e}s~Ib{\'a}{\~n}ez N{\'u}{\~n}ez, Tasuki
  Kinjo, and Tudor P{\u{a}}durariu, \emph{Cohomology of symmetric stacks},
  arXiv preprint arXiv:2502.04253 (2025).

\bibitem[BT72]{BT1}
F.~Bruhat and J.~Tits, \emph{Groupes r\'{e}ductifs sur un corps local}, Inst.
  Hautes \'{E}tudes Sci. Publ. Math. (1972), no.~41, 5--251. \MR{327923}

\bibitem[Cad07]{1117817b-301b-381d-b9c7-cb40e4814623}
Charles Cadman, \emph{Using stacks to impose tangency conditions on curves},
  American Journal of Mathematics \textbf{129} (2007), no.~2, 405--427.

\bibitem[CGP15]{CGP}
Brian Conrad, Ofer Gabber, and Gopal Prasad, \emph{Pseudo-reductive groups},
  second ed., New Mathematical Monographs, vol.~26, Cambridge University Press,
  Cambridge, 2015. \MR{3362817}

\bibitem[Ch{\^a}11]{MR2856372}
Ng{\^o}~Bao Ch{\^a}u, \emph{Decomposition theorem and abelian fibration}, On
  the stabilization of the trace formula, Stab. Trace Formula Shimura Var.
  Arith. Appl., vol.~1, Int. Press, Somerville, MA, 2011, pp.~253--264.
  \MR{2856372}

\bibitem[CLL16]{CL16}
Antoine Chambert-Loir and Fran{\c{c}}ois Loeser, \emph{Motivic height zeta
  functions}, American Journal of Mathematics (2016), 1--59.

\bibitem[Cor20]{MR4129311}
Christophe Cornut, \emph{Filtrations and buildings}, Mem. Amer. Math. Soc.
  \textbf{266} (2020), no.~1296, 147. \MR{4129311}

\bibitem[COW24]{COW21}
Francesca Carocci, Giulio Orecchia, and Dimitri Wyss, \emph{{BPS} invariants
  from p-adic integrals}, Compositio Mathematica \textbf{160} (2024), no.~7,
  1525--1550.

\bibitem[Del80]{deligne1980conjecture}
Pierre Deligne, \emph{La conjecture de {W}eil : I{I}}, Publications
  Math{\'e}matiques de l'IH{\'E}S \textbf{52} (1980), 137--252.

\bibitem[DL98]{DL98}
Jan Denef and Fran{\c{c}}ois Loeser, \emph{Motivic {I}gusa zeta functions},
  Journal of algebraic geometry \textbf{7} (1998), no.~3, 505--537.

\bibitem[DL02a]{DL02}
Jan Denef and Fran{\c{c}}ois Loeser, \emph{Lefschetz numbers of iterates of the
  monodromy and truncated arcs}, Topology \textbf{41} (2002), no.~5,
  1031--1040.

\bibitem[DL02b]{DL2002}
\bysame, \emph{Motivic integration, quotient singularities and the {M}c{K}ay
  correspondence}, Compositio mathematica \textbf{131} (2002), no.~3, 267--290.

\bibitem[DM20]{DM20}
Ben Davison and Sven Meinhardt, \emph{Cohomological {D}onaldson--{T}homas
  theory of a quiver with potential and quantum enveloping algebras},
  Inventiones mathematicae \textbf{221} (2020), no.~3, 777--871.

\bibitem[GI63]{MR0144889}
O.~Goldman and N.~Iwahori, \emph{The space of p-adic norms}, Acta Math.
  \textbf{109} (1963), 137--177. \MR{144889}

\bibitem[GWZ20a]{GWZ20b}
Michael Groechenig, Dimitri Wyss, and Paul Ziegler, \emph{Geometric
  stabilisation via p-adic integration}, Journal of the American Mathematical
  Society \textbf{33} (2020), no.~3, 807--873.

\bibitem[GWZ20b]{GWZ20a}
\bysame, \emph{Mirror symmetry for moduli spaces of {H}iggs bundles via p-adic
  integration}, Inventiones mathematicae \textbf{221} (2020), no.~2, 505--596.

\bibitem[HL15]{HL15}
Ehud Hrushovski and Fran{\c{c}}ois Loeser, \emph{Monodromy and the {L}efschetz
  fixed point formula}, Ann. Sci. {\'E}c. Norm. Sup{\'e}r.(4) \textbf{48}
  (2015), no.~2, 313--349.

\bibitem[HRV08]{MR2453601}
Tam{{\'a}}s Hausel and Fernando Rodriguez-Villegas, \emph{Mixed {H}odge
  polynomials of character varieties}, Invent. Math. \textbf{174} (2008),
  no.~3, 555--624, With an appendix by Nicholas M. Katz. \MR{2453601
  (2010b:14094)}

\bibitem[HS23]{HS23}
David Hansen and Peter Scholze, \emph{Relative perversity}, Communications of
  the American Mathematical Society \textbf{3} (2023), no.~09, 631--668.

\bibitem[HT03]{HT03}
Tam{{\'a}}s Hausel and Michael Thaddeus, \emph{Mirror symmetry, {L}anglands
  duality, and the {H}itchin system}, Invent. Math. \textbf{153} (2003), no.~1,
  197--229. \MR{1990670 (2004m:14084)}

\bibitem[JS12]{JS12}
Dominic Joyce and Yinan Song, \emph{A theory of generalized donaldson--thomas
  invariants}, vol. 217, American Mathematical Society, 2012.

\bibitem[KL04]{KL04}
Young-Hoon Kiem and Jun Li, \emph{Desingularizations of the moduli space of
  rank 2 bundles over a curve}, Mathematische Annalen \textbf{330} (2004),
  no.~3, 491--518.

\bibitem[KS08]{KS08}
Maxim Kontsevich and Yan Soibelman, \emph{Stability structures, motivic
  {D}onaldson-{T}homas invariants and cluster transformations}, arXiv preprint
  arXiv:0811.2435 (2008).

\bibitem[Lan00]{MR1739403}
Erasmus Landvogt, \emph{Some functorial properties of the {B}ruhat-{T}its
  building}, J. Reine Angew. Math. \textbf{518} (2000), 213--241. \MR{1739403}

\bibitem[Lie08]{Lieblich_2008}
Max Lieblich, \emph{Twisted sheaves and the period-index problem}, Compositio
  Mathematica \textbf{144} (2008), no.~1, 1–31.

\bibitem[Mei15]{Me15}
Sven Meinhardt, \emph{Donaldson-thomas invariants vs. intersection cohomology
  for categories of homological dimension one}, arXiv preprint arXiv:1512.03343
  (2015).

\bibitem[Mei17]{Me17b}
\bysame, \emph{An {I}ntroduction to ({M}otivic) {D}onaldson-{T}homas theory},
  Confluentes Mathematici \textbf{9} (2017), no.~2, 101--158.

\bibitem[Moz19]{Mo19}
Sergey Mozgovoy, \emph{Commuting matrices and volumes of linear stacks}, arXiv
  preprint arXiv:1901.00690 (2019).

\bibitem[MR19]{MR19}
Sven Meinhardt and Markus Reineke, \emph{{D}onaldson--{T}homas invariants
  versus intersection cohomology of quiver moduli}, Journal f{\"u}r die reine
  und angewandte Mathematik (Crelles Journal) \textbf{2019} (2019), no.~754,
  143--178.

\bibitem[MS23]{MS20}
Davesh Maulik and Junliang Shen, \emph{Cohomological $\chi$--independence for
  moduli of one-dimensional sheaves and moduli of {H}iggs bundles}, Geometry \&
  Topology \textbf{27} (2023), no.~4, 1539--1586.

\bibitem[Ser07]{Se07}
Jean-Pierre Serre, \emph{Cohomologie galoisienne}, vol.~5, Springer, 2007.

\bibitem[Ser09]{Se09}
\bysame, \emph{Lie algebras and {L}ie groups: 1964 lectures given at {H}arvard
  {U}niversity}, Springer, 2009.

\bibitem[SGA70]{SGA3II}
\emph{Sch\'emas en groupes. {II}: {G}roupes de type multiplicatif, et structure
  des sch\'emas en groupes g\'en\'eraux}, S\'eminaire de G\'eom\'etrie
  Alg\'ebrique du Bois Marie 1962/64 (SGA 3). Dirig\'e par M. Demazure et A.
  Grothendieck. Lecture Notes in Mathematics, Vol. 152, Springer-Verlag,
  Berlin-New York, 1970. \MR{0274459}

\bibitem[Sta80]{St80}
Richard~P Stanley, \emph{Decompositions of rational convex polytopes}, Ann.
  Discrete Math \textbf{6} (1980), no.~6, 333--342.

\bibitem[SU22]{SU22}
Matthew Satriano and Jeremy Usatine, \emph{Stringy invariants and toric {A}rtin
  stacks}, Forum of Mathematics, Sigma, vol.~10, Cambridge University Press,
  2022.

\bibitem[Sum75]{MR0387294}
Hideyasu Sumihiro, \emph{Equivariant completion. {II}}, J. Math. Kyoto Univ.
  \textbf{15} (1975), no.~3, 573--605. \MR{387294}

\bibitem[WY15]{wood2015mass}
Melanie~Machett Wood and Takehiko Yasuda, \emph{Mass formulas for local galois
  representations and quotient singularities. i: a comparison of counting
  functions}, International Mathematics Research Notices \textbf{2015} (2015),
  no.~23, 12590--12619.

\bibitem[Yas06]{Ya06}
Takehiko Yasuda, \emph{{Motivic integration over Deligne--Mumford stacks}},
  Advances in Mathematics \textbf{207} (2006), no.~2, 707--761.

\bibitem[Zie23]{Ziegler22}
Paul Ziegler, \emph{Normed fiber functors}, arXiv preprint arXiv:2203.17073v3
  (2023).

\end{thebibliography}
\end{document}